\documentstyle[11pt,amssymb,amsmath,amscd,amsbsy,amsfonts,theorem,hyperref]{article}

\input xy
\xyoption{all}

\newcommand{\zz}{{\Bbb Z}}

\newcommand{\cc}{\Bbb C}
\newcommand{\qq}{{\Bbb Q}}
\newcommand{\pp}{{\Bbb P}}
\newcommand{\aaa}{{\Bbb A}}
\newcommand{\ff}{{\Bbb F}}

\newcommand{\iis}{{\mathbf{i}}}
\newcommand{\ddim}{\operatorname{dim}}
\newcommand{\ddeg}{\operatorname{deg}}
\newcommand{\kker}{\operatorname{Ker}}
\newcommand{\coker}{\operatorname{Coker}}
\newcommand{\spec}{\operatorname{Spec}}

\newcommand{\op}[1]{\operatorname{#1}}

\newcommand{\ffi}{\varphi}

\newcommand{\eps}{\varepsilon}

\newcommand{\row}{\rightarrow}
\newcommand{\llow}{\longleftarrow}
\newcommand{\low}{\leftarrow}
\newcommand{\lrow}{\longrightarrow}
\renewcommand{\leq}{\leqslant}
\renewcommand{\geq}{\geqslant}

\newcommand{\nichego}[1]{}

\newcommand{\ov}[1]{\overline{#1}}

\newcommand{\wt}[1]{\widetilde{#1}}

\newcommand{\smk}{{\mathbf{Sm}}_k}
\newcommand{\schk}{{\mathbf{Sch}}_k}
\newcommand{\ri}{{\mathbf{R}}^*}
\newcommand{\nhk}{\mathbf{SmOp}}
\newcommand{\ab}{{\mathbf{Ab}}}

\newcommand{\slnT}{\op{S}_{LN}^{Tot}}
\newcommand{\stT}{\op{St}^{Tot}}
\newcommand{\barbi}[1]{\ov{\op{#1}}}
\newcommand{\fra}{\frak{a}}
\newcommand{\fraO}{{\frak{a}}'}
\newcommand{\frb}{\frak{b}}

\newcommand{\frc}{\frak{c}}
\newcommand{\Da}{a}
\newcommand{\DaO}{a'}
\newcommand{\Db}{b}

\newcommand{\Dc}{c}
\newcommand{\G}{{\mathbf{G}}}
\newcommand{\gl}[1]{G_{#1}}
\newcommand{\fl}[1]{F_{#1}}
\newcommand{\wgl}[1]{\wt{G}_{#1}}

\newcommand{\laz}{{\Bbb L}}
\newcommand{\co}{{\cal O}}

\newcommand{\ci}{{\cal I}}

\newcommand{\cm}{{\cal M}}

\newcommand{\cv}{{\cal V}}
\newcommand{\cw}{{\cal W}}

\newcommand{\ch}{{\cal H}}

\newcommand{\smu}{{\cal S}}
\newcommand{\rc}{{\cal RC}}
\newcommand{\cmor}{Mor}

\newcommand{\komp}[1]{\widehat{#1}}
\newcommand{\sstar}{\hat{\star}}

\newcommand{\Qed}{\hfill$\square$\smallskip}

\newenvironment{proof}{\noindent{\it Proof}:}{\vskip 5mm}

\newtheorem{proposition}{Proposition}[section]{\bf}{\it}
\newtheorem{theorem}[proposition]{Theorem}{\bf}{\it}
\newtheorem{lemma}[proposition]{Lemma}{\bf}{\it}
{\bf}{\it}
\newtheorem{definition}[proposition]{Definition}{\bf}{\it}
{\bf}{\it}
{\bf}{\it}
\newtheorem{example}[proposition]{Example}{\bf}{\it}
\newtheorem{remark}[proposition]{Remark}{\bf}{\rm}

{\bf}{\it}
\newtheorem{corollary}[proposition]{Corollary}{\bf}{\it}
{\bf}{\it}

\begin{document}

\title{Stable and Unstable Operations in Algebraic Cobordism}
\author{Alexander Vishik\footnote{School of Mathematical Sciences, University
of Nottingham}}
\date{}

\maketitle

\begin{abstract}
We describe additive (unstable) operations from a theory $A^*$
obtained from the Levine-Morel algebraic cobordism
by change of coefficients to any oriented cohomology theory $B^*$ (over a field of characteristic zero).
We prove that there is 1-to-1 correspondence between operations $A^n\row B^m$ and
families of homomorphisms
$A^n((\pp^{\infty})^{\times r})\row B^m((\pp^{\infty})^{\times r})$
satisfying certain simple properties.
This provides an effective tool of constructing such operations.
As an application, we prove that (unstable) additive operations in algebraic cobordism
are in 1-to-1 correspondence with the $\laz\otimes_{\zz}\qq$-linear combinations of
Landweber-Novikov operations which take integral values on the products of projective
spaces. Furthermore, the stable operations are precisely the $\laz$-linear
combinations of the Landweber-Novikov operations. We also show that multiplicative operations
$A^*\row B^*$ are in 1-to-1 correspondence with the morphisms of the respective
formal group laws. We construct integral Adams operations in algebraic cobordism,
and all theories obtained from it by change of coefficients, extending the classical Adams
operations in algebraic K-theory. We also construct symmetric operations
and Steenrod operations (\`a la T. tom Dieck) in algebraic cobordism for all primes.
(Only symmetric operations for the prime $2$ were previously known to exist.)
Finally, we prove the Riemann-Roch Theorem for additive operations which extends
the multiplicative case done in \cite{P2}.

\end{abstract}

\tableofcontents

\section{Introduction}
\label{Intro}
In the current article we study operations between oriented cohomology theories (over a field of characteristic zero).
In the algebro-geometric context operations were studied by Voevodsky \cite{VoOP}, Brosnan \cite{Br}, Panin-Smirnov \cite{PS},\cite{P},\cite{P2},\cite{Sm1},\cite{Sm2}, and Levine-Morel \cite{LM}. By the work of Levine-Morel \cite{LM},
one has a universal oriented cohomology theory, called algebraic cobordism, and denoted by $\Omega^*$. The universality of $\Omega^*$ combined with the reorientation procedure of Panin-Smirnov (following Quillen \cite{Qu71}, see also \cite[pages 99-105]{LM})
permitted to produce the multiplicative operations $\Omega^*\row B^*$ easily and to classify them (in the "invertible" case). In particular, one gets that all such operations are specializations of the Total Landweber-Novikov operation $\Omega^*\row\Omega^*[b_1,b_2,\ldots]$.  Previously, the only example of unstable operations (in the algebro-geometric context),
the, so-called, Symmetric operations (mod 2) were introduced in \cite{so1} and \cite{so2}.
Originally constructed with the aim of producing maps between Chow groups of different
quadratic Grassmannians (of the same quadratic form), these operations in algebraic
cobordism were successfully applied to the question of rationality of algebraic
cycles (\cite{GPQCG},\cite{RIC}), where they provide the only known method to deal with $2$-torsion.
These operations can be combined into a Total one which is a "formal half" of the "negative part"
of the Total Steenrod operation (mod 2) in Algebraic Cobordism - see \ref{SymSt}. The topological counterpart of it was
used by Quillen in \cite{Qu71}.
Symmetric operations (mod 2) are
more subtle than the Landweber-Novikov ones. They lack the 2-primary divisibilities of the latter, and so,
in some
sense, "plug the gap" between $\laz$ and $H_*(MU)$ left by the Hurewicz map, plug 2-adically.
To have an integral variant of such statements one would need Symmetric operations for
all primes. Unfortunately, the case $p=2$ was produced by an explicit geometric construction
(using $Hilb_2$),
and it is unclear how to extend it for other primes. The desire to construct these operations
was the main motivation behind the current article. In the end, it appeared that to produce
Symmetric operation for $p>2$ is about as "simple" as to produce all (unstable) additive operations
in algebraic cobordism. But to do it, one has to develop some new tools.
One needs to understand the internal structure of algebraic cobordism and, more precisely, the way $\Omega^*(X)$ can be described
in terms of the restriction of $\Omega^*$ to varieties of dimension lower than the dimension of $X$.
This leads to the notion of a {\it theory of rational type}.
Such theories appear to be the same as the {\it free theories} of Levine-Morel.
In particular, all the "standard" theories, like, $CH$, $K_0$, $BP$, higher Morava's K-theories
$K(n)$ are of this sort.
At this stage I should recall that there are two types of cohomology
theories in Algebraic Geometry: "large" ones $A^{j,i}$ - represented by some spectrum in $\aaa^1$-homotopy theory, numbered by two indices, and "small" ones $A^i$, typically,
represented by the $(2*,*)$-part
of large theories. The Levine-Morel algebraic cobordism $\Omega^*$ belongs
to the second type
and, by the result of Levine (\cite{Lcomp}, see also \cite{Ho}), is the $(2*,*)$-part of Voevodsky's $MGL$. In this article, we work with "small" theories.
The fact that $\Omega^*$ is a {\it theory of rational type} is non-trivial. Our proof uses
the mentioned comparison result of Levine (\cite{Lcomp}).
Any theory $A^*$ of rational type on a variety $X$ is described by the values of $A^*$ on varieties of lower dimension.
We provide three alternative descriptions here: two in terms of push-forwards,
and one in terms of pull-backs - see Subsections \ref{a},\ref{b},\ref{c}.
After that it becomes possible to construct operations inductively on dimension.

This enables us to show that an operation can be reconstructed from it's action on
$(\pp^{\infty})^{\times r}$, for all $r$. This is our main result (see Theorem \ref{MAIN}):

\begin{theorem}
\label{I-MAIN}
Let $A^*$ be a theory, obtained from $\Omega^*$ by change of coefficients, and
$B^*$ be any theory in the sense of Definition \ref{goct}.
Fix $n,m\in\zz$. Then there is a one-to-one correspondence between additive operations
$A^n\stackrel{G}{\row}B^m$ and families of homomorphisms
$$
A^n((\pp^{\infty})^{\times l})\stackrel{G}{\row}B^m((\pp^{\infty})^{\times l}),\,\,\text{for}\,\,
l\in\zz_{\geq 0}
$$
commuting with pull-backs for:
\begin{itemize}
\item[$(i)$ ] the action of ${\frak{S}}_l$;
\item[$(ii)$ ] the partial diagonals;
\item[$(iii)$ ] the partial Segre embeddings;
\item[$(iv)$ ] $(\op{Spec}(k)\hookrightarrow\pp^{\infty})\times
(\pp^{\infty})^{\times r},\,\,\forall r$;
\item[$(v)$ ] the partial projections.
\end{itemize}
\end{theorem}

In Topology an analogous result was obtained by T.Kashiwabara
in \cite[Theorem 4.2]{Kash}.
The "multiplicative" variant of our result
(Proposition \ref{mainMULT})
says that multiplicative operations correspond to families of homomorphisms as above commuting also with
the external products of projective spaces. These results permit to describe and construct operations
effectively, as one only needs to define them on $(\pp^{\infty})^{\times r}$, which is a cellular
space. As a first application, we describe all additive (unstable) operations in the Levine-Morel algebraic cobordism.
These appears to be exactly those
$\laz\otimes_{\zz}\qq$-linear combinations (infinite, in general) of the Landweber-Novikov operations
which take "integral" values on $\Omega^*((\pp^{\infty})^{\times r})$, for all $r$.
This is done in Theorem \ref{UOACpoln}:

\begin{theorem}
\label{I-UOACpoln}
Let $\psi\in\op{Hom}_{\laz}(\laz[\barbi{b}],\laz\otimes_{\zz}\qq)_{(m-n)}$ be a homomorphism of $\laz$-modules.
Denote by $S_{\psi}:\Omega^n\row\Omega^m\otimes_{\zz}\qq$ the respective $\laz\otimes_{\zz}\qq$-linear combination
of the Landweber-Novikov operations, i.e., the composition of
$$
\Omega^*\stackrel{\slnT}{\lrow}\Omega^*[\barbi{b}]\cong\Omega^*\otimes_{\laz}\laz[\barbi{b}]\stackrel{\otimes\psi}{\lrow}
\Omega^{*-n+m}\otimes_{\zz}\qq
$$
in degree $n$. Assume that $S_{\psi}$ satisfies the following integrality condition:
$S_{\psi}(\Omega^n((\pp^{\infty})^{\times r}))\subset \Omega^m((\pp^{\infty})^{\times r})$,
for all $r\geq 0$. Then there exists a unique additive operation $G_{\psi}:\Omega^n\row\Omega^m$ such that
$S_{\psi}=G_{\psi}\otimes\qq$. Moreover, every additive operation arises in this way, for a unique $\psi$.
Thus, $\psi\leftrightarrow G_{\psi}$ is
a 1-to-1 correspondence between linear combinations of Landweber-Novikov operations satisfying integrality conditions
and integral additive operations.
\end{theorem}

With the above notation, the stable operations are precisely the $G_{\psi}$ for
$\psi\in\op{Hom}_{\laz}(\laz[\barbi{b}],\laz)$, i.e., they are the $\laz$-linear combinations of the Landweber-Novikov operations.
(See Theorem \ref{stabOsln} whose proof is much simpler than the above theorem.)

Next, we get a complete description of multiplicative operations from a free theory (in the sense of Levine-Morel)
to any other theory in terms of formal group laws. It is given by Theorem \ref{multFGL}:

\begin{theorem}
\label{I-multFGL}
Let $A^*$ be a free theory, and $B^*$ be any oriented cohomology theory.
The map sending the multiplicative operation $A^*\row B^*$ to the induced homomorphism of
formal group laws $(A^*(k),F_A)\row (B^*(k),F_B)$ is a bijection.
\end{theorem}

Using this, we are able to extend a result of Panin-Smirnov and Levine-Morel on multiplicative operations
$\Omega^*\row B^*$ (see Theorem \ref{PSLM}). This is done in Theorem \ref{neobrB0}:

\begin{theorem}
\label{neobrB0}
Let $B^*$ be an oriented cohomology theory. Let $\gamma=b_0x+b_1x^2+b_2x^3+\ldots\in B^*(k)[[x]]$ be a power series
such that $b_0\in B^*(k)$ is a non zero-divisor. Then there exists a multiplicative operation $G:\Omega^*\row B^*$
with $\gamma_G=\gamma$ if and only if the twisted formal group law $F_B^{\gamma}\in B^*(k)[b_0^{-1}][[x,y]]$
has coefficients in $B^*(k)$.
In this case, such an operation is unique.
\end{theorem}

As an immediate application of this we construct integral Adams operations $\Psi_k$ in algebraic cobordism and all other free theories.
This is Theorem \ref{Adams}:

\begin{theorem}
\label{I-Adams}
For any free theory $A^*$, there are multiplicative $A^*(k)$-linear operations
$\Psi_k:A^*\row A^*$, $k\in\zz$, such that $\gamma_{{\Psi_k}}=[k]\cdot_{A}x$.
These operations do not depend on the choice of orientation of $A^*$.
In the case of $K_0$ these are the usual Adams operations.
\end{theorem}

As these unstable multiplicative operations are $A^*(k)$-linear, they are all obtained from the ones in algebraic cobordism by change of coefficients. Previously, in the case of algebraic cobordism, such operations were known only with rational coefficients
(in which case they can be expressed through Landweber-Novikov operations).

Similar considerations permit to construct the Steenrod operations in algebraic cobordism \`{a} la tom Dieck (see Theorem \ref{TTD}). Finally, using the Main Theorem \ref{MAIN} itself we
construct Symmetric operations for all primes $p$ - see Theorem \ref{SOp}. The last two results
form a separate paper \cite{SPso}, not to overburden the present text.
Aside from the mentioned major results we present various smaller ones - see Section
\ref{applications}. In particular, we show that
all operations in Chow groups mod $p$ are essentially stable (each extends to a unique stable operation),
and consist of Steenrod operations only (Theorem \ref{unstCH}), and we describe additive operations in $K_0$ (see Theorem \ref{unstK}).

Also, as a byproduct of the proof of our main theorem we obtain the Riemann-Roch Theorem for unstable
additive operations - see Theorem \ref{adRR}. It generalizes the multiplicative version obtained earlier - see \cite{P2}.

In Section \ref{basictools} we introduce some tools used in the main part of the article. In particular, various blowup results.
We also discuss combinatorial pull-backs. This version of the refined pull-backs (see \cite[Subsection 6.6]{LM}) for divisors with strict
normal crossings is given by an explicit formula. It is one of our main tools.

{\bf Acknowledgements:}
First of all, I would like to thank A.Smirnov and I.Panin for many stimulating discussions since our 2004-2005 common stay at IAS. These really influenced my way of thinking about the subject.
I want to thank O.Haution, with whom we tried to produce the geometric construction of Symmetric operations for $p=3$, and discussed various other related topics.
Also, I'm very grateful to P.Brosnan,
S.Gille, A.Kuznetsov, A.Lazarev, M.Levine, F.Morel, M.Rost, B.Totaro, V.Voevodsky, N.Yagita, S.Yagunov, and other people for many useful conversations.
Special thanks to A.Lazarev and F.Morel for pointing me in the direction of the Quillen's paper \cite{Qu71}. Finally, I would like to express my gratitude to W.Wilson, whose book \cite{Wi} gave me
the inspiration for the current article, and to T.Kashiwabara, who explained me the topological side of the picture (in particular, that many of my results were known in Topology) and drew my attention to such works as \cite{Kash,BJW,BT}.
And I'm really indebted to the Referees for numerous useful suggestions which substantially improved the exposition and simplified the
arguments, in places.
The support
of EPSRC Responsive Mode grant EP/G032556/1 is gratefully acknowledged.

\section{Algebraic Cobordism and other oriented cohomology theories}
\label{ACiGOC}

\subsection{Main definitions}
\label{MainDef}

Throughout the article $k$ will denote the base field of characteristic $0$.
$\smk$ will denote the category of smooth quasi-projective varieties over $k$,
and $\schk$ - the category of separated schemes of finite type over $k$.
Let $\ri$ be the category of graded commutative rings.

Following D.Quillen (\cite{Qu71}), I.Panin-A.Smirnov (\cite[Definition 3.1.1]{PS}), and M.Levine-F.Morel (\cite[Definition 1.1.2]{LM})
we introduce the notion of an oriented cohomology theory on $\smk$.
The only difference in comparison with \cite[Definition 1.1.2]{LM} is that
we impose the localization axiom $(EXCI)$. All the "standard" theories, like $\Omega^*$, $\op{CH}^*$
and $K_0$ do satisfy this axiom, but not their {\it completed versions} $MGL^{*,*'}$, $\op{H}_{\cm}^{*,*'}$
and $K_*$. Thus, the new axiom $(EXCI)$ is rather restrictive.
And the techniques that we develop in this article rely crucially on it.

\begin{definition}
\label{goct}
{\rm (cf. \cite[Definition 1.1.2]{LM})}
An oriented cohomology theory on $\smk$ is given by:
\begin{itemize}
\item[$(D1)$ ]  An additive (pull-back) functor $A^*: \smk^{op}\row\ri$.
\item[$(D2)$ ]  A push forward structure: for each projective morphism
$f:Y\row X$ of virtual relative codimension $d$, a homomorphism of graded $A^*(X)$-modules:
$$
f_*:A^*(Y)\row A^{*+d}(X).
$$
These data satisfy:
\item[$(A1)$] Functoriality of push-forwards: $(Id_X)_*=Id_{A^*(X)}$, and
%\item[\phantom{$(A1)$}] 
for projective morphisms $f:Y\row X$, $g:Z\row Y$
of virtual relative codimensions $d$ and $e$,
$$
(f\circ g)_*=f_*\circ g_*: A^*(Z)\row A^{*+d+e}(X).
$$
\item[$(A2)$] For pair of transversal morphisms (see \cite[5.3]{so2}) $f:X\row Z$, $g:Y\row Z$ fitting
into a cartesian square
$$
\xymatrix @-0.7pc{
W \ar @{->}[r]^(0.5){g'} \ar @{->}[d]_(0.5){f'}&
X \ar @{->}[d]^(0.5){f}\\
Y \ar @{->}[r]_(0.5){g} & Z,
}
$$
with $f$ projective of relative dimension $d$,
$$
g^*f_*=f'_*{g'}^*.
$$
\item[$(PB)$] For a rank $n$ vector bundle $E\row X$ with canonical quotient line bundle
$O(1)\row\pp(E)$, zero section $s:\pp(E)\row O(1)$, and $\xi\in A^1(\pp(E))$ defined
by
$$
\xi:=s^*s_*(1),
$$
one has: $A^*(\pp(E))$ is a free $A^*(X)$-module with basis
$$
(1,\xi,\xi^2,\ldots,\xi^{n-1}).
$$
\item[$(EH)$] For a vector bundle $E\row X$ and an $E$-torsor $p:V\row X$,
$p^*:A^*(X)\row A^*(V)$ is an isomorphism.
\item[$(EXCI)$] For a smooth quasi-projective variety $X$ with closed subscheme
$Z\stackrel{i}{\row} X$
and open complement $U\stackrel{j}{\row}X$, one has an exact sequence:
$$
A_*(Z)\stackrel{i_*}{\lrow}A_*(X)\stackrel{j^*}{\lrow}A_*(U)\row 0,
$$
where, for a smooth quasi-projective equidimensional variety $Y$, $A_*(Y)=A^{\ddim(Y)-*}(Y)$ and,
for a quasi-projective variety $Y$ which is not assumed to be smooth, $A_*(Y)=\op{colim}_{V\row Y}A_*(V)$
where $V\row Y$ are projective morphisms from a smooth quasi-projective variety $V$ and where the transition maps in
the colimit are push-forward maps.
\end{itemize}
\end{definition}

\begin{remark}
Notice, that $(D2)$ contains the projection formula.

\end{remark}

Whenever we refer to an oriented cohomology theory,
we will mean a theory satisfying the above set of axioms.

Quite often (especially, in our main results) we will need to impose an additional condition
demanding our theory to be constant along field extensions.
To formulate this condition, we set, for a finitely generated extension $L/k$,
$$
A^*(L)=\op{colim}_{U\subset X}A^*(U)
$$
where $X$ is a connected smooth quasi-projective variety such that $k(X)=L$ and $U\subset X$ runs over all
(non-empty) open subschemes of $X$. (See \cite[Subsection 4.4.1]{LM}.)
Then we have the notion of a {\it generically constant} theory
of Levine-Morel - see \cite[Definition 4.4.1]{LM}.

\begin{itemize}
\item[$(CONST)$] {\it The theory is called "generically constant" if the natural map $A^*(k)\row A^*(L)$ is
an isomorphism, for each finitely generated field extension $L/k$.
}
\end{itemize}

All standard theories are generically constant but it is easy to construct theories which are not.

\begin{example}
\label{nonconst}
Let $A^*$ be any theory (say, a generically constant one), and $Y$ be a smooth quasi-projective variety over $k$.
Then we can define a new theory: $A^*_{Y/k}(X):=A^*(Y\times_{\spec(k)}X)$. For example, we can take
$Y=\spec(L)$, where $L/k$ is a finite field extension. This theory will not be generically constant.
For example, if $L/k$ is Galois of degree $n$, then
$A^*_{L/k}(\spec(L))=\oplus_{i=1}^nA^*(\spec(L))$, while $A^*_{L/k}(\spec(k))=A^*(\spec(L))$.
\end{example}

M.Levine and F.Morel constructed the universal oriented cohomology theory
$\Omega^*$ called algebraic cobordism (see \cite[Theorem 1.2.6]{LM}). It has a unique
map to any other theory $A^*$. This theory satisfies $(CONST)$.
It is an algebraic analogue of complex cobordism in topology. Fixing a complex embedding
$k\hookrightarrow\cc$, there is a topological realization morphism $\Omega^*(X)\row MU^{2*}(X(\cc))$ which is an isomorphism
for $X=\spec(k)$.

\subsection{An associated Borel-Moore theory}

Each oriented cohomology theory on $\smk$ can be extended to a Borel-Moore
functor on $\schk$ in the sense of \cite[Definition 2.1.2]{LM} - see \cite[Remark 2.1.4]{LM}.
We will not need most of the features of such a functor, only the push-forward maps
which are completely straightforward, so will not list it's axioms here. Later, in Subsection \ref{c},
in the case of {\it theories of rational type} we will need the {\it refined pull-backs}, but those will be
deduced from the refined pull-backs in algebraic cobordism constructed by Levine-Morel \cite[Theorem 6.6.6]{LM}.

\begin{definition}
\label{non-smDef}
For a quasi-projective scheme $Z$, define $A_*(Z)=\op{colim}_{V\row Z}A_*(V)$
where $V\row Z$ are projective morphisms from a smooth quasi-projective variety $V$ and where the transition maps in
the colimit are push-forward maps.
\end{definition}

Clearly, $A_*(Z)=A_*(Z_{red})$, and if $Z=\cup_{i=1}^mZ_i$ is the decomposition into irreducible
components, then we have an exact sequence:
\begin{equation}
\label{comp-seq}
0\llow A_*(Z)\llow\bigoplus_{i=1}^mA_*(Z_i)\llow\bigoplus_{i,j=1}^mA_*(Z_i\cap Z_j).
\end{equation}
More generally, for a closed embedding $S\subset Z$ with the open compliment $U$ we have
an excision sequence:
\begin{equation}
\label{excsing-seq}
0\llow A_*(U)\llow A_*(Z)\llow A_*(S).
\end{equation}
Here (\ref{comp-seq}) follows immediately from (\ref{excsing-seq}), while the latter one can be easily reduced to the
case of a projective $Z$ which, in turn, is a simple consequence of the resolution of singularities combined
with the usual (smooth) $(EXCI)$ axiom. I leave the details of this exercise to the reader.

A priori, $A_*(Z)$ for a singular scheme $Z$ is expressed in terms of $A_*$ of infinitely many smooth
schemes. But Proposition \ref{AbmZ} shows that one has a finite presentation related
to the resolution of singularities.

\subsection{Formal group law}
\label{FGL}

Any theory in the sense of Definition \ref{goct} (even without $(EXCI)$)
has Chern classes. Namely, if $E$ is a vector bundle of dimension $d$ on $X$,
then $\xi\in A^1(\pp_X(E^{\vee}))$ (as in the axiom $(PB)$) satisfies the unique equation:
$$
\sum_{i=0}^d(-1)^ic^A_i(E)\cdot\xi^{d-i}=0,
$$
where $c^A_0(E)=1$, and $c^A_i(E)\in A^i(X)$ are some elements.
These satisfy the usual Cartan formula, and in the case of a line bundle $L$,
$c^A_1(L)=s^*s_*(1)$, where $s:X\row L$ is the zero section.

Consider the variety $\pi:Flag_X(E)\row X$ of complete flags of $E$. By construction, $\pi^*(E)$ has a natural filtration with graded
pieces of rank $1$. Then the Cartan formula implies that $\sum_{i=0}^d\pi^*c^A_i(E)=\prod_{i=1}^d(t+\lambda_i)$,
for some elements $\lambda_i\in A^1(Flag_X(E)), i=1,\ldots,d$. These are called {\it $A$-roots} of $E$.
By the $(PB)$ axiom, the map $\pi^*:A^*(X)\row A^*(Flag_X(E))$ is split injective, which permits one to make computations with $\pi^*(\alpha)$ instead of $\alpha$ and so, to use the $A$-roots.

By \cite[Theorem 2.3.13]{LM}, any theory $A^*$ as above satisfies the axiom:
\begin{itemize}
\item[$(DIM)$ ] For any line bundles $L_1,\ldots,L_n$ on a smooth $X$ of dimension $<n$,
one has:  $c^A_1(L_1)\cdot\ldots\cdot c^A_1(L_n)=0\in A_*(X)$.
\end{itemize}
Thus, any power series with coefficients in $A^*(X)$ can be evaluated on Chern classes.

To any theory $A^*$ as above one can associate
the formal group law (FGL, for short) $(A^*(k),F_A)$, and
$$
F_A(x,y)=Segre^*(t)\in A^*(k)[[x,y]]=A^*(\pp^{\infty}\times\pp^{\infty}),
$$
where
$\pp^{\infty}\times\pp^{\infty}\stackrel{Segre}{\lrow}\pp^{\infty}$ is the Segre embedding,
and $x,y,t$ are the $1$-st Chern classes of $O(1)$ of the respective copies of $\pp^{\infty}$.
We will denote the coefficients of $F_A$ by $a^A_{i,j}$. Thus,
$$
F_A(x,y)=\sum_{i,j}a^A_{i,j}\cdot x^i\cdot y^j
$$
is a homogeneous power series of degree $1$ (i.e., $a^A_{i,j}\in A^{1-i-j}(k)$), where $a^A_{0,0}=0$ and $a^A_{1,0}=a^A_{0,1}=1$.
The formal group law describes how to compute the $1$-st Chern class of a tensor product of two
line bundles in terms of the $1$-st Chern classes of the factors:
$$
c^A_1(L\otimes M)=F_A(c^A_1(L),c^A_1(M)).
$$
For general facts about formal group laws see \cite{Laz}.
The universal formal group law $(\laz,F_U)$ has a unique morphism to any formal group law,
in particular, to $(A^*(k),F_A)$. M.Levine and F.Morel have shown that, in the case of algebraic cobordism,
the respective map is an isomorphism - see \cite[Theorem 1.2.7]{LM}. In particular,
$\Omega^*(k)=\laz^*$, for any field $k$.

A theory $A^*$ is called {\it additive}, if its formal group law is additive, i.e., $F_A(x,y)=x+y$. By a result of Levine-Morel
(see \cite[Theorem 1.2.2]{LM}), $\op{CH}^*$ is the universal additive theory.

\section{Operations}
\label{op}

\subsection{The category $\nhk$}
\label{NHk}

As in topology, an operation from a theory $A^*$ to a theory $B^*$ is a natural transformation from $A^*$ to $B^*$ considered as
a contravariant functors from $\smk$, i.e., operations commute with pull-backs (but not necessarily with push-forwards).
The most commonly studied operations are the stable ones with the exception of Adams operations in K-theory. (See also \cite{so2}
where Symmetric operations (mod $2$) in algebraic cobordism are introduced.)
The aim of the current article is to develop an effective method of producing unstable operations. And, although, in the end, stable operations is not what we are after (there are more or less no questions left about them), they provide an important "coordinate system" in which one can describe unstable ones.
To be able to talk about "stability" we need to introduce some notion of {\it suspension}.
Following V.Voevodsky and I.Panin-A.Smirnov, we introduce:
\begin{definition}
\label{KatNHk}
The category $\nhk$ has objects $(X,U)$, where $X$ is a smooth quasi-projective variety over $k$, and
$U\stackrel{i}{\hookrightarrow} X$ is an open subvariety. Morphisms from $(X,U)$ to $(Y,V)$ are maps $X\stackrel{f}{\row}Y$
which map $U$ to $V$. We have a natural functor:
$$
N:\smk\lrow\nhk,
$$
sending $X$ to $(X,\emptyset)$.
\end{definition}

In $\nhk$ we have cartesian product given by:
$$
(X,U)\times (Y,V):=(X\times Y,U\times V),
$$
and we can define smash product by the formula:
$$
(X,U)\wedge(Y,V):=(X\times Y, X\times V\cup U\times Y),
$$
which permits to introduce the suspension:
\begin{definition}
\label{susp}
$$
\Sigma_T(X,U):=(X,U)\wedge(\pp^1,\pp^1\backslash 0).
$$
\end{definition}

Any theory $A^*$ in the sense of Definition \ref{goct} can be extended to
a contravariant functor $A^*:\nhk\row\ab$ as follows:
$$
A^*((X,U)):=\kker(A^*(X)\stackrel{i^*}{\row}A^*(U)),
$$
with the pull-backs naturally induced by those from $\smk$.
We have an external product:
$$
A^*((X,U))\otimes A^*((Y,V))\stackrel{\wedge}{\lrow} A^*((X,U)\wedge(Y,V)),
$$
and a canonical element $\eps^A=c^A_1(O(1))\in A^1((\pp^1,\pp^1\backslash 0))$ -
the class of a rational point. We get the natural isomorphism:
\begin{equation*}
\begin{split}
\sigma_T: A^n((X,U))&\stackrel{=}{\lrow} A^{n+1}(\Sigma_T(X,U))\\
x&\mapsto x\wedge\eps^A.
\end{split}
\end{equation*}

\begin{definition}
\label{operation}
Let $A^*$ and $B^*$ be theories in the sense of Definition \ref{goct}. An operation
$G:A^n\row B^m$ is a natural transformation between $A^n$ and $B^m$ considered as contravariant
functors from $\smk$ to the category of pointed sets. In other words, it is a family of maps
$G_X:A^n(X)\row B^m(X)$, for $X\in\smk$, commuting with pull-backs and sending zero to zero.
An operation is called additive, if the maps $G_X$ are homomorphisms of abelian groups.
\end{definition}

Note, that such an operation extends uniquely to a morphism of contravariant functors on $\nhk$.
Moreover, the condition $0\mapsto 0$ is equivalent to the existence of such an extension (since
$A^*((X,X))=0$, and there exists a morphism $(X,U)\row (X,X)$).

\begin{definition}
\label{stableoperation}
A stable operation $G:A^*\row B^{*+l}$ is a set of operations $\{G^n:A^n\row B^{n+l},\,n\in\zz\}$, which
commute with the isomorphisms $\sigma_T$.
\end{definition}

As one would expect,
\begin{proposition}
\label{stable-additive}
Any stable operation is additive.
\end{proposition}

\begin{proof}
Let $\alpha,\beta,\gamma:(\pp^1,\pp^1\backslash 0)\row
(\pp^1,(\pp^1\backslash 0))^{\times 2}$ be defined as follows:
$\alpha=id\times\infty$, $\beta=\infty\times id$, $\gamma=\Delta$.
The $A^*(k)$-module $A^*((\pp^1,\pp^1\backslash 0)^{\times 2})$ is freely generated by
$\eps_1^A=c^A_1(O(1,0))$, $\eps_2^A=c^A_1(O(0,1))$ and $\eps_1^A\cdot\eps_2^A$ where
$O(1,0)$ and $O(0,1)$ are the obvious two line bundles on $\pp^1\times\pp^1$.
Consequently, for $R\in \nhk$,
$$
A^*((\pp^1,\pp^1\backslash 0)^{\times 2}\wedge R)=\eps_1^A\wedge A^*(R)\oplus\eps_2^A\wedge A^*(R)\oplus
(\eps_1^A\cdot\eps_2^A)\wedge A^*(R).
$$
Using these coordinates, it is easy to see that $\gamma_R^*=\alpha_R^*+\beta_R^*$.
Let $x$ and $y$ be elements of $A^*(R)$. Then
\begin{equation*}
\begin{split}
G(\eps^A\wedge(x+y))=G(\eps^A\wedge x+\eps^A\wedge y)=G(\gamma_R^*(\eps_1^A\wedge x+\eps_2^A\wedge y))=
\gamma_R^*G(\eps_1^A\wedge x+\eps_2^A\wedge y)=\\
\alpha_R^*G(\eps_1^A\wedge x+\eps_2^A\wedge y)+\beta_R^*G(\eps_1^A\wedge x+\eps_2^A\wedge y)=
G(\eps^A\wedge x)+G(\eps^A\wedge y).
\end{split}
\end{equation*}
Since $G$ is stable, we obtain that $G(x+y)=G(x)+G(y)$.
\Qed
\end{proof}

\begin{definition}
\label{multiplicativeoperation}
A multiplicative operation $G:A^*\row B^*$ is a natural transformation between $A^*$ and $B^*$ considered as contravariant functors
from $\smk$ to the category of rings. In other words, $G_X:A^*(X)\row B^*(X)$ is a ring homomorphism for all $X\in\smk$.
\end{definition}

Let us stress that we do not require multiplicative operation to respect grading.
To a multiplicative operation $G:A^*\row B^*$, one can associate a power series
$
\gamma_G=b_0x+b_1x^2+\ldots\in B^*(k)[[x]],
$
called the {\it inverse Todd genus}.
It is uniquely determined by the following condition.
If
$x^A=c^A_1(O_{\pp^{\infty}}(1))\in A^*(\pp^{\infty})$, and similarly for $x^B$, then
$G(x^A)=\gamma_G(x^B)\in B^*(k)[[x^B]]=B^*(\pp^{\infty})$. Moreover, if $\ffi_G=G_k:A^*(k)\row B^*(k)$ is the ring homomorphism induced
by $G$ on the coefficients, then the pair $(\ffi_G,\gamma_G): (A^*(k),F_A)\row (B^*(k),F_B)$ is a morphism of formal group laws.
In other words,
$$
\ffi_G(F_A)(\gamma_G(u),\gamma_G(v))=\gamma_G(F_B(u,v)).
$$
Of course, the composition of multiplicative operations corresponds to the
composition of morphisms of formal group laws:
$$
(\ffi_{H\circ G}\,,\,\gamma_{H\circ G}(x))=(\ffi_H\circ\ffi_G\,,\,\ffi_H(\gamma_G)(\gamma_H(x))).
$$
In the case of $A^*=\Omega^*$, and $b_0$ invertible in $B^*(k)$, the homomorphism
$\ffi_G$ is completely determined by $\gamma_G$.
Namely, $\laz$ is generated as a ring by the coefficients $a^{\Omega}_{i,j}$ of the universal formal group law,
and $\ffi_G(a^{\Omega}_{i,j})$ is the respective coefficient of the formal group law
$F^{\gamma_G}_B(x,y)=\gamma_G(F_B(\gamma_G^{-1}(x),\gamma_G^{-1}(y)))$.
Moreover, we have the following result:
\begin{theorem} {\rm (Panin-Smirnov+Levine-Morel)}
\label{PSLM}
Let $\gamma=b_0x+b_1x^2+b_2x^3+\ldots\in B^*(k)[[x]]$. Assume that $b_0$ is invertible
in $B^*(k)$. Then there exists a unique multiplicative operation $G:\Omega^*\row B^*$ with $\gamma_G=\gamma$.
\end{theorem}

\begin{proof}
Having a power series $\gamma$ as above with invertible $b_0$, one can construct the new theory $\widetilde{B}$ by changing the
orientation (push-forward structure) of the theory $B^*$ using the recipe
from \cite[Theorem 5.1.4]{PS} (take $\ffi=id$, loc. cit.), so that $c^{\widetilde{B}}_1=\gamma(c^B_1)$.
The details can be found in \cite[Theorem 2.3.2]{P2}.
The same method (which is essentially due to Quillen \cite{Qu71}) is employed in \cite[99-102]{LM} and \cite{Lst} (where the definition of an oriented theory is closer to our's).
Since the pull-back structure on $B^*$ and $\wt{B}^*$ is the same, the operations $G:\Omega^*\row B^*$ with $\gamma_G=\gamma$
correspond exactly to morphisms of theories $\Omega^*\row\widetilde{B}^*$ (mapping $c_1^{\Omega}$ to $c_1^{\widetilde{B}}$).
So, the existence and uniqueness of such an operation follows from the universality of algebraic cobordism of Levine-Morel
\cite[Theorem 1.2.6]{LM}.
\Qed
\end{proof}

Below we will be able to generalize this result substantially - see Theorems \ref{neobrB0}
and \ref{OBmult}.

The following statement describes the relation between stable and multiplicative operations.

\begin{proposition}
\label{multstable}
Let $G:A^*\row B^*$ be a multiplicative operation with $\gamma_G=b_0x+b_1x^2+\ldots$.
Then $G$ is stable if and only if $b_0=1$.
\end{proposition}

\begin{proof}
Let $R\in\nhk$. Since $G$ is multiplicative, we have for $x\in A^*(R)$:
$$
G(x\wedge\eps^A)=G(x)\wedge G(\eps^A)=G(x)\wedge (b_0\cdot\eps^B).
$$
This shows that $G(\sigma_T(x))=b_0\cdot\sigma_T(G(x))$. Thus, $G$ is stable iff $b_0=1$.
\Qed
\end{proof}

\begin{example}
\label{LN}
Let $\slnT:\Omega^*\lrow\Omega^*[b_1,b_2,\ldots]=\Omega^*[\barbi{b}]$ be the
total Landweber-Novikov operation.
It is the multiplicative operation corresponding to the power series
$x+b_1x^2+b_2x^3+\ldots$, where $b_i$ are independent variables
(see {\rm \cite[Example 4.1.25]{LM}} and {\rm\cite{Qu71}}).
By Proposition $\ref{multstable}$ this operation is stable.
\end{example}

Any stable multiplicative operation $G:\Omega^*\row B^*$ is a specialization of $\slnT$.
Namely, for each such $G$ there exists unique morphism of theories
$\theta_G:\Omega^*[b_1,b_2,\ldots]\row B^*$ such that $G=\theta_G\circ\slnT$.
This $\theta_G$ is the canonical morphism of theories on $\Omega^*$, and sends $b_i$'s
to the coefficients of $\gamma_G$.

\subsection{Stable operations in algebraic cobordism}

We already have seen an example of a stable operation, namely the total Landweber-Novikov operation
$\slnT:\Omega^*\row\Omega^*[\barbi{b}]$. In fact, all stable operations in algebraic cobordism are deduced from
the total Landweber-Novikov operation by change of coefficients, exactly as in the topological setting. More
precisely, we have the following result.

\begin{theorem}
\label{stabOsln}
The map sending a graded $\laz$-linear morphism $\psi:\laz[\barbi{b}]\row\laz$ of degree $l$
to the composition
$$
G_{\psi}:\Omega^*\stackrel{\slnT}{\lrow}\Omega^*[\barbi{b}]=\Omega^*\otimes_{\laz}\laz[\barbi{b}]\stackrel{\otimes\psi}{\lrow}\Omega^{*+l}
$$
is a bijection from the set $\op{Hom}_{\laz}(\laz[\barbi{b}],\laz)^{\ddeg=l}$ to the set of stable operations $\Omega^*\row\Omega^{*+l}$.
\end{theorem}

\begin{proof}
Since $\slnT$ and $\otimes\psi$ are stable operations, so is their composition.
Now, let $G:\Omega^*\row\Omega^{*+l}$ be a stable operation.
Then $G$ is additive. In particular, $G_k$ is an
additive graded homomorphism $\laz^*\row\laz^{*+l}$.
Consider the commutative diagram:
$$
\xymatrix @-0.2pc{
\laz \ar @{->}[r]^(0.5){\slnT} \ar @{^{(}->}[d]_(0.5){}&
\laz[\barbi{b}] \ar @{^{(}->}[d]^(0.5){}\\
\zz[\barbi{d}] \ar @{->}[r]_(0.4){S} & \zz[\barbi{d}][\barbi{b}],
}
$$
where the vertical maps are induced by the natural embedding of rings
$\laz\hookrightarrow\zz[d_1,d_2,\ldots]$ corresponding to the twist of the
additive formal group law by the change of parameter: $\delta(y)=y+d_1y^2+d_2y^3+\ldots$ (see \cite{Ad}, \cite{Qu71}),
and $S$ maps $d_i$ to the $i$-th coefficient $e_i$ of the power series
$\rho(y)=\beta(\delta(y))$ with $\beta(x)=x+b_1x^2+b_2x^3+\ldots$.
In particular, the vertical maps are isomorphisms after tensoring with $\qq$,
\begin{equation*}
\begin{split}
&\zz[\barbi{d}][\barbi{b}]=\zz[\barbi{d}][\barbi{e}]\hspace{5mm}\text{and}\hspace{5mm}
\laz[\barbi{b}]\otimes_{\zz}\qq=\laz[\barbi{e}]\otimes_{\zz}\qq.
\end{split}
\end{equation*}
We now claim that there exists a unique graded $\laz$-linear map $\psi_G:\laz[\barbi{b}]\row\laz\otimes_{\zz}\qq$ of degree $l$
such that the composition
$$
\laz\stackrel{\slnT}{\lrow}\laz[\barbi{b}]\stackrel{\psi_G}{\lrow}\laz\otimes_{\zz}\qq
$$
factors through the additive homomorphism $G_k:\laz\row\laz$. Indeed, by the preceding discussion, $G_k$ induces an additive
homomorphism $G_k\otimes\qq:\qq[\barbi{d}]\row\qq[\barbi{d}]$ and we need to show that there exists a $\qq[\barbi{d}]$-linear map
$\psi'_G:\qq[\barbi{d}][\barbi{b}]\row\qq[\barbi{d}]$ such that $G_k\otimes\qq=\psi'_G\circ(S\otimes\qq)$.
But this is now clear since $\qq[\barbi{d}][\barbi{b}]=\qq[\barbi{d}][\barbi{e}]$ and the map $S$ sends $d_i$ to $e_i$.
Consider the operation:
$$
H=G-\psi_G\circ\slnT:\Omega^*\lrow\Omega^{*+l}\otimes_{\zz}\qq.
$$
Let us show that $H=0$.

\begin{lemma}
\label{lemstab1}
Let $H:A^*\row B^{*+l}$ be a stable operation such that $H_X=0$. Then $H_{X\times\pp^1}=0$.
\end{lemma}

\begin{proof}
The maps $\xymatrix @-0.2pc{\pp^1 \ar @/^0.2pc/ @{->}[r] &\spec(k) \ar @/^0.2pc/ @{->}[l]}$
define the decomposition:
$C^*(X\times\pp^1)=C^*(X)\oplus C^*(\Sigma_T X)$ respected by additive operations.
Moreover, since $H$ is stable, and $H_X$ is zero, so is $H_{\Sigma_T X}$. Hence,
$H_{X\times\pp^1}=0$.
\Qed
\end{proof}

\begin{lemma}
\label{lemstab2}
Let $H:A^*\row B^{*+l}$ be a stable operation such that $H_k:A^*(k)\row B^{*+l}(k)$ is zero.
Assume that $B^*(k)$ has no torsion.
Then $H_{(\pp^{\infty})^{\times r}}:A^*((\pp^{\infty})^{\times r})\row B^{*+l}((\pp^{\infty})^{\times r})$
is zero for all $r$.
\end{lemma}

\begin{proof}
We need to show that $H_{(\pp^N)^{\times r}}=0$, for all $N$ and $r$. Consider the map
$p: ((\pp^1)^{\times N})^{\times r}\row (\pp^N)^{\times r}$ which is the $r$-th power of a linear projection
$q$ from $(\pp^1)^{\times N}\subset\pp^M$
to $\pp^N$. Let $x=c^B_1(O(1))$ be the 1-st Chern class of the canonical line bundle on $\pp^N$, and $x_1,\ldots,x_N$ be the 1-st Chern 
classes of the canonical line bundles on $\pp^1$-factors. Then
\begin{equation*}
\begin{split}
&q^*(x)=x_1+_Bx_2+_B\ldots+_Bx_N\equiv x_1+\ldots+x_N\,\,(\,mod\, \ddeg>1),\hspace{5mm}\text{and so}\\
&q^*(x^k)\equiv \sum_{\substack{I\subset \{1,\ldots,N\}\\ \#(I)=k}}\binom{N}{k}\prod_{i\in I}x_i\,\,(\,mod\,\ddeg>k).
\end{split}
\end{equation*}
Since $B^*(k)$ has no torsion, this proves that
$p^*:B^*((\pp^N)^{\times r})\row B^*(((\pp^1)^{\times N})^{\times r})$ is injective.
By Lemma \ref{lemstab1},
$H_{(\pp^{\infty})^{\times r}}=0$.
\Qed
\end{proof}

\begin{remark}
The condition that $B$ has no torsion is essential. Take, for example $A^*=B^*=\op{CH}^*/2$, and $H=G_1-G_2$, where $G_1=id$ with $\gamma_{G_1}=x$ and $G_2=\stT$ with $\gamma_{G_2}=x+x^2$ - the Total Steenrod operation.
Then $\ffi_{G_1}=\ffi_{G_2}$ since there exists only one homomorphism of rings $\zz/2\row\zz/2$, and
so $H_k=0$. At the same time, for $x=c^A_1(O(1))$, we have $H_{\pp^{\infty}}(x)=\gamma_{G_1}(x)-\gamma_{G_2}(x)=x^2\neq 0$.
\end{remark}

\begin{proposition}
\label{vazhnoe}
Let $A^*$ and $B^*$ be two theories in the sense of Definition \ref{goct}, and assume that $A^*$ satisfies $(CONST)$.
Let $H:A^n\row B^m$ be an additive operation (not necessarily stable!) such that
$H_{(\pp^{\infty})^{\times r}}=0$, for any $r$. Then $H=0$.
\end{proposition}

\begin{proof}
Let us prove by induction on the dimension of $X$ that $H_{X\times(\pp^{\infty})^{\times r}}=0$,
for all $r$.
The base ($\ddim(X)=0$) follows from our conditions.
Suppose $\ddim(X)=d$, and the statement is known for varieties of smaller dimension.
We know that $A^*(X\times(\pp^N)^{\times r})$ is a free module over $A^*(X)$ with basis consisting
of monomials $\ov{\xi}^{\ov{m}}=\prod_{i=1}^r\xi_i^{m_i}$ with $0\leq m_i\leq N$, where $\xi_i=c^A_1(O(1)_i)$ and
$O(1)_i$ is the line bundle on $(\pp^N)^{\times r}$ obtained by pulling back $O(1)$ along the projection to the $i$-th factor. Thus,
it is sufficient to prove that $H(x\cdot\ov{\xi}^{\ov{m}})=0$, for any $x\in A^{n-\sum_im_i}(X)$, for any $\ov{m}$.
Because $A^*$ satisfies $(CONST)$, we have: $H(x|_{\spec(k(X))}\cdot\ov{\xi}^{\ov{m}})=0$, and
by additivity of $H$ we can assume that $x|_{\spec(k(X))}=0$, that is, $x$ is supported on some
closed subvariety $Y\subset X$ (here we use $(EXCI)$). By Hironaka's resolution of singularities (see Theorem \ref{Hi}),
there exists a permitted blow up (see Definition \ref{perm-bu})
$\pi:\wt{X}\row X$ with centers over $Y$ and of dimension $<\ddim(Y)$,
such that the strict transform $\wt{Y}$ of $Y$ is smooth. Since $\pi^*:B^*(X)\row B^*(\wt{X})$ is
injective, it is sufficient to show that $H(\pi^*(x)\cdot\ov{\xi}^{\ov{m}})=0$. We have:
\begin{equation*}
\begin{split}
&\wt{X}=X_n\stackrel{\pi_n}{\row}X_{n-1}\stackrel{\pi_{n-1}}{\row}\ldots\stackrel{\pi_2}{\row}
X_1\stackrel{\pi_1}{\row}X_0=X\\
&\wt{Y}=Y_n\stackrel{\pi'_n}{\row}Y_{n-1}\stackrel{\pi'_{n-1}}{\row}\ldots\stackrel{\pi'_2}{\row}
Y_1\stackrel{\pi'_1}{\row}Y_0=Y,
\end{split}
\end{equation*}
where $X_{i+1}=Bl_{Z_i}X_i$, $Z_i\subset X_i$ is smooth of dimension $<\ddim(Y)$, and $Y_{i+1}$ is the strict transform of $Y_i$.
Let $y_i\in A^*(X_i)$ be some element with support on $Y_i$. Then it follows from (\ref{comp-seq}) (after the Definition \ref{non-smDef}) 
that $\pi_{i+1}^*(y_i)=y_{i+1}+u_{i+1}$,
where $y_{i+1}$ has support in $Y_{i+1}$ and $u_{i+1}$ has support in the special divisor
$\pp_{Z_i}(N_{Z_i\row X_i})$.

\begin{lemma}
\label{lem1stabO}
Let $H:A^n\row B^m$ be an additive operation. Let $X$ be a smooth quasi-projective variety, and let $Z\subset X$ be a smooth closed
subvariety of $X$ of codimension $l$. Consider the regular closed immersions $f:Z\hookrightarrow X$ and
$g:Z\hookrightarrow\pp_Z(N_f\oplus O)$.
Then, for every $u\in A^{n-l}(Z)$, the following implication holds: $H(g_*(u))=0 \hspace{2mm}\Rightarrow\hspace{2mm} H(f_*(u))=0$.
\end{lemma}

\begin{proof}
We use the deformation to the normal cone construction.
We have varieties $\wt{W}=Bl_{Z\times\{0\}}(X\times\aaa^1)$, $\wt{Z}=Z\times\aaa^1$,
$W_0=\pp_Z(N_f\oplus\co)$, $W_1=X\times\{1\}$, fitting into the diagram:
$$
\xymatrix @-0.2pc{
W_0 \ar @{->}[r]^(0.5){i_0} &
\wt{W}  & W_1 \ar @{->}[l]_(0.5){i_1}\\
Z \ar @{->}[r]_(0.5){j_0} \ar @{->}[u]^(0.5){g}& \wt{Z} \ar @{->}[u]^(0.5){h}&
Z \ar @{->}[l]^(0.5){j_1} \ar @{->}[u]_(0.5){f}
}
$$
with both squares transversal cartesian.
Let $\wt{Z}\stackrel{p}{\lrow}Z$ be the natural projection.
Since $B^*$ satisfies $(EXCI)$, $H(h_*p^*(u))$ has support in $\wt{Z}$. That is,
$H(h_*p^*(u))=h_*(v)$, for some $v\in B^*(\wt{Z})$. Then $i_0^*H(h_*p^*(u))=
H(i_0^*h_*p^*(u))=H(g_*j_0^*p^*(u))=H(g_*(u))=0$ should be equal to $i_0^*h_*(v)=g_*j_0^*(v)$.
But $j_0^*$ is an isomorphism, and $g_*$ is an injection. Hence, $v=0$, and so $H(h_*p^*(u))=0$.
This implies that: $0=i_1^*H(h_*p^*(u))=H(i_1^*h_*p^*(u))=H(f_*j_1^*p^*(u))=H(f_*(u))$.
\Qed
\end{proof}

\begin{lemma}
\label{lem2stabO}
Let $V$ be a vector bundle on $Z$, and $P=\pp_Z(V)$.
Let $H:A^n\row B^m$ be an additive operation s.t. $H_{Z\times(\pp^{\infty})^{\times r}}=0$,
$\forall r$. Then $H_{P\times(\pp^{\infty})^{\times r}}=0$, $\forall r$.
\end{lemma}

\begin{proof}
$A^*(P)$ as an $A^*(Z)$-module is generated by powers of $c^A_1(O_P(1))$. There are
very ample line bundles $L_1,L_2$ on $P$ such that $O_P(1)=L_1\otimes L_2^{-1}$. Hence, any element
in $A^*(P)$ can be written as an $A^*(Z)$-linear combination of
$c^A_1(L_1)^{m_1}\cdot c^A_1(L_2)^{m_2}$. And each such element is a pull-back of a certain
element from $A^*(Z\times(\pp^{\infty})^{\times 2})$. Thus, any element from
$A^*(P\times(\pp^N)^{\times r})$ is a sum of elements pulled back from $A^*(Z\times (\pp^M)^{\times r+2})$,
and so $H$ must be trivial on it.
\Qed
\end{proof}

\begin{lemma}
\label{lem3stabO}
Let $f:Z\hookrightarrow X$ be a closed immersion between smooth varieties.
Assume that $H_{Z\times(\pp^{\infty})^{\times r}}=0$ for all $r$.
Then $H_{X\times(\pp^{\infty})^{\times r}}$ is zero on the image of
$(f\times id)_*:A^*(Z\times(\pp^{\infty})^{\times r})\row A^*(X\times(\pp^{\infty})^{\times r})$ for all $r$.
\end{lemma}

\begin{proof}
Follows immediately from Lemmas \ref{lem1stabO} and \ref{lem2stabO}.
\Qed
\end{proof}

We now return to the proof of Proposition \ref{vazhnoe}.
Take now $y_0=x$, and construct the elements $y_i,u_i$ as above. 
Since $u_{i+1}$ has support on a smooth subvariety $\pp_{Z_i}(N_{Z_i\subset X_i})$,
it follows from the inductive assumption and Lemma \ref{lem3stabO} that
$H(u_{i+1}\cdot\ov{\xi}^{\ov{m}})=0$ and, thus, $H(u_{i+1}|_{\wt{X}}\cdot\ov{\xi}^{\ov{m}})=0$.
Then $\wt{y}=y_n$ has support in $\wt{Y}$,
and by the above, $H(\wt{y}\cdot\ov{\xi}^{\ov{m}})=H(\pi^*(x)\cdot\ov{\xi}^{\ov{m}})$. Thus, we can reduce
to the case where $x$ has support on a smooth subvariety $Y\subset X$, where it follows
from the inductive assumption and Lemma \ref{lem3stabO}.
Induction step is done, and Proposition \ref{vazhnoe} is proven.
\Qed
\end{proof}

We now return to the proof of Theorem \ref{stabOsln}. By Lemma \ref{lemstab2}  and Proposition \ref{vazhnoe},
the composition
$\Omega^*\stackrel{G}{\lrow}\Omega^*\hookrightarrow\Omega^*\otimes_{\zz}\qq$ coincides with the
composition
$\Omega^*\stackrel{\slnT}{\lrow}\Omega^*[\barbi{b}]\stackrel{\psi_G}{\lrow}\Omega^*\otimes_{\zz}\qq$,
that is $G$ is a linear combination (infinite, in general) of the Landweber-Novikov operations.
It remains to show that $\psi_G:\laz[\barbi{b}]\row\laz\otimes\qq$ takes values in $\laz$. As $\psi_G$ is
$\laz$-linear, it is enough to show that $\psi_G(\ov{b}^{\ov{r}})\in\laz$. We argue by induction on the degree
of the monomial $\ov{b}^{\ov{r}}$. (When the degree is zero, the result follows from the fact that
$\psi_G(1)=\psi_G\circ\slnT(1)=G_k(1)$ which belongs to $\laz$.)
Consider $X=\times_i(\pp^{i+1})^{\times r_i}$, and
$x=\times_i(h_i)^{\times r_i}$, where $h_i=c^{\Omega}_1(O_{\pp^{i+1}}(1))$.
First, we compute $\slnT(x)$. As $\slnT$ is multiplicative, we have
$$
\slnT(x)=\times_i\slnT(h_i)^{\times r_i}.
$$
Moreover, by the very definition of $\slnT$ (see Example \ref{LN}),
$$
\slnT(h_i)=h_i+h_i^2\cdot b_1+h_i^3\cdot b_2+\ldots+h_i^{i+1}\cdot b_i
$$
and $h_i^{i+1}=[pt]$ is the class of a point in $\pp^{i+1}$. It follows that
$$
\slnT(x)=[pt]\cdot\ov{b}^{\ov{r}}+\sum_{\ddeg(\ov{s})<\ddeg(\ov{r})}\mu_{\ov{s}}\cdot\ov{b}^{\ov{s}}
$$
for some coefficients $\mu_{\ov{s}}\in\Omega^*(X)$. By our induction hypothesis $\psi_G(\ov{b}^{\ov{s}})\in\laz$.
Therefore, to show that $\psi_G(\ov{b}^{\ov{r}})\in\laz$, it is enough to show that $\psi_G(\slnT(x))\in\Omega^*(X)$.
But, we already know that $\psi_G(\slnT(x))=G(x)$. This proves that $\psi_G$ takes values in $\laz$.
Using \ref{lemstab2} and Proposition \ref{vazhnoe} again, we obtain that $G=\psi_G\circ\slnT$ integrally.
This finishes the proof of Theorem \ref{stabOsln}.
\Qed
\end{proof}

\subsection{Unstable operations in algebraic cobordism (uniqueness)}

Unstable operations can be described in terms of stable ones.
In analogy with topology we have:

\begin{theorem}
\label{unstUNIQ}
Let $G:\Omega^n\row\Omega^m$ be an additive operation. Then there exists unique
$\psi_G\in\op{Hom}_{\laz}(\laz[\barbi{b}],\laz\otimes_{\zz}\qq)_{(m-n)}$ such that
$G\otimes\qq:\Omega^n\row\Omega^m\otimes\qq$ coincides with the composition
$\Omega^*\stackrel{\slnT}{\lrow}\Omega^*[\barbi{b}]\stackrel{\otimes\psi_G}{\lrow}\Omega^{*+m-n}\otimes_{\zz}\qq$
in degree $n$. This way, the set of additive operations $G:\Omega^n\row\Omega^m$ is identified with a subset
of $\op{Hom}_{\laz}(\laz[\barbi{b}],\laz\otimes_{\zz}\qq)_{(m-n)}$.
\end{theorem}

\begin{proof}
By Proposition \ref{vazhnoe} we know that any additive operation $G:\Omega^n\row\Omega^m$
is completely determined by it's action on $\Omega^n((\pp^{\infty})^{\times r})$, for all $r$. Thus, it
is sufficient to show that there exists a unique $\laz\otimes_{\zz}\qq$-linear combination of the
Landweber-Novikov operations which coincides with $G\otimes\qq$ on $\Omega^n((\pp^{\infty})^{\times r})$,
for all $r$.
We have mutually inverse operations:
$$
\xymatrix @-0.2pc{
\Omega^*\otimes_{\zz}\qq \ar @/^0.5pc/  @{->}[r]^(0.45){\alpha} &
\op{CH}^*\otimes_{\zz}\qq[\barbi{d}]\ar @/^0.5pc/ @{->}[l]^(0.55){\beta},
}
$$
where $\gamma_{\alpha}^{-1}=x+d_1x^2+d_2x^3+\ldots=\ffi_{\alpha}(log_{\Omega})$, and
$\gamma_{\beta}=log_{\Omega}$.
Thus, we obtain a commutative diagram:
$$
\xymatrix @-0.7pc{
\Omega^n\otimes_{\zz}\qq \ar @{->}[r]^(0.5){G\otimes\qq} &
\Omega^m\otimes_{\zz}\qq \ar @{->}[d]^(0.5){\alpha}\\
(\op{CH}^*\otimes_{\zz}\qq[\barbi{d}])_{(n)} \ar @{->}[r]_(0.5){H} \ar @{->}[u]^(0.5){\beta} & (\op{CH}^*\otimes_{\zz}\qq[\barbi{d}])_{(m)},
}
$$
where $H$ is an additive operation between additive theories.

Let $A^*$ and $B^*$ be two theories in the sense of Definition \ref{goct}.
Let $x_i=c^A_1(O_{\pp^{\infty}}(1)_i)$ and $y_i=c^B_1(O_{\pp^{\infty}}(1)_i)$, where $O_{\pp^{\infty}}(1)_i$ is the pull-back of the
canonical line bundle $O_{\pp^{\infty}}(1)$
along the projection to the $i$-th component $\pi_i:(\pp^{\infty})^{\times r}\row\pp^{\infty}$.
Then $A^*((\pp^{\infty})^{\times r})$ is a free $A^*(k)$-module with the monomial basis $\ov{x}^{\ov s}=x_1^{s_1}\cdots x_r^{s_r}$,
and $B^*((\pp^{\infty})^{\times r})$ is a free $B^*(k)$-module with the basis $\ov{y}^{\ov s}$.

\begin{lemma}
\label{lemUNST}
Let $H:A^n\row B^m$ be an additive operation of additive theories. Suppose $B^*(k)$ has no torsion.
Then there exists unique homomorphism of abelian groups
$A^*(k)\stackrel{\wt{H}}{\row}B^{*+m-n}(k)$ such that\\
$H(u\cdot\ov{x}^{\ov{s}})=\wt{H}(u)\cdot\ov{y}^{\ov{s}}$, for all $\ov{s}$ and all $u\in A^{n-\ddeg(\ov{s})}(k)$.
\end{lemma}

\begin{proof}
Because of the partial diagonals, it is sufficient to treat the case
$\ov{x}^{\ov{s}}=x_1\cdot x_2\cdot\ldots\cdot x_r$.
We need to show that $H(u\cdot\ov{x}^{\ov{s}})$ is a multiple of $\ov{y}^{\ov{s}}$.
In other words, that it is poly-linear in $y_i$'s.
Let $P_i=\prod_{\substack{1\leq j\leq r\\ j\neq i}}\pp^{\infty}=(\pp^{\infty})^{\times (r-1)}$.
Consider the theories
$(A')^*:=A^*_{P_i/k}$ and $(B')^*:=B^*_{P_i/k}$ (Example \ref{nonconst}).
Identifying $A^*((\pp^{\infty})^{\times r})$ with $(A')^*(\pp^{\infty})$ and
$B^*((\pp^{\infty})^{\times r})$ with $(B')^*(\pp^{\infty})$ we reduce the problem to the case $r=1$.
In the case of one variable, consider the Segre embedding
$\pp^{\infty}\times\pp^{\infty}\stackrel{f}{\hookrightarrow}\pp^{\infty}$.
Then $f^*(u\cdot x)=u\cdot x_1+u\cdot x_2$ (recall that the theory $A^*$ is additive).
Let $H(u\cdot x)=\gamma(y)=\gamma_0+\gamma_1y+\gamma_2y^2+\ldots\in B^*[[y]]$.
Restricting along $\spec(k)\hookrightarrow\pp^{\infty}$, we see that $\gamma_0=0$.
Write $\gamma(y)=\gamma_1\cdot y+\gamma_s\cdot y^s+\ldots$ with $\gamma_s\neq 0$. Then from the equality:
$f^*(H(u\cdot x))=H(f^*(u\cdot x))$, we get:
$$
\gamma(y_1+y_2)=\gamma(y_1)+\gamma(y_2).
$$
Comparing coefficients at $y_1\cdot y_2^{s-1}$, we obtain: $s\cdot\gamma_s=0$.
Since $B$ has no torsion, we get that $\gamma(y)=\gamma_1\cdot y$ is linear.
Thus, we have shown that $H(u\cdot(x_1\cdot\ldots\cdot x_r))=v\cdot(y_1\cdot\ldots\cdot y_r)$,
and the correspondense $u\mapsto v$ defines an additive map
$A^{n-r}(k)\stackrel{\wt{H}}{\row}B^{m-r}(k)$. The uniqueness of $\wt{H}$ is obvious.
\Qed
\end{proof}

The map $\ffi_G:\laz\row\laz\otimes_{\zz}\qq\xrightarrow{\beta\circ\wt{H}\circ\alpha|_{\spec(k)}}
\laz\otimes_{\zz}\qq$ is additive.
As we saw in the proof of Theorem \ref{stabOsln}, this map can be presented as the composition:
$S_k:\laz\stackrel{\slnT}{\lrow}\laz[\barbi{b}]\stackrel{\otimes\psi}{\lrow}\laz\otimes_{\zz}\qq$,
for a unique $\psi\in\op{Hom}_{\laz}(\laz[\barbi{b}],\laz\otimes_{\zz}\qq)_{(m-n)}$.
Note that, for the respective $\laz\otimes_{\zz}\qq$-linear combination of the Landweber-Novikov operations
$S:\Omega^*\stackrel{\slnT}{\lrow}\Omega^*[\barbi{b}]\stackrel{\otimes\psi}{\lrow}\Omega^{*-n+m}\otimes_{\zz}\qq$ (in degree $n$),
the analogous map $\ffi_S:\laz\row\laz\otimes_{\zz}\qq$ coincides with $S_k$
(since the operations $\slnT$, $\alpha$ and $\beta$ are multiplicative and stable), and so, with $\ffi_G$.
Then Lemma \ref{lemUNST} shows that on $(\pp^{\infty})^{\times r}$, for all $r$, $G$ coincides with $S$.
\Qed
\end{proof}

The natural question arises: which rational combinations of Landweber-Novikov operations
correspond to (unstable) operations $\Omega^n\row\Omega^m$?
This question will be effectively answered later in this paper; see Theorem \ref{UOACpoln} which is one of our main results.

\section{Theories of rational type}
\label{TRT}
Our method of constructing unstable operations relies on a description of the source theory which is inductive on dimensions.
Not all theories admit such a description; those who do will be called of {\it rational type}.
Later we will see that these are exactly the {\it free} theories
of M.Levine-F.Morel. The needed description of the theory will be obtained
in stages. The one which is actually used is provided by the short bi-complex $\frc$,
but to get there we will need to introduce short bi-complexes $\fra$ and $\frb$,
and to show that the Levine-Morel algebraic cobordism is a theory of
{\it rational type}.

\subsection{The short bi-complex $\fra$}
\label{a}
Everywhere in this and the next Subsection we will assume that
$A^*$ is a theory in the sense of Definition \ref{goct} satisfying $(CONST)$.
Some statements are valid without the latter assumption, which will be indicated.
Let $X$ be a smooth irreducible variety over $k$.
Consider the category $\smu(X)$ whose objects are maps
$V\stackrel{v}{\row}X$, where $V$ is smooth, $v$ is projective, and $\ddim(V)<\ddim(X)$,
and morphisms are projective maps $V_2\stackrel{f}{\row}V_1$ such that $v_2=v_1\circ f$.

Similarly, we have a category $\smu^1(X)$ whose objects are
maps $W\stackrel{w}{\row}X\times\pp^1$, where $W$ is smooth, $w$ is projective,
$\ddim(W)\leq\ddim(X)$, and
$W_0=w^{-1}(X\times\{0\})\stackrel{i_0}{\hookrightarrow} W$,
$W_{\infty}=w^{-1}(X\times\{\infty\})\stackrel{i_{\infty}}{\hookrightarrow} W$ are divisors
with strict normal crossings. The morphisms are projective maps $W_2\stackrel{g}{\row}W_2$ such that
$w_2=w_1\circ g$.

We have natural maps $\partial_0,\partial_{\infty}:\smu^1(X)\row\smu(X)$ defined by:
$$
\partial_l(W\stackrel{w}{\row}X\times\pp^1)=\komp{W}_l\stackrel{\komp{w}_l}{\row}X
$$
where, for a divisor with strict normal crossings $D$ with irreducible components
$D_1,\ldots,D_r$, $\komp{D}=\coprod_{\emptyset\neq J\subset \{1,\ldots,r\}}D_J$ with $D_J=\cap_{j\in J}D_j$ (see Definition \ref{ncd}).

Below we will use the term {\it short bi-complex} for a bi-complex which is
zero except possibly in homological degrees $(0,0)$, $(1,0)$ and $(0,1)$.
Consider the following short bi-complex ${\fra}={\fra}(A^*)$:
$$
\begin{CD}
\Da_{1,0}& @>{d_{1,0}}>>& \Da_{0,0}\\
@.& & @AA{d_{0,1}}A\\
@.& & \Da_{0,1}
\end{CD},
$$
where
\begin{itemize}
\item[$\cdot$ ]
$\Da_{0,0}=\bigoplus\limits_{V\in Ob(\smu(X))} A_*(V)$; \hspace{10mm}
$\Da_{1,0}=\bigoplus\limits_{V_2\row V_1\in\cmor(\smu(X))}A_*(V_2)$; \hspace{10mm}
$\Da_{0,1}=\bigoplus\limits_{W\in Ob(\smu^{1}(X))}A_{*+1}(W)$.
\end{itemize}
and the differentials are defined as follows:
\begin{itemize}
\item[$\cdot$ ] $d_{1,0}(V_2\stackrel{f}{\row}V_1,y)=
(V_2,y)-(V_1,f_*(y))$;
\item[$\cdot$ ] $d_{0,1}(W,z))=
(\partial_0W,i_0^{\sstar}(z))-(\partial_{\infty}W,i_{\infty}^{\sstar}(z))$ - see
Definition \ref{starpullback}, where we use the standard choice for the coefficients
$F^{l_1,\ldots,l_r}_J$ (as soon as we pass to $\coker(d_{1,0}^{\fra})$ the latter becomes
irrelevant).
\end{itemize}

We denote by $H(\fra)$ the $0$-th homology of the total complex $Tot(\fra)$ of $\fra$.
In other words,
$$
H({\fra})=\coker(\Da_{1,0}\oplus \Da_{0,1}\xrightarrow{d_{1,0}\oplus d_{0,1}}\Da_{0,0}).
$$

Assume that $X$ is connected and that $A^*$ satisfies $(CONST)$. Then the restriction to the
generic point $A^*(X)\row A^*(k(X))$ is surjective and has a canonical section given by
$A^*(k)\row A^*(X)$. Setting $\ov{A}^*(X)=\kker(A^*(X)\row A^*(k(X)))$, one gets a canonical decomposition
$A^*(X)=\ov{A}^*(X)\oplus A^*(k)$.

The push-forwards define a natural map $\Da_{0,0}\row \ov{A}_*(X)$, and it follows from
Proposition \ref{MPcor} that it descends to a map $\theta_{\fra}: H({\fra})\row\ov{A}_*(X)$.
By $(EXCI)$ and resolution of singularities (Theorem \ref{Hi}), this map is surjective.

\begin{definition}
\label{DRT}
Let $A^*$ be an oriented cohomology theory in the sense of Definition \ref{goct}
satisfying $(CONST)$. We say that $A^*$ is "of rational type" if the map
$\theta_{\fra}: H({\fra})\row\ov{A}_*(X)$ is an isomorphism.
\end{definition}

\begin{remark} Not all constant theories are of rational type. For example,
$\op{CH}_{alg}$ - the Chow groups modulo algebraic equivalence is not such.
Indeed, in this case, for a curve $C$, the map $\theta_{\fra}$ can be identified
with the natural map $\ov{CH}_*(C)\row\ov{CH}_{alg,*}(C)$, which has a nontrivial kernel
when the genus of $C$ is nonzero.
\end{remark}

Below, we will see that rational theories are precisely the free theories in the sense of Levine-Morel
(see Proposition \ref{KlassRT}). In fact, this is an easy consequence of the following result.

\begin{proposition}
\label{CobRT}
The Levine-Morel
algebraic cobordism is a theory of rational type.
\end{proposition}

\begin{proof}
We will use a result of Levine \cite{Lcomp} which is a by-product of his proof that $\Omega^*$ is the $(2*,*)$-part of
Voevodsky's $MGL$. Let $X$ be an irreducible smooth quasi-projective variety of dimension $d$. Set
$\Omega^{(1)}_*(X)=\op{colim}_{W\subset X}\Omega_*(X)$, where the colimit is over closed subvarieties $W\subset X$
different from $X$. In \cite[pages 3315-3316]{Lcomp} Levine constructs a map
$$
\op{div}^{\Omega}:\zz[k(X)^{\times}]\otimes\laz_{*-d+1}\row\Omega^{(1)}_*(X)
$$
and it follows from the commutative diagram of \cite[page 3315]{Lcomp} and \cite[Theorem 3.1 and Lemma 4.3]{Lcomp}
that we have an exact sequence
$$
\zz[k(X)^{\times}]\otimes\laz_{*-d+1}\xrightarrow{\op{div}^{\Omega}}\Omega^{(1)}_*(X)\row\Omega_*(X)
\row\laz_{*-d}\row 0.
$$
Recall that $\op{div}^{\Omega}$ is $\laz$-linear and its value on a rational function $f\in k(X)^{\times}$
is described as follows. By Hironaka's resolution of singularities (Theorem \ref{Hif}), we may find a blowup
$\pi:\wt{X}\row X$ such that $f$ extends to a morphism $\wt{f}:\wt{X}\row\pp^1$ such that $X_0=\wt{f}^{-1}(0)$
and $X_{\infty}=\wt{f}^{-1}(\infty)$ are divisors with strict normal crossings.
Then $\op{div}^{\Omega}(f)=\pi_*([\wt{X}_0]-[\wt{X}_{\infty}])$ where $[\wt{X}_0]$, $[\wt{X}_{\infty}]$ are as in
Definition \ref{divclass}.

Consider the categories $\smu'(X)$ and ${\smu'}^1X)$ defined similarly to $\smu(X)$ and $\smu^1(X)$,
but with different dimension conditions: $\ddim(image(v))<\ddim(X)$, $\ddim(image(w))\leq\ddim(X)$.

For any theory $A^*$ we can define the following short bi-complex $\fraO=\fraO(A^*)$:
\begin{itemize}
\item[$\cdot$ ]
$\Da'_{0,0}=\bigoplus\limits_{V\in Ob(\smu'(X))} A_*(V)$; \hspace{10mm}
$\Da'_{1,0}=\bigoplus\limits_{V_2\row V_1\in\cmor(\smu'(X))}A_*(V_2)$; \hspace{10mm}
$\Da'_{0,1}=\bigoplus\limits_{W\in Ob({\smu'}^{1}(X))}A_{*+1}(W)$,
\end{itemize}
where the differentials and $H(\fraO)$ are defined as for ${\fra}$.

Now assume that $A^*=\Omega^*$.
The push-forwards provide a natural map $\Da'_{0,0}\row\Omega^{(1)}_*$, which clearly descends to
the map $\alpha':\coker(d_{1,0}^{\fraO})\row\Omega^{(1)}_*(X)$.

\begin{lemma}
\label{bO}
The map $\alpha':\coker(d_{1,0}^{\fraO})\row\Omega^{(1)}_*(X)$ is an isomorphism.
\end{lemma}

\begin{proof}
We have
$$
\Omega^{(1)}_*(X)=\op{colim}_{Z\subsetneq X}\Omega_*(Z)=\op{colim}_{Z\subsetneq X}\op{colim}_{V\row Z}A_*(V)
$$
where $V\row Z$ runs over projective maps from smooth varieties. This shows that
$$
\Omega^{(1)}_*(X)=\op{colim}_{V\row X\in\smu'(X)}A_*(X)
$$
which is computed by $\coker(d_{1,0}^{\fraO})$ as needed.
\Qed
\end{proof}

In the same way, for any theory in the sense of Definition \ref{goct}, we have:
$$
\operatornamewithlimits{colim}\limits_{\stackrel{\row}{Z\subsetneq X}}A_*(Z)=
\operatornamewithlimits{colim}\limits_{\stackrel{\row}{\smu'(X)}}A_*(V)=
\coker(d_{1,0}^{\fraO}).
$$

From here it is easy to see that $H(\fraO)\row\ov{\Omega}_*(X)$ is an isomorphism,
but we will compare $\fraO$ and ${\fra}$ first.

We have a natural map of bi-complexes ${\fra}\row\fraO$ which gives us the map
$\alpha:\coker(d_{1,0}^{\fra})\row\coker(d_{1,0}^{\fraO})$, and
$\hat{\alpha}: H({\fra})\row H({\fraO})$.

\begin{lemma}
\label{ab}
For any theory $A^*$ in the sense of Definition \ref{goct},
the map
$$
\alpha:\coker(d_{1,0}^{\fra})\row\coker(d_{1,0}^{\fraO})
$$
is an isomorphism.
\end{lemma}

\begin{proof}
({\bf surjectivity})
Consider some $v:V\row X$ in $\smu'(X)$, and $x\in A_*(V)$. We want to show that the class of $x$ in the $\coker(d_{1,0}^{\fraO})$
belongs to the image of $\alpha$.
Let $Z\subset X$ be the image of $V$. Using Hironaka's resolution of singularities (Theorems \ref{Hi} and \ref{Hif})
we can find a commutative square
\begin{equation}
\label{(1)}
\begin{CD}
\widetilde{V}&@>{\pi_v}>>&V\\
@V{\widetilde{p}}VV&&@VV{p}V\\
\widetilde{Z}&@>>{\pi_z}>&Z
\end{CD}
\end{equation}
where $\wt{V}$ and $\wt{Z}$ are smooth and $\pi_v$, $\pi_z$ are blowups. Denote by $\wt{z}:\wt{Z}\row X$ the obvious map.
Since $(\pi_v)_*:A_*(\wt{V})\row A_*(V)$ is surjective (see Proposition \ref{pi1invert}), we can find 
$\wt{x}\in A_*(\wt{V})$ such that $x=(\pi_v)_*(\wt{x})$.
But then, $(v,x)$, $(v\circ\pi_v,\wt{x})$ and $(\wt{z},\wt{p}_*(\wt{x}))$ have the same class in $\coker(d_{1,0}^{\fraO})$.
As $\ddim(\wt{Z})<\ddim(X)$, we are done.

\noindent({\bf injectivity})
It is enough to construct a map
$$
s:\coker(d_{1,0}^{\fraO})\row\coker(d_{1,0}^{\fra})
$$
which is a section to $\alpha$, i.e., such that $s\circ\alpha=id$.

Given $v:V\row X$ in $\smu'(X)$, we choose $\pi_v:\wt{V}\row V$ and $\pi_z:\wt{Z}\row Z$ as in (\ref{(1)}).
Given $x\in A_*(V)$ we choose $\wt{x}\in A_*(\wt{V})$ such that $(\pi_v)_*(\wt{x})=x$. We then set
$$
s((v,x))=[(\wt{z},\wt{x})]
$$
where the class is taken in $\coker(d_{1,0}^{\fra})$.

We claim that $s:\Da'_{0,0}\row\coker(d_{1,0}^{\fra})$ is well defined, i.e., $s((v,x))$ is independent of the choices we made.
First, we treat the independence of the choice of the lift $\wt{x}$. This is a consequence of the following lemma.

\begin{lemma}
\label{153}
Consider a commutative square
$$
\begin{CD}
W & @>{p}>> & V \\
@V{q}VV && @VV{v}V \\
T & @>>{t}> & X
\end{CD}
$$
where $t,v,p$ and $q$ are projective, $p$ is dominant, $T$, $V$ and $W$ are smooth, and $\ddim(T)<\ddim(X)$.
Let $y\in A_*(W)$ and assume that $p_*(y)=0$. Then, the class of $(t,q_*(y))$ in $\coker(d_{1,0}^{\fra})$ is zero.
\end{lemma}

\begin{proof}
We argue by induction on $\ddim(W)$. (The case of empty $W$ (negative dimension) is clear.)
We start by some reductions.
\begin{itemize}
\item[$\bullet$] If $W'$ is a blowup with $W'$ smooth, we may replace $W$ by $W'$ since $A_*(W')\row A_*(W)$ is surjective.
Doing this, we may assume that $q$ factors as $W\row S\row T$ where $S$ is smooth of dimension $<\ddim(X)$, $S\row T$ projective
and $W\row S$ dominant. Replacing $T$ by $S$, we may assume that $W\row T$ is dominant.
\item[$\bullet$] Let $Z$ be a resolution of the image of $W$ (and also $T$) in $X$. Replacing $W$ and $T$ by blowups, we may assume
that we have a chain of projective morphisms $W\row T\row Z\row X$. At this stage, we may replace $T$ by $Z$ and assume that
$T\row X$ is generically a locally closed immersion to $X$.
\item[$\bullet$] Since $W\row V$ is dominant, we may find a blowup $V'\row V$ such that $V'\row X$ factors (uniquely) through $T$.
Replacing $W$ by a blowup, we may assume that $W\row V$ factors as $W\row V'\row V$. At this stage, we may replace $W$ by $V'$
and assume that $p:W\row V$ is a blowup, a more precisely, a permitted sequence of blowups at smooth centers
(see Definition \ref{perm-bu}).
\end{itemize}
We may now proceed to the actual proof of the Lemma. Write $p:W\row V$ as a sequence of blowups
$$
W=V_n\row V_{n-1}\row\ldots\row V_0=V
$$
where $V_i\row V_{i-1}$ is a blowup of smooth subvariety $R_i\subset V_{i-1}$. We denote $E_i\subset W$ the strict transform of
the exceptional divisor of the blowup $V_i\row V_{i-1}$. Thus, we have a map $E_i\row R_i$ and, by Proposition \ref{vvter}(1),
we have an exact sequence
$$
0\low A_*(V)\llow A_*(W)\llow
\oplus_i\kker(A_*(E_i)\row A_*(R_i)).
$$
Thus, we may assume that $y=(e_i)_*(z)$ for $z\in\kker(A_*(E_i)\row A_*(R_i))$ for some $1\leq i\leq n$, where $e_i:E_i\row W$ is
the obvious inclusion.
We may now use the induction hypothesis in the case of the square
$$
\begin{CD}
E_i & @>>> & R_i\\
@VVV && @VVV \\
T & @>>> & X
\end{CD}
$$
to conclude.
\Qed
\end{proof}

It is now easy to prove independence of the choices of $\wt{V}$ and $\wt{Z}$. Indeed, let $\pi_v':\wt{V}'\row V$ and
$\pi_z':\wt{Z}'\row Z$ be two other choices. We may assume that these are finer resolutions, i.e., that we have commutative diagrams
$$
\xymatrix{\wt{V}' \ar@{->}[r]^{f} \ar@{->}[rd] & \wt{V} \ar@{->}[d] \\
& V}\hspace{8mm}
\xymatrix{\wt{Z}' \ar@{->}[r]^{g} \ar@{->}[rd] & \wt{Z} \ar@{->}[d] \\
& Z}\hspace{8mm}
\xymatrix{\wt{V}' \ar@{->}[r]^{f} \ar@{->}[d]_{\wt{p}'} & \wt{V} \ar@{->}[d]^{\wt{p}} \\
\wt{Z}' \ar@{->}[r]^{g} & \wt{Z}}.
$$
As we know independence of lifts, we may assume that $\wt{x}=f_*(\wt{x}')$. The equality of the classes of $(\wt{z},\wt{p}_*(\wt{x}))$
and $(\wt{z}',\wt{p}'_*(\wt{x}'))$ follows from the equalities $g_*\wt{p}'_*(\wt{x}')=\wt{p}_*f_*(\wt{x}')=\wt{p}_*(\wt{x})$.

Next we claim that $s:\Da'_{0,0}\row\coker(d_{1,0}^{\fra})$ descends to $s:\coker(d_{1,0}^{\fraO})\row\coker(d_{1,0}^{\fra})$ giving
the needed section.
This is now quite easy. Indeed, given a morphism $f:V_2\row V_1$ in $\smu'(X)$, we may find commutative diagrams
$$
\xymatrix{
\wt{V}_2 \ar@{->}[r]^{\wt{f}} \ar@{->}[d]_{\pi_2} & \wt{V}_1 \ar@{->}[d]^{\pi_1} \\
V_2 \ar@{->}[r]_{f} & V_1
}\hspace{1.0cm}
\xymatrix{
\wt{V}_2 \ar@{->}[r]^{\wt{f}} \ar@{->}[d]^{q_2} \ar @/_1.5pc/ @{->}[dd]_{p_2} & \wt{V}_1 \ar@{->}[d]^{p_1} \\
T \ar@{->}[r]^{h} \ar@{->}[d]^{l} & \wt{Z}_1 \\
\wt{Z}_2 &
}
$$
where $\pi_1$, $\pi_2$ are blowups, all the varieties are smooth, $\wt{Z}_1$ and $\wt{Z}_2$ are resolutions od singularities
of the images of $V_1$ and $V_2$ in $X$. $T$ is the irreducible component of
the inverse image of $Z_2$ in $\wt{Z}_1$ containing the image of $\wt{V}_2$.

This is said, the equality $s((v_1, f_*(x_2)))=s((v_2,x_2))$ follows from
$$
[(\wt{z}_1,(p_1)_*\wt{f}_*(\wt{x}_2))]=[(\wt{z}_1\circ h,(q_2)_*(\wt{x}_2))]=[(\wt{z}_2,(l\circ q_2)_*(\wt{x}_2))]=
[(\wt{z}_2,(p_2)_*(\wt{x}_2))].
$$
Lemma \ref{ab} is proven.
\Qed
\end{proof}

We return now to the case $A^*=\Omega^*$ and complete the proof of Proposition \ref{CobRT}.
We have the commutative diagram with exact columns:
\begin{equation}
\label{BD}
\xymatrix @-0.7pc{
\Da_{0,1} \ar @{->}[r] \ar @{->}[d]_(0.5){\hat{d}_{0,1}^{\fra}} &
\DaO_{0,1} \ar @{->}[d]_(0.5){\hat{d}_{0,1}^{\fraO}} & \zz[k(X)^{\times}]\otimes\laz\ar @{->}[d]^(0.5){\op{div}}\\
\coker(d_{1,0}^{\fra}) \ar @{->}[r]_(0.5){\alpha} \ar @{->}[d]&
\coker(d_{1,0}^{\fraO}) \ar @{->}[r]_(0.5){\alpha'} \ar @{->}[d]& \Omega^{(1)}_*(X) \ar @{->}[d]\\
H({\fra}) \ar @{->}[r]_(0.5){\hat{\alpha}} \ar @{->}[d]&
H(\fraO) \ar @{->}[r]_(0.5){\hat{\alpha'}} \ar @{->}[d]& \ov{\Omega}_*(X) \ar @{->}[d]\\
0&0&0
}
\end{equation}
where $\alpha$ and $\alpha'$ are isomorphisms.
It remains to observe that the map $\op{div}$ can be factored through $\Da_{0,1}$ by the very definition.
This shows that the maps
$$
H({\fra})\stackrel{\hat{\alpha}}{\lrow}H(\fraO)\stackrel{\hat{\alpha}'}{\lrow}\ov{\Omega}_*(X)
$$
are isomorphisms.
\Qed
\end{proof}

Using Proposition \ref{CobRT} we can describe all theories of
rational type as follows.

\begin{proposition}
\label{KlassRT}
Let $A^*$ be a theory in the sense of Definition \ref{goct} satisfying $(CONST)$. Then
$A^*$ is of rational type if and only if $A^*$ is free in the sense of Levine-Morel, i.e.,
if the natural map $\Omega^*\otimes_{\laz}A^*(k)\row A^*$ is an isomorphism.
\end{proposition}

\begin{proof}
Since the tensor product functor is right exact, any theory
of the form $\Omega^*\otimes_{\laz}A$ will be of rational type,
since $\Omega^*$ is (Proposition \ref{CobRT}).

Conversely, assume that $A^*$ is of rational type. By the universality of $\Omega^*$ (see \cite[Theorem 1.2.6]{LM}),
we have a canonical morphism $\Omega^*\otimes_{\laz}A^*(k)\row A^*$ which is an isomorphism when $X$ has dimension zero.
Therefore, it is enough to show that a morphism of theories of rational type $A'^*\row A^*$, inducing an isomorphism
for varieties of dimension zero, is necessarily an isomorphism. We argue by induction on the dimension of $X$.
It is enough to show that $\ov{A}'^*(X)\row\ov{A}^*(X)$ is an isomorphism. Consider the commutative diagram with exact rows:
\begin{equation}
\label{(3)}
\xymatrix @-0.7pc{
(\Da_{1,0}\oplus \Da_{0,1})(A'^*)\ar @{->}[r] \ar @{->}[d]^(0.5){(3)}&
\Da_{0,0}(A'^*) \ar @{->}[r] \ar @{->}[d]^(0.5){(2)}& \ov{A}'^*(X) \ar @{->}[r]
\ar @{->}[d]^(0.5){(1)} & 0\\
(\Da_{1,0}\oplus \Da_{0,1})(A^*)\ar @{->}[r] &
\Da_{0,0}(A^*) \ar @{->}[r] & \ov{A}^*(X) \ar @{->}[r] & 0
}
\end{equation}
Using the induction hypothesis, we see that $(2)$ is an isomorphism. This implies that $(1)$ is surjective.
In particular, we know that $A'^*(Y)\row A^*(Y)$ is surjective for all varieties $Y$ with $\ddim(Y)\leq\ddim(X)$.
This implies in turn that $(3)$ is surjective. Thus, $(1)$ is in fact an isomorphism as wanted.
\Qed
\end{proof}

To any theory $A^*$ one can assign a theory of rational type $(A^{(0)})^*$ defined as
$\Omega^*\otimes_{\laz}A^*(k)$ together with the canonical map of theories $g:(A^{(0)})^*\lrow A^*$.
In the case of a generically constant theory, we get:

\begin{proposition}
\label{A0}
For a theory $A^*$ satisfying $(CONST)$, the map $g:(A^{(0)})^*\twoheadrightarrow A^*$ is surjective.
\end{proposition}

\begin{proof}
Since $\Da_{0,0}(A^*)\row\ov{A}^*(X)$ is surjective for any theory satisfying $(EXCI)$, the surjectivity of
$g$ follows by induction on the dimension of $X$.
\Qed
\end{proof}

Finally, the following result shows that the set of theories of rational type is closed with respect to
reparametrization and multiplicative projectors (recall, that a {\it multiplicative projector} is a multiplicative
operation $\rho:A^*\row A^*$ such that $\rho\circ\rho=\rho$).

\begin{proposition}
\label{repmultproj}
Let $A^*$ be a theory of rational type. Then:
\begin{itemize}
\item[$1)$ ] For any $\gamma(x)=a_0x+a_1x^2+\ldots\in A[[x]]$ with invertible $a_0$, the
reparametrization $(A^*)^{\gamma}$ is a theory of rational type.
\item[$2)$ ] For any multiplicative projector $\rho:A^*\row A^*$, the theory $\rho A^*$ (with the
quotient structure) is a theory of rational type. The formal group law of this theory is
$(\ffi_{\rho}(A^*(k)),\ffi_{\rho}(F_A(x,y)))$.
\end{itemize}
\end{proposition}

\begin{proof}
1) Using Theorem \ref{PSLM}, we get a multiplicative operation
$\Omega^*\row \Omega^*\otimes_{\laz}A^*(k)=A^*$ corresponding to $\gamma$. Restricted to
$\op{Spec}(k)$ it gives the
formal group law $\laz\row A^{\gamma}$, where $A^{\gamma}=A^*(k)$ and $F_{A^{\gamma}}(x,y)=\gamma(F_A(\gamma^{-1}(x),\gamma^{-1}(y)))$.
Hence, our operation extends to a multiplicative operation
$$
G_{\gamma}:\Omega^*\otimes_{\laz}A^{\gamma}\row\Omega^*\otimes_{\laz}A^*(k).
$$
It is an isomorphism of pull-back structures since there is a multiplicative inverse $G_{\gamma^{-1}}$. Thus, $\Omega^*\otimes_{\laz}A^{\gamma}$ is just a reparametrization
of $\Omega^*\otimes_{\laz}A^*(k)$. By construction, the morphism of formal group laws corresponding to
$G_{\gamma}$ is $(id,\gamma)$. Hence, $\Omega^*\otimes_{\laz}A^{\gamma}=(A^*)^{\gamma}$
- see \cite{P2}, and the latter theory is of rational type.

2) Let $(\ffi_{\rho},\gamma_{\rho})$ be the respective morphism of formal group laws. From the
condition $\rho\circ\rho=\rho$ (and the fact that $\gamma$ is invertible) we get:
$\ffi_{\rho}(\gamma_{\rho})(x)=x$. We have an invertible multiplicative operation
$G_{\gamma_{\rho}}:(A^*)^{\gamma_{\rho}}\row A^*$ and a morphism
$\pi=id\otimes\ffi_{\rho}:A^*\row (A^*)^{\gamma}$ of free theories of Levine-Morel.
The respective morphisms of formal group laws are $FGL(\pi)=(\ffi_{\rho},x)$ and
$FGL(G_{\gamma_{\rho}})=(id,\gamma_{\rho})$. Since $FGL(G_{\gamma_{\rho}}\circ\pi)=FGL(\rho)$, from Theorem \ref{PSLM}
we obtain that $\rho=G_{\gamma_{\rho}}\circ\pi$. On the other hand, $FGL(\pi\circ G_{\gamma_{\rho}})=(\ffi_{\rho},x)$.
Hence, $\mu:=\pi\circ G_{\gamma_{\rho}}:(A^*)^{\gamma}\row (A^*)^{\gamma}$ preserves the structure of Chern classes, and by
\cite{PS} is a morphism of theories (commutes with pull-backs and push-forwards).
Moreover, since $G_{\gamma_{\rho}}$ is invertible, $\mu$ is a projector. Hence, $\rho A^*$ is realized
as a quotient of $A^*$ and as a direct summand of $(A^*)^{\gamma}$ under a projector endomorphism $\mu$. By 1), $(A^*)^{\gamma}$ is a theory of rational type, hence so is $\rho A^*$ by
Definition \ref{DRT}. Clearly,
$FGL(\rho A^*)=(\ffi_{\rho}(A^*(k)),\ffi_{\rho}(F_{A}(x,y)))$.
\Qed
\end{proof}

By the results of Levine-Morel (\cite[Theorems 1.2.18, 1.2.19]{LM}), Chow groups $CH^*$
and $K_0$ are {\it{free}} theories, and hence, theories of rational type.
It follows from Proposition \ref{repmultproj} that
other "standard small theories" such as $BP^*$ and higher Morava K-theories $K(n)$ are of rational type as well.\\

\smallskip

\subsection{The short bi-complex ${\frb}$.}
\label{b}

The short bi-complex ${\fra}(A^*)$ describes $A^*(X)$ in terms of $A^*$ of
smaller dimensional varieties and push-forward maps. But, to construct cohomological
operations we will need to find a presentation in terms of pull-backs. This will be done
in two steps.
First we modify ${\fra}$ by, roughly speaking, replacing the objects $v:V\row X$ of $\smu(X)$
by divisors with strict normal crossings contained in a blowup of $X$.
This is done in the present subsection where we introduce the short bi-complex ${\frb}$
and show that $H({\fra})=H({\frb})$. The second step is relatively easy and will be achieved in
Subsection \ref{c}.

We define a category $\rc(X)$ as follows.
Objects of $\rc(X)$ are diagrams $\cv=(Z\stackrel{z}{\row}X\stackrel{\rho}{\low}\wt{X})$, where
$z$ is an embedding of a closed proper subscheme, and $\rho$ is a projective birational morphism,
which is an isomorphism outside $Z$ and such that $V=\rho^{-1}(Z)$ is a divisor with strict
normal crossings.

Morphisms $(i,\pi)$ from $\cv_2$ to $\cv_1$ are commutative diagrams:
\begin{equation}
\label{cmor}
\xymatrix @-1.2pc{
Z_2 \ar @{->}[r]^(0.5){z_2} \ar @{->}[d]_(0.5){i} &X \ar @{=}[d] & \wt{X}_2 \ar @{->}[l]_(0.5){\rho_2}
\ar @{->}[d]^(0.5){\pi}\\
Z_1 \ar @{->}[r]_(0.5){z_1} & X & \wt{X}_1 \ar @{->}[l]^(0.5){\rho_1}.
}
\end{equation}
Among these, we distinguish two types of morphisms:
\begin{itemize}
\item[type I:]
 \ \ $i=id$, $\pi$ is a blowup, permitted with respect to $V_1=\rho_1^{-1}(Z_1)$ of a smooth closed subvariety
of $\wt{X}_1$ contained in $V_1$;
\item[type II:]
 \ \ $\pi=id$.
\end{itemize}
We denote by $\cmor_I$ and $\cmor_{II}$ the set of morphisms of type I and II respectively.
Note, that for morphisms of type I, $\pi^{-1}(V_1)=V_2$.

We also consider $\rc^1(X)$, the subcategory of $\rc(X\times\pp^1)$ whose objects are the diagrams
$$
\cw=(Z\stackrel{z}{\row}X\times\pp^1\stackrel{\rho}{\low}\wt{X\times\pp^1})
$$
satisfying the following additional condition: the inverse images $\wt{X}_0=\rho^{-1}(X\times 0)$ and
$\wt{X}_{\infty}=\rho^{-1}(X\times\infty)$ are smooth divisors on $\wt{X\times\pp^1}$ and, setting $W=\rho^{-1}(Z)$,
the scheme-theoretic intersection of $W\cap\wt{X}_0$ is a (non-necessarily reduced) divisor with
strict normal crossings in $\wt{X}_0$, and similarly for $\wt{X}_{\infty}$. Note that the last condition
is weaker than asking that the divisor $W\cup\wt{X}_0$ has strict normal crossings. At the same time, it implies that
for every component $W_s$ of $W$ the intersection $W_{s,0}=W_s\cap\wt{X}_0$ is a divisor with strict normal crossings in $W_s$,
and the same holds for $W_{s,\infty}$.

We have maps $\partial_0,\partial_{\infty}: Ob(\rc^1(X))\row Ob(\rc(X))$ defined by:
$$
\partial_l(Z\stackrel{z}{\row}X\times\pp^1\stackrel{\rho}{\low}\wt{X\times\pp^1})=
(Z_l\stackrel{z_l}{\row}X\stackrel{\rho}{\low}\wt{X}_l),
$$
where $Z_l=(X\times l)\cap Z$.

Consider the short bi-complex ${\frb}={\frb}(A^*)$:
\begin{itemize}
\item[$\cdot$ ]
$\Db_{0,0}:=\bigoplus\limits_{\cv\in Ob(\rc(X))} A_*(V)$; \hspace{7mm}
$\Db_{1,0}:=\bigoplus\limits_{\cv_2\row\cv_1\in\cmor_I\cup\cmor_{II}}A_*(V_2)$; \hspace{7mm}
$\Db_{0,1}:=\bigoplus\limits_{\cw\in Ob(\rc^1(X))}A_{*+1}(W)$
\end{itemize}
(we will also use notations $\Db_{1,0}^I$ and $\Db_{1,0}^{II}$ for the direct summands of $\Db_{1,0}$ corresponding to
$\cmor_I$ and $\cmor_{II}$)
and the differentials are defined as follows:
\begin{itemize}
\item[$\cdot$ ] $d_{1,0}((\pi,i):\cv_2\row\cv_1,y)=(\cv_1,(\pi_V)_*(y))-(\cv_2,y)$ where $\pi_V:V_2\row V_1$ is the map induced by $\pi$.
\item[$\cdot$ ]
$
d_{0,1}(\cw,\sum_s(h_s)_*(y_s))=
(\partial_0\cw,\sum_{s} (h_{s,0})_*i_{s,0}^{\star}(y_s))-
(\partial_{\infty}\cw,\sum_{s} (h_{s,\infty})_*i_{s,\infty}^{\star}(y_s))
$
where $h_s:W_s\row W$ are the inclusions of the irreducible components of $W$, $y_s\in A_{*+1}(W_s)$, and $i_{s,0}$
and $i_{s,\infty}$ are inclusions of the divisors $W_{s,0}$ and $W_{s,\infty}$ in $W_s$.
(The maps $i_{s,0}^{\star}$ and $i_{s,\infty}^{\star}$ are as in Definition \ref{starpullback}.)
\end{itemize}
Note that $d^{\frb}_{0,1}$ is well-defined. Indeed, we need to show that for any two components $W_s$ and $W_t$ of $W$, and any
element $x\in A_{*+1}(W_{\{s,t\}})$, we have: $d^{\frb}_{0,1}((h_s)_*(j_{s/t})_*(x))=d^{\frb}_{0,1}((h_t)_*(j_{t/s})_*(x))$
where $j_{s/t}:W_{\{s,t\}}\hookrightarrow W_s$ and $j_{t/s}:W_{\{s,t\}}\hookrightarrow W_t$ are the inclusions of the intersection
$W_{\{s,t\}}=W_s\cap W_t$. By Proposition \ref{DEFi} we know that the intersection $W_{\{s,t\},l}=
W_{\{s,t\}}\cap\wt{X}_l$ is a divisor with strict normal crossings in $W_{\{s,t\}}$. And the same is true about $W_s$ and $W_t$.
We get cartesian squares:
\begin{equation*}
\xymatrix @-0.7pc{
W_{\{s,t\},l} \ar @{->}[r]^(0.5){i_{\{s,t\},l}} \ar @{->}[d]_(0.5){j_{s/t,l}}&
W_{\{s,t\}} \ar @{->}[d]^(0.5){j_{s/t}}\\
W_{s,l} \ar @{->}[r]_(0.5){i_{s,l}} & W_s
}\hspace{1cm}
\xymatrix @-0.7pc{
W_{\{s,t\},l} \ar @{->}[r]^(0.5){i_{\{s,t\},l}} \ar @{->}[d]_(0.5){j_{t/s,l}}&
W_{\{s,t\}} \ar @{->}[d]^(0.5){j_{t/s}}\\
W_{t,l} \ar @{->}[r]_(0.5){i_{t,l}} & W_t
}
\end{equation*}
where the horizontal maps are inclusions of divisors with strict normal crossings.
Now it follows from Proposition \ref{MPEIF} that
$i_{s,l}^{\star}(j_{s/t})_*(x)=(j_{s/t,l})_*i_{\{s,t\},l}^{\star}(x)$ and
$i_{t,l}^{\star}(j_{t/s})_*(x)=(j_{t/s,l})_*i_{\{s,t\},l}^{\star}(x)$. And so, the respective elements of $\Db_{0,0}$
coincide. Hence, $d^{\frb}_{0,1}((h_s)_*(j_{s/t})_*(x))$ and $d^{\frb}_{0,1}((h_t)_*(j_{t/s})_*(x))$
are also equal.

Let us denote by $H(\frb)$ the $0$-th homology of the total complex $Tot(\frb)$ of $\frb$.
We have natural maps:
$$
\beta:\coker(d_{1,0}^{{\frb}})\row\coker(d_{1,0}^{{\fra}}),\hspace{3mm}\text{and}\hspace{3mm}
\hat{\beta}:H({\frb})\row H({\fra})
$$
sending the class of an element $(\cv,\sum_t (g_t)_*(x_t))\in\Db_{0,0}$ to the class of $\sum_t(v_t,x_t)\in\Da_{0,0}$,
where $g_t:V_t\row V$ are the inclusions of the irreducible components of $V$, and $v_t=v\circ g_t$, where
$v:V\row X$ is the natural projection
(the fact that $\beta$ and $\hat{\beta}$ are well-defined follows from
the existence of similar maps $b_{1,0}\row a_{1,0}$ and $b_{0,1}\row a_{0,1}$ and the observation that the choices of a
decomposition of $x$ into a sum $\sum_t (g_t)_*(x_t)$ are eliminated in $\coker(d^{{\fra}}_{1,0})$).

Let $\cv=(Z\stackrel{z}{\row}X\stackrel{\rho}{\low}\wt{X})\in Ob(\rc(X))$, and $V=\rho^{-1}(Z)$.
Denote by $im_Z$ (or $im_Z(\cv)$, but see Lemma \ref{p1} below) the image of the map
$$
A_*(V)\row \Db_{0,0}/(\,\text{part $I$ of}\, d^{{\frb}}_{1,0}).
$$
We have a well-defined map: $pr_Z: im_Z\row A_*(Z)$ induced by $(\rho_V)_*$ where $\rho_V:V\row Z$ is the restriction of $\rho$.
Below an element of the form $(\cv,x)\in\Db_{0,0}$ is said to be "defined over $Z$".

\begin{lemma}
\label{p1}
\phantom{a}\hspace{5mm}\\
\vspace{-5mm}
\begin{itemize}
\item[$(1)$ ] The subgroup $im_Z$ depends only on $Z$ (and not on the choice of $\rho$ in $\cv$).
\item[$(2)$ ] $im_Z=im_{Z_{red}}$.
\end{itemize}
\end{lemma}

\begin{proof}
Let $\cv_1=(Z\stackrel{z}{\row}X\stackrel{\rho_1}{\low}\wt{X}_1)$ and $\cv_2=(Z\stackrel{z}{\row}X\stackrel{\rho_2}{\low}\wt{X}_2)$
be two objects of $\rc(X)$.
We have a birational map $\rho_2^{-1}\circ\rho_1:\wt{X}_1\dashrightarrow \wt{X}_2$ which is an isomorphism
outside $V_1$ and $V_2$. By the Weak Factorization Theorem (see Theorem \ref{WF}(6)), there exists
a diagram
$$
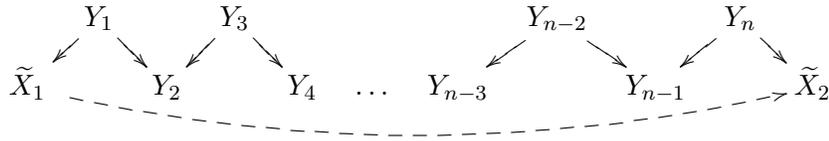
\xymatrix @=3mm{
& Y_1 \ar @{->}[ld] \ar @{->}[rd] && Y_3 \ar @{->}[ld] \ar @{->}[rd] &&
 && Y_{n-2} \ar @{->}[ld] \ar @{->}[rd] && Y_n \ar @{->}[ld] \ar @{->}[rd]  &\\
\wt{X}_1 \ar @/_1.5pc/ @{-->}[rrrrrrrrrr]_(0.5){}&& Y_2 && Y_4 &\ldots  & Y_{n-3}  && Y_{n-1}&& \wt{X}_2
}
$$
of smooth projective varieties over $X$ where each map is a blowup of a smooth center over $Z$ permitted
with respect the inverse image of $Z$.
Thus, to show that $\cv_1$ and $\cv_2$ define the same subgroup $im_Z$, we may assume that
$\wt{X}_1$ and $\wt{X}_2$
are related by a blowup of a smooth center over $Z$ permitted with respect the inverse image of $Z$. Said differently,
we may assume that there exists a map $(id,\pi):\cv_2\row\cv_1$ of type $I$. This clearly implies that
$im_Z(\cv_2)\subset im_Z(\cv_1)$. The reverse inclusion follows from the fact that $A_*(V_2)\row A_*(V_1)$ is surjective.
This proves $(1)$.

To prove (2) observe that we can find
$\cv_1=(Z_{red}\stackrel{z_{red}}{\row}X\stackrel{\rho_1}{\low}\wt{X}_1)$ and
$\cv_2=(Z\stackrel{z}{\row}X\stackrel{\rho_2}{\low}\wt{X}_2)$ objects of $\rc(X)$,
where $\rho_2=\rho_1\circ\pi$, for some $\pi: \wt{X}_2\row\wt{X}_1$ which is a blowup with
smooth centers over $Z_{red}$ permitted with respect the inverse image of $Z_{red}$.
What we need now follows from the surjectivity of the map $A_{*}(V_2)\row A_{*}(V_1)$.
\Qed
\end{proof}

\begin{lemma}
\label{lTY}
Let $T\subset X$ be a closed subscheme, and $Y\subset X$ a divisor such that $Y\backslash T$ is smooth.
Let $\ov{x}\in im_T$. Then there exists $\ov{x}'\in im_{T\cup Y}$ such that $\ov{x}$ and $\ov{x}'$
have the same image in $\coker(d_{1,0}^{{\frb}})$, and $j_*(pr_T(\ov{x}))=pr_{T\cup Y}(\ov{x'})$, where
$j:T\hookrightarrow T\cup Y$ is an obvious inclusion.
\end{lemma}

\begin{proof}
Since $Y\backslash T$ is a smooth divisor of $X\backslash T$, we may find $\cv=(T\row X\stackrel{\rho}{\low}\wt{X})$
in $\rc(X)$ such that $\rho^{-1}(T\cup Y)$ is a divisor with strict normal crossings in $\wt{X}$.
Thus, we have also the object $\cv'=(T\cup Y\row X\stackrel{\rho}{\low}\wt{X})$ in $\rc(X)$ and a morphism $\cv'\row\cv$
of type II. As usual, set $V=\rho^{-1}(T)$ and $V'=\rho^{-1}(T\cup Y)$. By Lemma \ref{p1}, we may assume that
$\ov{x}$ is the class of $(\cv,x)$ for some $x\in A_*(V)$. Set $x'=(j_V)_*(x)\in A_*(V')$ where $j_V:V\row V'$
is the obvious inclusion. Clearly, the class $\ov{x}'$ of $(\cv',x')$ in $im_{T\cup Y}$ satisfies the claimed properties.
\Qed
\end{proof}

\begin{lemma}
\label{lSing}
Let $Z\subset X$ be a proper closed subscheme. Then there exist divisors $Y_i,\,i=1,\ldots,m$
such that $Z_{red}\subset \bigcup_i Y_i$, and $Y_j$ is smooth outside $\bigcup_{i=1}^{j-1}Y_i$.
\end{lemma}

\begin{proof}
Use Noetherian induction. The base ($Z=\emptyset$) is trivial. Suppose, we know the
statement for all proper closed subschemes of $Z_{red}$. By Proposition \ref{Bd}, there exist
a divisor $Y$ of $X$ which contains $Z_{red}$, and is smooth outside $Z$, and in the generic points
of the components of $Z$. Thus, the locus of singular points $S$ of $Y$ is a proper subscheme
of $Z_{red}$.
By induction, there exist divisors $Y_i,\,i=1,\ldots,m-1$, such that
$S_{red}\subset \bigcup_{i=1}^{m-1} Y_i$, and $Y_j$ is smooth outside $\bigcup_{i=1}^{j-1}Y_i$.
Taking $Y_m=Y$, we get the needed sequence of divisors for $Z$.
\Qed
\end{proof}

\begin{lemma}
\label{p2}
Let $\cv=(Z\row X\stackrel{\rho}{\low}\wt{X})\in Ob(\rc(X))$ with $V=\rho^{-1}(Z)$, and
$x\in A_*(V)$ be an element whose image in $\operatornamewithlimits{colim}_{T\subsetneq X}A_*(T)$ is zero.
Then there exists $Z'\supset Z$ and
$\cv'=(Z'\row X\stackrel{\rho'}{\low}\wt{X}')\in Ob(\rc(X))$ with $V'=\rho^{-1}(Z')$, and
$x'\in A_*(V')$ such that $(\cv',x')$ and $(\cv,x)$ have the same image in $\coker(d^{{\frb}}_{1,0})$ and
$pr_{Z'}(x')=0$.
\end{lemma}

\begin{proof}
Let $T\subsetneq X$ be a closed subscheme, containing $Z$ and such that $(j_T)_*(pr_Z(x))=0$.
By Lemma \ref{lSing}, we may find divisors $Y_1,\ldots, Y_m$ such that $T_{red}\subset\bigcup_{i=1}^mY_i$
and $Y_j$ is smooth outside $\bigcup_{i=1}^{j-1}Y_i$ for all $j$. By Lemma \ref{lTY} we may find elements
$\ov{x}_j\in im_{T\cup\bigcup_{i=1}^{j}Y_i}$ having the same class as $\ov{x}=[(\cv,x)]$ in $\coker(d^{{\frb}}_{1,0})$
and such that $pr_{T\cup\bigcup_{i=1}^{j}Y_i}(\ov{x}_j)$ is the push-forward along
$T\cup\bigcup_{i=1}^{j-1}Y_i\hookrightarrow T\cup\bigcup_{i=1}^{j}Y_i$ of $pr_{T\cup\bigcup_{i=1}^{j-1}Y_i}(\ov{x}_{j-1})$.
Setting $Z'=\bigcup_{i=1}^{m}Y_i$ and choosing a representative $(\cv',x')$ of $\ov{x}_m$ give what we want.
Indeed, $pr_{Z'}(x')$ is the push-forward of $pr_Z(x)$ along $Z\hookrightarrow Z'$ which is zero since $T_{red}\subset Z'$.
\Qed
\end{proof}

\begin{lemma}
\label{p3}
Let $u\in \Db_{0,0}$ be an element whose image in
$\operatornamewithlimits{colim}_{T\subsetneq X}A_*(T)$ is zero.
Then $u\in image(d_{1,0}^{{\frb}})$.
\end{lemma}

\begin{proof}
Write $u=(\cv_1,x_1)+\ldots+(\cv_n,x_n)$ and denote by $Z_l$ the closed subvariety that appears in $\cv_l$ for $l=1,\ldots,n$.
By Lemma \ref{lSing}, applied to the subscheme $T=Z_1\cup\ldots\cup Z_n$, there exist divisors $Y_1,\ldots,Y_m$ such that
$T_{red}\subset Y=\bigcup_{i=1}^mY_i$ and $Y_j$ is smooth outside $\bigcup_{i=1}^{j-1}Y_i$.
Applying Lemma \ref{lTY} to each of the inclusions
$$
Z_l\subset Z_l\cup Y_1\subset\ldots\subset Z_l\cup Y_1\cup\ldots\cup Y_m=Z_l\cup Y,
$$
and using the fact that $(Z_l\cup Y)_{red}=Y_{red}$,
we obtain an element $\ov{x}_l\in im_{Y}$ having the same image as $(\cv_l,x_l)$ in
$\coker(d^{{\frb}}_{1,0})$. This shows that the class of $u$ in $\coker(d^{{\frb}}_{1,0})$ can be represented
by a single element defined over $Y$. In other words, it is enough to treat the case
when $u=(\cv,x)$ with $\cv=(Z\row X\stackrel{\rho}{\low}\wt{X})$. Furthermore, applying Lemma \ref{p2},
we may assume that $pr_Z(x)=0$.

We now treat the case of $u=(\cv,x)$ with $pr_Z(x)=0$, by noetherian induction on $Z$. Recall that we want to show that
$u\in\op{Image}(d^{{\frb}}_{1,0})$. (When $Z$ is empty, $x$ is zero and there is nothing to prove.)
By Proposition \ref{Bd}, there exists a divisor $Y$ containing
$Z$ and smooth outside some proper closed subset $S\subsetneq Z$.
Fix $\cv'=(Y\row X\stackrel{\rho'}{\low}\wt{X}')$ in $\rc(X)$ where $\rho'$ is
a sequence of blowups in smooth centers over $S$. By Lemmas \ref{p1} and \ref{lTY},
we may find $x'\in A_*(V')$ such that $(\cv,x)$ and $(\cv',x')$ have the same image in $\coker(d^{{\frb}}_{1,0})$,
and $pr_Y(x')=0$.

Denote by $\wt{Y}$ the strict transform of $Y$ along the blowup $\rho':\wt{X}'\row X$ and by $E$ the exceptional divisor of $\rho'$.
Since $x'$ is supported on $\wt{Y}\cup E$, we can write $x'=\ov{y}-\ov{e}$ where $\ov{y}$ is the push-forward of $y\in A_*(\wt{Y})$
and $\ov{e}$ the push-forward of $e\in A_*(E)$. Since $pr_Y(x')=0$, we see that $pr_Y(\ov{y})=pr_Y(\ov{e})$. Thus, if $s\in A_*(S)$
denotes the push-forward of $e\in A_*(E)$, we see that $s\in A_*(S)$ and $y\in A_*(\wt{Y})$ have the same push-forward in $A_*(Y)$.
Now by Lemma \ref{Abm1}, we have an exact sequence
$$
0\low A_*(Y)\low A_*(\wt{Y})\oplus A_*(S)\low A_*(\wt{Y}\cap E).
$$
Thus, we may find $c\in A_*(\wt{Y}\cap E)$ which maps to $s$ and $y$.

Now, by the previous discussion, $\ov{y}\in A_*(\wt{Y}\cup E)$ is the push-forward of $c$ along the inclusion
$\wt{Y}\cap E\hookrightarrow\wt{Y}\cup E$. Hence, letting $d\in A_*(E)$ be the push-forward of $c$ along the inclusion
$\wt{Y}\cap E\hookrightarrow E$, we see that $\ov{y}\in A_*(\wt{Y}\cup E)$ is the push-forward of $d\in A_*(E)$ by the inclusion
$E\hookrightarrow\wt{Y}\cup E$. This shows that $x'\in A_*(\wt{Y}\cup E)$ is the push-forward of $d-e\in A_*(E)$.
Moreover, both $d$ and $e$ map to $s\in A_*(S)$ showing that the push-forward of $d-e\in A_*(E)$ along the projection
$E\row S$ is zero. To conclude, consider the object $\cv''=(S\row X\stackrel{\rho'}{\low}\wt{X}')$ of $\rc(X)$.
Clearly, the classes of $(\cv'',x''=d-e)$, $(\cv',x')$ and $(\cv,x)$ are equal in $\coker(d^{{\frb}}_{1,0})$.
Moreover, $pr_S(x'')=0$. We may now use the inductive hypothesis to conclude.
\Qed
\end{proof}

\begin{corollary}
\label{c'bCoker}
Let $A^*$ be any theory in the sense of Definition \ref{goct}. Then
the natural map
$$
\coker(d_{1,0}^{{\frb}})\row\operatornamewithlimits{colim}_{T\subsetneq X}A_*(T)
$$
is an isomorphism.
\end{corollary}

\begin{proof}
By Lemma \ref{p3}, we know that the map is injective. To prove surjectivity, it is enough to show that for
$\cv=(Z\row X\stackrel{\rho}{\low}\wt{X})$ in $\rc(X)$, the map $A_*(V)\row A_*(Z)$ is surjective.
This follows easily from the fact that the projective morphism $V\row Z$ is surjective and its fibers
are unions of rational varieties.
\Qed
\end{proof}

\begin{proposition}
\label{NbOm}
Assume that $A^*=\Omega^*$. Then, the natural map $H({\frb})\row\ov{\Omega}^*(X)$ is an isomorphism.
\end{proposition}

\begin{proof}
In short, the result follows from the fact that the map $\op{div}$ in the diagram (\ref{BD}) factors not only through
$\Da_{0,1}$, but through
$\Db_{0,1}$ as well. In more details, we use Levine's exact sequence
$$
\zz[k(X)^{\times}]\otimes\laz_{*-d+1}\stackrel{\op{div}^{\Omega}}{\lrow}\Omega_*^{(1)}(X)\row\Omega^*(X)\row\laz\row 0.
$$
From Corollary \ref{c'bCoker}, we know that the natural map
$$
\coker(d_{1,0}^{{\frb}})\row\Omega_*^{(1)}(X)
$$
is an isomorphism. Now, consider the commutative diagram with exact rows
\begin{equation*}
\xymatrix{
b_{0,1} \ar @{->}[r]^(0.5){\hat{d}^{{\frb}}_{0,1}} & \coker(d_{1,0}^{{\frb}}) \ar @{->}[r] \ar @{->}[d]^{\cong} &
H(\frb) \ar @{->}[r] \ar @{->}[d] & 0\\
\zz[k(X)^{\times}]\otimes\laz \ar @{-->}[u] \ar @{->}[r]^(0.5){\op{div}^{\Omega}} & \Omega_*^{(1)}(X) \ar @{->}[r] &
\ov{\Omega}_*(X) \ar @{->}[r] & 0.
}
\end{equation*}
It is enough to show that there is a dashed arrow as above making the left square commutative.
The needed map $\zz[k(X)^{\times}]\otimes\laz\row b_{0,1}$ is $\laz$-linear and its value on a rational function
$f\in k(X)^{\times}$ is constructed as follows.

Let $\pi:\wt{X}\row X$ be a sequence of blowups in smooth centers such that $f$ extends to a morphism $\wt{f}:\wt{X}\row\pp^1$.
Let $Z\subsetneq X$ be a closed subset such that $\pi$ is an isomorphism outside $Z$. Blowing up further, we may assume that
$\pi^{-1}(Z)\cup\wt{f}^{-1}(0)\cup\wt{f}^{-1}(\infty)$ is a divisor with strict normal crossings. Set $E=\pi^{-1}(Z)$.

Now consider $\cw=(T\row X\times\pp^1\stackrel{\pi\times id}{\llow}\wt{X}\times\pp^1)$, where
$T=\ov{\Gamma_f}\cup (Z\times\pp^1)$. We claim that
$\cw$ is an object of $\rc^1(X)$. Clearly, the fibers at $0$ and $\infty$ of the projection $\wt{X}\times\pp^1\row\pp^1$
are smooth. Moreover, the inverse image of $T$ is exactly $W=\Gamma_{\wt{f}}\cup (E\times\pp^1)$ which is a divisor with strict
normal crossings. (Indeed, note that the $\Gamma_{\wt{f}}\cap(E\times\pp^1)$ is isomorphic to $E$.)
Finally, $W_0=\wt{f}^{-1}(0)\cup E$ and $W_{\infty}=\wt{f}^{-1}(\infty)\cup E$ are divisors with strict normal crossings on $\wt{X}$.

Now it is easy to conclude: the image of $f\in k(X)^{\times}$ by the map $\zz[k(X)^{\times}]\row b_{0,1}$ is given by
$(\cw,g_*(1))$ with $g:\Gamma_{\wt{f}}\hookrightarrow W$ the obvious inclusion.
\Qed
\end{proof}

\begin{corollary}
\label{NbFree}
If $A^*$ is a free theory in the sense of Levine-Morel, then the natural map $H({\frb})\row\ov{A}_*(X)$ is an isomorphism.
\end{corollary}

\begin{remark}
One could have introduced the theories of rational type using the short bi-complex ${\frb}$ instead of ${\fra}$.
This would have permitted to bi-pass the Subsection \ref{a}, substituting the crucial Proposition \ref{CobRT} by
the Proposition \ref{NbOm}. We though find the bi-complex ${\fra}$ much more natural
than the bi-complex ${\frb}$. Another point the author wanted to make is that all these bi-complexes represent useful tools for
studying cohomology theories, providing alternative descriptions of a theory.
\end{remark}

Beyond Corollary \ref{NbFree}, we have more generally:

\begin{proposition}
\label{c'aH}
For any theory $A^*$ in the sense of Definition \ref{goct} satisfying $(CONST)$,
the natural map
$$
\hat{\beta}:H({\frb})\row H({\fra})
$$
is an isomorphism.
\end{proposition}

We will not use this result below, so we omit the proof.

\subsection{The short bi-complex ${\frc}$.}
\label{c}

Now we are ready to give the description of $A^*$ in terms of pull-backs.
In this subsection we will assume that $A^*$ is a {\it theory of rational type}.
By Proposition \ref{KlassRT}, this means that $A^*$ is {\it free} in the sense
of M.Levine-F.Morel, that is $A^*=\Omega^*\otimes_\laz A^*(k)$.
In particular, we have at our disposal the refined pull-backs defined in \cite[Subsection 6.6]{LM}.
That is, given a cartesian square
$$
\begin{CD}
W & @>>> & Y\\
@VVV & & @VVV \\
Z & @>{f}>> & X
\end{CD}
$$
where $f$ is an l.c.i. morphism of relative codimension $d$, we have a morphism $f^!:A_*(Y)\row A_{*-d}(W)$
satisfying a number of properties (see \cite[Theorem 6.6.6]{LM}).

Consider the short bi-complex ${\frc}={\frc}(A^*)$:
\begin{itemize}
\item[$\cdot$ ]
$\Dc_{0,0}:=\bigoplus\limits_{\cv\in Ob(\rc(X))} \op{Image}(\rho^!:A_*(Z)\row A_*(V))$;
\item[$\cdot$ ]
$\Dc^I_{1,0}:=\bigoplus\limits_{\cv_2\row\cv_1\in\cmor_I}\op{Image}(\rho_1^!:A_*(Z_1)\row A_*(V_1))$ - see (\ref{cmor})
in Subsection \ref{b};
\item[$\cdot$ ]
$\Dc^{II}_{1,0}:=\bigoplus\limits_{\cv_2\row\cv_1\in\cmor_{II}}\op{Image}(\rho_2^!:A_*(Z_2)\row A_*(V_2))$, \hspace{1cm}
$\Dc_{1,0}=\Dc_{1,0}^I\oplus\Dc_{1,0}^{II}$;
\item[$\cdot$ ]
$\Dc_{0,1}:=\bigoplus\limits_{\cw\in Ob(\rc^1(X))}\op{Image}(\rho^!:A_{*+1}(Z)\row A_{*+1}(W))$.
\end{itemize}
and the differentials are defined as follows:
\begin{itemize}
\item[$\cdot$ ] $d^I_{1,0}((id,\pi):\cv_2\row\cv_1,x)=(\cv_1,x)-(\cv_2,\pi^{!}(x))$
where $\pi^!:A_*(V_1)\row A_*(V_2)$ is the refined pull-back relative to $\pi:\wt{X}_2\row\wt{X}_1$.
\item[$\cdot$ ]
$d^{II}_{1,0}((i,id):\cv_2\row\cv_1,y)=(\cv_1,(i_V)_*(y))-(\cv_2,y)$ where $i_V:V_2\row V_1$ is the obvious inclusion.
Thus, $d^{II}_{1,0}:\Dc^{II}_{0,1}\row\Dc_{0,0}$ is simply the restriction of $d^{II}_{1,0}:\Db^{II}_{0,1}\row\Db_{0,0}$.
\item[$\cdot$ ]
$d_{0,1}:\Dc_{0,1}\row \Dc_{0,0}$ is the restriction of $d_{0,1}:\Db_{0,1}\row \Db_{0,0}$.
\end{itemize}

\begin{lemma}
\label{c-well-deff}
The above differentials are well-defined.
\end{lemma}

\begin{proof}
The fact that $d^I_{1,0}$ takes $\op{Image}(\rho_1^!)$ to $\op{Image}(\rho_2^!)$ follows from the formula
$\rho_2^!=\pi^!\circ\rho_1^!$ (see \cite[Theorem 6.6.6(3)]{LM}). The fact that $d^{II}_{1,0}$ takes $\op{Image}(\rho_2^!)$ to
$\op{Image}(\rho_1^!)$ follows from the formula $(i_V)_*\rho_2^!=\rho_1^!(i_Z)_*$ where $i_Z:Z_1\hookrightarrow Z_2$ is an
obvious inclusion (see \cite[Theorem 6.6.6(2.a)]{LM}).

We now turn to the case of the differential $d_{0,1}$. By Lemma \ref{star!-e-only}, we have that
$d_{0,1}(\cw,y)=(\partial_0\cw,k_0^!(y))-(\partial_{\infty}\cw,k_{\infty}^!(y))$ where $k_0^!:A_{*+1}(W)\row A_*(W_0)$ and
$k_{\infty}^!:A_{*+1}(W)\row A_*(W_{\infty})$ are the refined pull-backs associated to the regular immersions
$k_l:X\times\{l\}\hookrightarrow X\times\pp^1$, for $l=0,\infty$.

For $l=0,\infty$ we have the
cartesian diagram (with $Z_l=Z\cap X\times\{l\}$):
\begin{equation}
\label{diag-c}
\xymatrix @-0.7pc{
Z_l \ar @{->}[d]_(0.5){z_l} & W_l\ar @{->}[r]^(0.5){j'_l} \ar @{->}[d]_(0.5){w_l}
\ar @{->}[l]_(0.4){\rho'_l} &W \ar @{->}[r]^(0.5){\rho'}
\ar @{->}[d]_(0.5){w}&Z \ar @{->}[d]_(0.5){z}\\
X\times\{l\}&\wt{X}_l\ar @{->}[l]^(0.4){\rho_l}\ar @{->}[r]_(0.4){j_l}&\wt{X\times\pp^1} \ar @{->}[r]_(0.5){\rho}&X\times\pp^1,
}
\hspace{5mm}
\xymatrix @-0.7pc{
\wt{X}_l\ar @{->}[d]_(0.4){\rho_l}\ar @{->}[r]^(0.4){j_l}&\wt{X\times\pp^1} \ar @{->}[d]^(0.4){\rho}\\
X\times\{l\}\ar @{->}[r]_(0.5){k_l} &X\times\pp^1,
}
\end{equation}
And since $\wt{X}_0$, $\wt{X}_{\infty}$ are smooth divisors on $\wt{X\times\pp^1}$, the map $\rho$
is transversal to the immersions $k_l$,
$l=0,\infty$. This implies that $(k_l)^{!}=(j_l)^{!}:A_{*+1}(W)\row A_*(W_l)$ (see \cite[Lemma 6.6.2]{LM}).
The result follows again from the functoriality of refined pull-backs
(see \cite[Theorem 6.6.6(3)]{LM}).
\Qed
\end{proof}

We define a morphism of bi-complexes
$$
\mu:{\frb}\row{\frc}.
$$
as follows
\begin{itemize}
\item[$\cdot$ ] $\mu_{0,0}:\Db_{0,0}\row\Dc_{0,0}$ is given by $\mu_{0,0}(\cv,x)=(\cv,\rho^!(\rho_V)_*(x))$ where
$\rho_V:V\row Z$ is the obvious projection.
\item[$\cdot$ ] $\mu^I_{1,0}:\Db^I_{1,0}\row \Dc^I_{1,0}$ is given by
$\mu^I_{1,0}(\cv_2\row\cv_1,y)=(\cv_2\row\cv_1,\rho_1^!(\rho_{V_2})_*(y))$.
\item[$\cdot$ ] $\mu^{II}_{1,0}:\Db^{II}_{1,0}\row\Dc^{II}_{1,0}$ is given by
$\mu^{II}_{1,0}(\cv_2\row\cv_1,y)=(\cv_2\row\cv_1,\rho_2^!(\rho_{V_2})_*(y))$.
\item[$\cdot$ ] $\mu_{0,1}:\Db_{0,1}\row\Dc_{0,1}$ is given by $\mu_{0,1}(\cw,y)=(\cw,\rho^!(\rho_W)_*(y))$.
\end{itemize}
Let's check that these morphisms commute with the differentials. The cases of $d^{II}_{1,0}$ and $d_{0,1}$ follow from
\cite[Theorem 6.6.6]{LM}, where in the latter one we use the fact that $\rho$ is transversal to the immersions $k_l$, $l=0,\infty$,
and so $(k_l)^{!}=(j_l)^{!}:A_{*+1}(W)\row A_*(W_l)$ (notations from (\ref{diag-c})).
For $d^I_{1,0}$, we compute:
\begin{equation*}
\begin{split}
\mu_{0,0}\circ d^I_{1,0}(\cv_2\row\cv_1,y)=&\mu_{0,0}((\cv_1,(\pi_V)_*(y))-(\cv_2,y))
=(\cv_1,\rho_1^!(\rho_{V_2})_*(y))-(\cv_2,\rho_2^!(\rho_{V_2})_*(y)).
\end{split}
\end{equation*}
On the other hand, we have
\begin{equation*}
\begin{split}
d^I_{1,0}\circ\mu_{0,0}(\cv_2\row\cv_1,y)=&d^I_{1,0}(\cv_2\row\cv_1,\rho_1^!(\rho_{V_2})_*(y))
=(\cv_1,\rho_1^!(\rho_{V_2})_*(y))-(\cv_2,\pi^!\rho_1^!(\rho_{V_2})_*(y)).
\end{split}
\end{equation*}
The result follows from the equality $\pi^!\rho_1^!=\rho_2^!$ (\cite[Theorem 6.6.6(3)]{LM}).

\begin{lemma}
\label{c-lN-odin}
The map $\mu_{0,0}:b_{0,0}\row c_{0,0}$ is surjective. Thus, the induced map $H({\frb})\row H({\frc})$ is also surjective.
\end{lemma}

\begin{proof}
Given $\cv$ in $\rc(X)$, we need to show that the map $\rho^!(\rho_V)_*:A_*(V)\row\op{Image}(\rho^!)$ is surjective.
It is enough to show that the map $(\rho_V)_*:A_*(V)\row A_*(Z)$ is surjective which follows from the fact that $V\row Z$
is surjective and has fibers which are unions of rational varieties.
\Qed
\end{proof}

We now consider the map $\gamma_{0,0}:c_{0,0}\row\ov{A}^*(X)$ given by
$$
\gamma_{0,0}(\cv,x)=\frac{\rho_*(i_V)_*(x)}{\rho_*(1)}
$$
where, as usual, $i_V:V\row\wt{X}$ is the obvious inclusion (the denominator is invertible by Proposition \ref{pi1invert}). 
This map descends to a map
$$
\gamma: H({\frc})\row\ov{A}^*(X).
$$
(We only treat the case of $d^I_{1,0}$ which is the most interesting one. The map $\gamma_{0,0}\circ d^I_{1,0}$ sends
$(\cv_2\row\cv_1,x)$ to
$$
\gamma_{0,0}(\cv_2,\pi^!(x))-\gamma_{0,0}(\cv_1,x)=\frac{(\rho_2)_*(i_{V_2})_*\pi^!(x)}{(\rho_2)_*(1)}-
\frac{(\rho_1)_*(i_{V_1})_*(x)}{(\rho_1)_*(1)}.
$$
Recall that $x=\rho_1^!(s)$ where $s\in A_*(Z_1)$. It follows that $(i_{V_1})_*(x)=\rho_1^*(i_{Z_1})_*(s)$. Thus, the second
fraction in the formula simplifies as follows:
$$
\frac{(\rho_1)_*(i_{V_1})_*(x)}{(\rho_1)_*(1)}=\frac{(\rho_1)_*\rho_1^*(i_{Z_1})_*(s)}{(\rho_1)_*(1)}=(i_{Z_1})_*(s)
$$
where the last equality follows from the projection formula.

Similarly, $(i_{V_2})_*\pi^!(x)=\rho_2^*(i_{Z_1})_*(s)$ and we have, for the same reasons,
$$
\frac{(\rho_2)_*(i_{V_2})_*\pi^!(x)}{(\rho_2)_*(1)}=(i_{Z_1})_*(s).
$$
This proves that $\gamma_{0,0}\circ d^I_{1,0}=0$ as needed.)

\begin{lemma}
\label{c-lN-dva}
The following triangle is commutative
$$
\xymatrix{
H({\frb}) \ar @{->}[r]^(0.5){\mu} \ar @{->}[rd] & H({\frc}) \ar @{->}[d]^(0.5){\gamma} \\
& \ov{A}^*(X).
}
$$
\end{lemma}

\begin{proof}
The composition $\gamma\circ\mu$ sends the class $(\cv,x)$ to
$$
\gamma_{0,0}(\cv,\rho^!(\rho_V)_*(x))=\frac{\rho_*\rho^!(\rho_V)_*(x)}{\rho_*(1)}=\frac{\rho_*(1)\cdot (i_Z)_*(\rho_V)_*(x)}{\rho_*(1)}=
(i_Z)_*(\rho_V)_*(x)
$$
as needed.
\Qed
\end{proof}

\begin{theorem}
\label{Ac}
Let $A^*$ be a theory of rational type. Then the map $\gamma: H({\frc})\row\ov{A}^*(X)$ is an isomorphism.
\end{theorem}

\begin{proof}
Use the Lemma \ref{c-lN-dva}, the fact that $\mu:H({\frb})\row H({\frc})$ is surjective and that the composition
$\gamma\circ\mu:H({\frb})\row\ov{A}^*(X)$ is an isomorphism by Corollary \ref{NbFree}.
\Qed
\end{proof}

\section{From products of projective spaces to $\smk$}

In this section, $A^*$ is a {\it theory of rational type}, and
$B^*$ is any theory in the sense of Definition \ref{goct}.
Our aim here is to prove the main result of the article:

\begin{theorem}
\label{MAIN}
Let $A^*$ be a {\it theory of rational type}, and
$B^*$ be any theory in the sense of Definition \ref{goct}.
Fix $n,m\in\zz$. Then any family of homomorphisms
$$
A^n((\pp^{\infty})^{\times l})\stackrel{G}{\row}B^m((\pp^{\infty})^{\times l}),\,\,\text{for}\,\,
l\in\zz_{\geq 0}
$$
commuting with the pull-backs for:
\begin{itemize}
\item[$(i)$ ] the action of ${\frak{S}}_l$;
\item[$(ii)$ ] the partial diagonals;
\item[$(iii)$ ] the partial Segre embeddings;
\item[$(iv)$ ] $(\op{Spec}(k)\hookrightarrow\pp^{\infty})\times
(\pp^{\infty})^{\times r},\,\,\forall r$;
\item[$(v)$ ] the partial projections
\end{itemize}
extends to a unique additive operation $A^n\stackrel{G}{\row}B^m$ on $\smk$.
\end{theorem}

\begin{remark}
\begin{itemize}
\item[$1)$ ]
The condition on $A^*$ is necessary. For example, the identity maps
$$
\op{CH}^*_{alg}((\pp^{\infty})^{\times l})\stackrel{id}{\row}
\op{CH}^*((\pp^{\infty})^{\times l})
$$
can not be extended to a morphism of theories.
\item[$2)$ ] In Topology an analogous result was obtained by T.Kashiwabara -
see {\rm \cite[Theorem 4.2]{Kash}}.
\end{itemize}
\end{remark}

\begin{proof}

Let $X$ be a smooth quasi-projective variety. Our goal is to construct, by induction on the dimension of $X$, a family of
homomorphisms
$$
G: A^n(X\times(\pp^{\infty})^{\times l})\row B^m(X\times(\pp^{\infty})^{\times l}),\hspace{5mm}\text{for all}\,\,l\in\zz_{\geq 0}
$$
satisfying conditions $(i)$-$(v)$ as in Theorem \ref{MAIN}. Set $z^A_i=c^A_1(O(1)_i)\in A^1((\pp^{\infty})^{\times l})$ and similarly
for $B$. (As usual, $O(1)_i$ is the pull-back of $O_{\pp^{\infty}}(1)$ along the $i$-th projection.)
A family as above satisfying $(ii)$ and $(v)$, i.e. commuting with pull-backs for the partial diagonals and partial projections,
is uniquely determined by its action on the elements of the form $\alpha\cdot(\prod_{i=1}^lz^A_i)$ with $\alpha\in A^{n-l}(X)$.
Moreover, writing
$$
G(\alpha\cdot(\prod_{i=1}^lz^A_i))=\gl{l}(\alpha)(z^B_1,\ldots,z^B_l)\in B^*(X)[[z^B_1,\ldots,z^B_l]]=
B^*(X\times(\pp^{\infty})^{\times l}),
$$
we see that we need to construct power series $\gl{l}(\alpha)$ in $l$ variables with coefficients in $B^*(X)$.
The remaining conditions $(i)$, $(iii)$ and $(iv)$ impose conditions on the power series $\gl{l}(\alpha)$.
These are respectively the conditions $(a_i)$, $(a_{iii})$ and $(a_{ii})$ of Definition \ref{a123} below.

\begin{definition}
\label{a123}
Let $X$ be smooth quasi-projective variety.
A compatible family for $X$ is a set $\G(X)=\{\gl{l},\,l\in\zz_{\geq 0}\}$ of homomorphisms
$$
\gl{l}:A^{n-l}(X)\row B^*(X)[[z_1,\ldots,z_l]]_{(m)}
$$
satisfying the following conditions:
\begin{itemize}
\item[$(a_{i})$ ] $\gl{l}$ is symmetric with respect to ${\frak{S}}_l$;
\item[$(a_{ii})$ ] $\gl{l}(\alpha)=\prod_{i=1}^lz_i\cdot\fl{l}(\alpha)$, for some $\fl{l}(\alpha)\in B^*(X)[[z_1,\ldots,z_l]]_{(m-l)}$.
\item[$(a_{iii})$ ]
$
\gl{l}(\alpha)(x+_By,z_2,\ldots,z_l)=
\sum\limits_{i,j}\gl{l+i+j-1}(\alpha\cdot a_{i,j}^A)(x^{\times i},y^{\times j},z_2,\ldots,z_l),
$
\\
where $a_{i,j}^A$ and $a_{i,j}^B$ are the coefficients of the
formal group laws of $A^*$ and $B^*$.
\end{itemize}
\end{definition}

Let $\chi_A(x)=(-_Ax)=\sum_{i\geq 0}e^A_i\cdot x^{i+1}$, $x-_Ay=\sum_{i,j}b^A_{i,j}x^iy^j$, and
similarly for $B$. Then we also have:
\begin{equation}
\begin{split}
\label{aiii-obr}
&\gl{l}(\alpha)(-_Bx,z_2,\ldots,z_l)=\sum_{i\geq 0}\gl{l+i}(\alpha\cdot e^A_i)(x^{\times i+1},z_2,\ldots,z_l),
\hspace{5mm}\text{and}
\end{split}
\end{equation}
\begin{equation}
\begin{split}
\label{aiii-minus}
&\gl{l}(\alpha)(x-_By,z_2,\ldots,z_l)=\sum_{i,j}\gl{l+i+j-1}(\alpha\cdot b_{i,j}^A)(x^{\times i},y^{\times j},z_2,\ldots,z_l).
\end{split}
\end{equation}
Indeed, let us prove by (simultaneous for all $l$) induction on $N$ that the first identity holds modulo $x^N$.
The base $N=1$ is clear from $(a_{ii})$. Suppose the statement holds for $N$. Plugging $y=-_Bx$ into $(a_{iii})$
and using $(a_{ii})$ we obtain an identity:
\begin{equation*}
\begin{split}
0=\gl{l}(\alpha)(x,z_2,\ldots,z_l)+\gl{l}(\alpha)(-_Bx,z_2,\ldots,z_l)
+\sum_{i,j\geq 1}\gl{l+i+j-1}(\alpha\cdot a^A_{i,j})(x^{\times i},-_Bx^{\times j},z_2,\ldots,z_l).
\end{split}
\end{equation*}
But for $i>0$, by $(a_{ii})$, $(a_i)$ and the inductive assumption, modulo $x^{N+1}$,
\begin{equation*}
\begin{split}
\gl{l+i+j-1}(\beta)(x^{\times i},-_Bx^{\times j},z_2,\ldots,z_l)\equiv
\sum_{k_1,\ldots,k_j\geq 0}\gl{l+i+j-1+\sum_{r=1}^jk_r}(\beta\cdot\prod_{r=1}^je^A_{k_r})(x^{\times i+j+\sum_{r=1}^j k_r},z_2,\ldots,z_l).
\end{split}
\end{equation*}
Thus, modulo $x^{N+1}$,
\begin{equation*}
\begin{split}
\gl{l}(\alpha)(-_Bx,z_2,\ldots,z_l)\equiv &-\gl{l}(\alpha)(x,z_2,\ldots,z_l)\\
&-\sum_{i,j\geq 1}\sum_{k_1,\ldots,k_j\geq 0}
\gl{l+i+j-1+\sum_{r=1}^jk_r}(\alpha\cdot a^A_{i,j}\cdot\prod_{r=1}^je^A_{k_r})(x^{\times i+j+\sum_{r=1}^j k_r},z_2,\ldots,z_l).
\end{split}
\end{equation*}
At the same time, the identity $x+_A (-_Ax)=0$ can be rewritten as:
$$
\sum_{i,j}\sum_{k_1,\ldots,k_j\geq 0}a^A_{i,j}\prod_{r=1}^je^A_{k_r}u^{i+j+\sum_{r=1}^jk_r}=0.
$$
Hence,
\begin{equation*}
\begin{split}
0=&\gl{l}(\alpha)(x,z_2,\ldots,z_l)+\sum_{k\geq 0}\gl{l+k}(\alpha\cdot e^A_k)(x^{\times 1+k},z_2,\ldots,z_l)\\
&+\sum_{i,j\geq 1}\sum_{k_1,\ldots,k_j\geq 0}
\gl{l+i+j-1+\sum_{r=1}^jk_r}(\alpha\cdot a^A_{i,j}\cdot\prod_{r=1}^je^A_{k_r})(x^{\times i+j+\sum_{r=1}^j k_r},z_2,\ldots,z_l)
\end{split}
\end{equation*}
where the first term corresponds to $(i,j)=(1,0)$ and the second one to $(0,1)$.
Thus,
$$
\gl{l}(\alpha)(-_Bx,z_2,\ldots,z_l)\equiv \sum_{k\geq 0}\gl{l+k}(\alpha\cdot e^A_k)(x^{\times 1+k},z_2,\ldots,z_l)
$$
modulo $x^{N+1}$, and so, the induction step and the first formula are proven.

The second formula easily follows from the first one and $(a_{iii})$, since $(x-_By)=x+_B(-_By)$, while
$$
\sum_{s,t}b^A_{s,t}x^sy^t=\sum_{i,j}a^A_{i,j}x^i(\sum_{k\geq 0}e^A_ky^{k+1})^j.
$$

If $E$ is a vector bundle of rank $r$ on $X$, and $\lambda^B_1,\ldots,\lambda^B_r$ are $B$-roots of $E$ (see Subsection \ref{FGL}),
then it follows from $(a_i)$ that $\fl{l+r}(\alpha)(\lambda^B_1,\ldots,\lambda^B_r,z_1,\ldots,z_l)$ is a function
of $c^B_1(E),\ldots,c^B_r(E)$, and so, this expression is well-defined as an element of $B^*(X)[[z_1,\ldots,z_l]]$.

\begin{definition}
\label{b12}
Let $d$ be a natural number.
A compatible family in dimension $\leq d$ is the data of compatible family $\G(X)=\{\gl{l};\,l\in\zz_{\geq 0}\}$
for each $X\in\smk$ of dimension $\leq d$ satisfying the following conditions (with $X$ and $Y$ of dimension $\leq d$):
\begin{itemize}
\item[$(b_{i})$ ] For any $f:X\row Y$ and any $\alpha\in A^{n-l}(Y)$,
$$
\gl{l}(f_A^*(\alpha))=f_B^*\gl{l}(\alpha).
$$
\item[$(b_{ii})$ ] For any regular embedding $j:X\row Y$ of codimension $r$ with normal bundle $N_j$
with $B$-roots $\mu_1^B,\ldots,\mu_r^B$, for any $\alpha\in A^{n-l-r}(X)$, one has:
$$
\fl{l}(j_*(\alpha))(z_1,\ldots,z_l)=j_*(\fl{l+r}(\alpha)(\mu_1^B,\ldots,\mu_r^B,z_1,\ldots,z_l)).
$$
\end{itemize}
We will say that a compatible family is secured in dimension $\leq d$ if a compatible family in dimension $\leq d$ is constructed
extending the compatible family $\G(\op{Spec}(k))$ initially given (by the statement of Theorem 5.1).
\end{definition}

The condition $(b_{ii})$ can be rewritten as:
$$
\gl{l}(j_*(\alpha))(z_1,\ldots,z_l)=j_*\operatornamewithlimits{Res}_{t=0}
\frac{\gl{l+r}(\alpha)(t+_B\mu_1^B,\ldots,t+_B\mu_r^B,z_1,\ldots,z_l)\cdot\omega_t^B}
{(t+_B\mu_1^B)\cdot\ldots\cdot(t+_B\mu_r^B)\cdot t}
$$
where $\omega_t^B$ is the canonical invariant $1$-form - see Subsection \ref{pb-bu}.

In such a situation we have the following specialization result.
To shorten the notations, we will denote $z_1,\ldots,z_l$ by $\ov{z}$

\begin{lemma}
\label{specializ}
Assume that a compatible family in dimension $\leq d$ is secured.
Let $X$ be a smooth quasi-projective variety of dimension $\leq d$,
and $L$ be a line bundle on $X$ with $\lambda^A=c^A_1(L)$, $\lambda^B=c^B_1(L)$.
Then, for any $\alpha\in A^{n-l-1}(X)$,
\begin{equation}
\label{Xspec}
\gl{l}(\alpha\cdot\lambda^A)(\ov{z})=\gl{l+1}(\alpha)(\lambda^B,\ov{z}).
\end{equation}
\end{lemma}

\begin{proof}
Assume that $L$ is very ample and let $j:Y\row X$ be the inclusion of a smooth divisor such that $L\cong O(Y)$. Then,
$\alpha\cdot\lambda^A=j_*j^*(\alpha)$. Also, note that the normal bundle $N_j$ is the pull-back of $L$.
Thus, if $\mu^B=c^B_1(N_j)$, we have $\mu^B=j^*(\lambda^B)$. This is said, (\ref{Xspec}) follows from the following chain of equalities:
\begin{equation*}
\begin{split}
\gl{l}(\alpha\cdot\lambda^A)(\ov{z})=&\fl{l}(j_*j^*(\alpha))(\ov{z})\cdot\prod_{i=1}^lz_i=
j_*\fl{l+1}(j^*(\alpha))(\mu^B,\ov{z})\cdot\prod_{i=1}^lz_i=
j_*\fl{l+1}(j^*(\alpha))(j^*\lambda^B,\ov{z})\cdot\prod_{i=1}^lz_i=\\
&j_*j^*\fl{l+1}(\alpha)(\lambda^B,\ov{z})\cdot\prod_{i=1}^lz_i=
\fl{l+1}(\alpha)(\lambda^B,\ov{z})\cdot\lambda^B\cdot\prod_{i=1}^lz_i=
\gl{l+1}(\alpha)(\lambda^B,\ov{z}).
\end{split}
\end{equation*}

In general, we can write $L=L_1\otimes L_2^{-1}$, where $L_1$ and $L_2$ are very ample line bundles. Using 1) and
(\ref{aiii-minus}), we get:
\begin{equation*}
\begin{split}
&\gl{l}(\alpha\cdot\lambda^A)(\ov{z})=
\sum_{i,j}\gl{l}(\alpha\cdot(\lambda_1^A)^i(\lambda_2^A)^j\cdot b^A_{i,j})(\ov{z})=\\
&\sum_{i,j}\gl{l+i+j}(\alpha\cdot b^A_{i,j})((\lambda^A_1)^{\times i},(\lambda^A_2)^{\times j},\ov{z})=
\gl{l+1}(\alpha)(\lambda_1^B-_B\lambda_2^B,\ov{z})=
\gl{l+1}(\alpha)(\lambda^B,\ov{z}).
\end{split}
\end{equation*}
\Qed
\end{proof}

Let $\G(X)$ be a compatible family for $X$. We define a compatible family $\G(X\times\pp^{\infty})=\{\gl{l};\,l\in\zz_{\geq 0}\}$
for $X\times\pp^{\infty}$ as follows. We have:
$A^*(X\times\pp^{\infty})=A^*(X)[[t]]$, where $t=c^A_1(O(1))$. For $\alpha(t)=\sum_{i=0}^{\infty}\alpha_i\cdot t^i$,
with $\alpha_i\in A^{n-l-i}(X)$, set:
$$
\gl{l}(\alpha(t))(\ov{z})=\sum_i\gl{l+i}(\alpha_i)(t^{\times i},\ov{z})\in B^*[[t]][[z_1,\ldots,z_l]],
$$
which converges by $(a_{ii})$.
It follows immediately from the definition that $(a_{i,ii,iii})$ are satisfied.

\begin{lemma}
\label{pinfspecializ}
Assume that $\G(X)$ satisfies $(\ref{Xspec})$. Then $\G(X\times\pp^{\infty})$, as defined above, satisfies also $(\ref{Xspec})$.
\end{lemma}

\begin{proof}
A line bundle $L$ on $X\times\pp^{\infty}$ has the form $M(r)$, for some
$r\in\zz$, and some line bundle $M$ on $X$. Let $\mu^A=c^A_1(M)$,
$\mu^A+_A[r]\cdot_At=\sum_{i,j}c^A_{i,j}(\mu^A)^it^j$, and $\gamma\in A^*(X)$. Then,
by the definition of $\G(X\times\pp^{\infty})$, the condition $(\ref{Xspec})$ for $X$, and $(a_{iii})$,
we get:
\begin{equation*}
\begin{split}
&\gl{l}(\gamma\cdot t^p\cdot (\mu^A+_A[r]\cdot_At))(\ov{z})=
\sum_{i,j}\gl{l+p+i+j}(\gamma\cdot c^A_{i,j})(t^{\times p+j},(\mu^B)^{\times i},\ov{z})=\\
&\gl{l+p+1}(\gamma)((\mu^B+_B[r]\cdot_Bt),t^{\times p},\ov{z})=
\gl{l+1}(\gamma\cdot t^p)((\mu^B+_B[r]\cdot_Bt),\ov{z}).
\end{split}
\end{equation*}

This extends to arbitrary element $\alpha(t)$ of $A^*(X\times\pp^{\infty})$ by linearity.
\Qed
\end{proof}

Suppose that a compatible family is secured in dimension $\leq d-1$. Let $X$ be smooth quasi-projective variety of dimension $d$
and let $D\subset X$ be a divisor with strict normal crossings. We denote by $D_i$ the irreducible components of $D$ and we denote
by $d:D\row X$, $\hat{d}_i:D_i\row D$ and $d_i=d\circ\hat{d}_i:D_i\row X$ the obvious inclusions.
Set $\lambda^B_i=c^B_1(O(D_i))$ and let $\gamma=\sum_i(\hat{d}_i)_*(\gamma_i)\in A^{n-l-1}(D)$. We define:
\begin{equation*}
\label{md}
\tag{$*$}
\fl{l}(\gamma|D)(\ov{z}):=
\sum_i(d_i)_*\fl{l+1}(\gamma_i)(\lambda^B_i,\ov{z})\in B^*(X)[[z_1,\ldots,z_l]]_{(m-l)}.
\end{equation*}
Notice, that $\ddim(D_i)\leq d-1$, so $\G(D_i)$ is defined.
Applying $(b_{ii})$ to $D_{\{i,j\}}\stackrel{d_{\{i,j\}/i}}{\lrow}D_{i}$, we get:
\begin{equation*}
\begin{split}
\fl{l+1}((d_{\{i,j\}/i})_*\delta)(\lambda^B_i,\ov{z})=(d_{\{i,j\}/i})_*
\fl{l+2}(\delta)(\lambda^B_j,\lambda^B_i,\ov{z}),
\end{split}
\end{equation*}
which implies that our definition does not depend on the presentation of $\gamma$ as a sum of $(\hat{d}_i)_*(\gamma_i)$.
Also it follows from $(b_{ii})$ that, in the case $\ddim(X)\leq (d-1)$ we have:
$$
\fl{l}(\gamma|D)(\ov{z})=\fl{l}(d_*(\gamma))(\ov{z}).
$$

\begin{proposition}
\label{flmp}
Consider a cartesian square
\begin{equation*}
%\label{divsquare}
\xymatrix @-0.7pc{
E \ar @{->}[r]^(0.5){e} \ar @{->}[d]_(0.5){\ov{f}}&
Y \ar @{->}[d]^(0.5){f}\\
D \ar @{->}[r]^(0.5){d} & X.
}
\end{equation*}
where $X$ and $Y$ are smooth quasi-projective of dimension $\leq d$, and $D$ and $E$ are divisors with strict normal crossings.
Then with the notations of the Subsection \ref{subsMPEIF}, we have:
$$
f^*\fl{l}(\gamma|D)(\ov{z})=\fl{l}(\ov{f}^{\,\star}(\gamma)|E)(\ov{z}).
$$
\end{proposition}

\begin{proof}
From the definition (\ref{md}) above and the definition of $\ov{f}^{\,\star}$
(Definition \ref{fstar}), it is clear that it is sufficient to treat the case of a smooth $D$.
Let $E=\sum_{j=1}^sm_j\cdot E_j$, where $E_i$ are irreducible components of $E$, $\lambda^A=c^A_1(O_X(D))$, $\mu^A_j=c^A_1(O_Y(E_j))$
(and similarly for $\lambda^B$, $\mu^B_j$).

We make a consistent choice of power series $(F^{m_1,\ldots,m_s}_J)^A$ and $(F^{m_1,\ldots,m_s}_J)^B$ as in Definition \ref{divclass}.
(These are power series in $s$ variables and with coefficients in $A^*(k)$ and $B^*(k)$ respectively.)
The integers $m_1,\ldots,m_s$ being fixed (until the end of the proof of Proposition \ref{flmp}), we will write below
$C^A_J$ for $(F^{m_1,\ldots,m_s}_J)^A(\mu^A_1,\ldots,\mu^A_s)$ and similarly for $B$.

\begin{lemma}
\label{lemflmp}
Keep the notation as in Proposition \ref{flmp} with $\ddim(Y)\leq d$ (while the dimension of $X$ can be arbitrary)
and assume that $D$ is smooth. We have
$$
\fl{l}(\ov{f}^{\,\star}(\gamma)|E)(\ov{z})=
\sum_{\emptyset\neq J\subset\{1,\ldots,s\}}(e_J)_*
\Big(C^B_J\cdot \fl{l+1}(\hat{f}_J^*(\gamma))(f_J^*(\lambda^B),\ov{z})\Big)
$$
where $e_J:E_J=\cap_{j\in J}E_j\row Y$ and $\hat{f}_J:E_J\row D$ are obvious maps, and $f_J=d\circ\hat{f}_J$.
\end{lemma}

\begin{proof}
We will denote the $1$-st Chern class of the bundle
$O(1)$ on $\pp^{\infty}$ (in both $A^*$ and $B^*$-theory) by $t$.
Let $\wt{\mu}^A_j=t+_A\mu^A_j$, and similarly for $B$.
Let us denote: $\mu^A_I=\sum_{j\in I}^A[m_j]\cdot_A\mu^A_j$,
$(\mu^A)^J=\prod_{j\in J}\mu^A_j$, $(\mu^A)^{\times J}=\times_{j\in J}\mu^A_j$,
and similarly for $\wt{\mu}^A$, $\mu^B$, $\wt{\mu}^B$.

From (\ref{div-coeff}) of Subsection \ref{subsMPEIF} together with $(b_i)$ and the projection formula
it follows that the RHS of our formula does not depend on the
choice of coefficients $C^{B}_J$ (recall, that these coefficients are selected by the property that
$\sum_JC^B_J\cdot(\mu^B)^J$ is a fixed expression).
Let us use the standard choice for $C^{A,B}_J$.
This choice satisfies: $C^A_J=\frac{\sum_{I\subset J}(-1)^{|J|-|I|}\mu^A_I}{(\mu^A)^J}$ - see (\ref{F-coeff}) of
Subsection \ref{subsMPEIF}.
Denote as $\wt{C}^A_J$ the analogous coefficients for $\wt{\mu}_j$.
We have:
\begin{equation*}
\begin{split}
\fl{l}(\ov{f}^{\,\star}(\gamma)|E)(\ov{z})=\sum_{\emptyset\neq J\subset\{1,\ldots,s\}}(e_J)_*
\fl{l+|J|}(\hat{f}_J^*(\gamma)\cdot C^A_J)((\mu^B)^{\times J},\ov{z}).
\end{split}
\end{equation*}
Indeed, it follows from $(b_{ii})$ and the projection formula that the RHS here does not depend on the choice of coefficients $C^A_J$
(since these are selected by the property that
$\sum_JC^A_J\cdot(\mu^A)^J$ is a fixed expression).
So, we can assume that
$C^A_J$ is zero for $|J|>1$. Then our formula is reduced to the definition $(*)$ of $\fl{l}(\ov{f}^{\,\star}(\gamma)|E)(\ov{z})$.

The latter expression can be rewritten as
\begin{equation*}
\begin{split}
\sum_{\emptyset\neq J\subset\{1,\ldots,s\}}(e_J)_*\operatornamewithlimits{Res}_{t=0}R_J\cdot\omega^B_t,
\hspace{5mm}\text{where}\hspace{5mm}
R_J=
\frac{\gl{l+|J|}(\hat{f}_J^*(\gamma)\cdot \wt{C}^A_J)((\wt{\mu}^B)^{\times J},\ov{z})}
{t\cdot(\wt{\mu}^B)^{J}\cdot\prod_{i=1}^lz_i}.
\end{split}
\end{equation*}

Applying Lemmas \ref{specializ}, \ref{pinfspecializ} and (\ref{F-coeff}) we get:
\begin{equation*}
\begin{split}
&R_J=
\frac{\gl{l+|J|}(\hat{f}_J^*(\gamma)\cdot \wt{C}^A_J)((\wt{\mu}^B)^{\times J},\ov{z})}
{t\cdot(\wt{\mu}^B)^{J}\cdot\prod_{i=1}^lz_i}=
\frac{\gl{l}(\hat{f}_J^*(\gamma)\cdot (\sum_{I\subset J}(-1)^{|J|-|I|}
\wt{\mu}^A_I))(\ov{z})}
{t\cdot(\wt{\mu}^B)^{J}\cdot\prod_{i=1}^lz_i}=\\
&\sum_{I\subset J}(-1)^{|J|-|I|}
\frac{\gl{l+1}(\hat{f}_J^*(\gamma))(\wt{\mu}^B_I,\ov{z})}
{t\cdot(\wt{\mu}^B)^{J}\cdot\prod_{i=1}^lz_i}=
\sum_{I\subset J}(-1)^{|J|-|I|}
\frac{\wt{\mu}^B_I}{(\wt{\mu}^B)^{J}}\cdot
\frac{\fl{l+1}(\hat{f}_J^*(\gamma))(\wt{\mu}^B_I,\ov{z})}{t}=\\
&\sum_{L\subset J}\frac{\wt{C}^B_L}{t\cdot(\wt{\mu}^B)^{J/L} }\sum_{L\subset I\subset J}
(-1)^{|J|-|I|}\fl{l+1}(\hat{f}_J^*(\gamma))(\wt{\mu}^B_I,\ov{z}),
\end{split}
\end{equation*}
where in the last equality we use the definition of $\wt{C}^B_L$, that is, the identity
$\wt{\mu}^B_I=\sum_{L\subset I}\wt{C}^B_L\cdot (\wt{\mu}^B)^L$ - see (\ref{div-coeff}) of Subsection \ref{subsMPEIF}.

Let us fix $L\subset J$. Then
$\sum_{L\subset I\subset J}F_{l+1}(\hat{f}^*_J(\gamma))(\wt{\mu}^B_I,\ov{z})\cdot (-1)^{|J|-|I|}$ is
divisible by $(\wt{\mu}^B)^{J/L}$.
Indeed, here $F_{l+1}(\hat{f}^*_J(\gamma))(x,\ov{z})$ can be considered as a power series $F(x)$ over the ring $R=B^*(E_J)[[t]][[\ov{z}]]$
with a formal group law on it.
We can plug $\wt{\mu}^B_I$ into this (or any other) power series, because it has nilpotent constant term.
Now, what we need follows from the fact that, for any collection of elements $v_j\in R,j\in J$ with nilpotent constant terms,
and for any power series $F(x)$
over $R$, we have: $\sum_{L\subset I\subset J}(-1)^{|J|-|I|}F(\sum^B_{i\in I}v_i)$ is divisible by $\prod_{i\in J\backslash L}v_i$.
We prove this by induction on $r=|J\backslash L|$. The base $r=0$ is obvious. For $r>0$, choose some $j\in J\backslash L$, and denote
$J'=J\backslash\{j\}$. Then we can write: $F(x+_Bv_j)-F(x)=v_j\cdot G(x)$, for some power series $G(x)\in R[[x]]$. And since
$$
\sum_{L\subset I\subset J}(-1)^{|J|-|I|}F(\sum\nolimits^B_{i\in I}v_i)=
v_j\cdot\sum_{L\subset I'\subset J'}(-1)^{|J'|-|I'|}G(\sum\nolimits^B_{i\in I'}v_i),
$$
the needed divisibility follows from the inductive assumption.

Now fix $L$. Then we obtain:
\begin{equation*}
\begin{split}
&\operatornamewithlimits{Res}_{t=0}\wt{C}^B_L\cdot\sum_{L\subset J}(e_{J/L})_*
\frac{1}{t\cdot(\wt{\mu}^B)^{J/L}}\cdot\hspace{-2mm}
\sum_{L\subset I\subset J}\hspace{-2mm}(-1)^{|J|-|I|}
F_{l+1}(\hat{f}^*_J(\gamma))(\wt{\mu}^B_I,\ov{z})\cdot \omega^B_t=\\
&\operatornamewithlimits{Res}_{t=0}\wt{C}^B_L\cdot\sum_{L\subset J}
\frac{(\mu^B)^{J/L}}{t\cdot(\wt{\mu}^B)^{J/L}}
\cdot\sum_{L\subset I\subset J}(-1)^{|J|-|I|}F_{l+1}(\hat{f}^*_L(\gamma))(\wt{\mu}^B_I,\ov{z})\cdot \omega^B_t=\\
&\operatornamewithlimits{Res}_{t=0}\wt{C}^B_L\cdot\sum_{L\subset J}
\frac{1}{t}
\cdot\sum_{L\subset I\subset J}(-1)^{|J|-|I|}F_{l+1}(\hat{f}^*_L(\gamma))(\wt{\mu}^B_I,\ov{z})\cdot  \omega^B_t,
\end{split}
\end{equation*}
where in the first equality we are using $(b_i)$ which guarantees that $F_{l+1}(\hat{f}^*_J(\gamma))(\wt{\mu}^B_I,\ov{z})$ belongs to
the image of $(e_{J/L})^*$, and in the second one we apply the above divisibility.

Hence,
$\sum_{J\subset\{1,\ldots,s\}}(e_J)_*\operatornamewithlimits{Res}_{t=0}R_J\cdot\omega^B_t=
\sum_{L\subset\{1,\ldots,s\}}(e_L)_*\operatornamewithlimits{Res}_{t=0}S_L\cdot\omega^B_t$, where
\begin{equation*}
\begin{split}
S_L=\sum_{L\subset K}\frac{\wt{C}^B_L}{t}\sum_{L\subset N\subset K}(-1)^{|K|-|N|}
\fl{l+1}(\hat{f}_L^*(\gamma))(\wt{\mu}^B_N,\ov{z})=
\frac{\wt{C}^B_L}{t}
\fl{l+1}(\hat{f}_L^*(\gamma))(\wt{\mu}^B_{\{1,\ldots,s\}},\ov{z}).\hspace{1cm}\text{Thus,}
\end{split}
\end{equation*}
\begin{equation*}
\begin{split}
&\fl{l}(\ov{f}^{\,\star}(\gamma)|E)(\ov{z})=\sum_{L\subset\{1,\ldots,s\}}(e_L)_*
\Big(C^B_L\cdot \fl{l+1}(\hat{f}_L^*(\gamma))(f_L^*(\lambda^B),\ov{z})\Big).
\end{split}
\end{equation*}
\Qed
\end{proof}
We now return to the Proposition \ref{flmp}.
It remains to observe that our expression is equal to
$f^*\fl{l}(\gamma|D)(\ov{z})$,
by Proposition \ref{MPEIF} and $(b_i)$.
\Qed
\end{proof}

Suppose that a compatible family is secured in dimension $\leq d-1$. Fix a smooth quasi-projective variety $X$ of dimension $\leq d$.
We will now explain how to construct a compatible family $\G(X)$ for $X$. We have a decomposition
$A^*(X)\cong A^*(k)\oplus\ov{A}^*(X)$, and by Theorem \ref{Ac}, we have an isomorphism
$\ov{A}^*(X)\cong H({\frc})$ where ${\frc}$ is the short bi-complex introduced in Subsection \ref{c}.
(Recall that we are assuming that the theory $A^*$ is of rational type.)

For the constant part, we set: $\gl{l}(p_X^*(\alpha))(z_1,\ldots,z_l)=p_X^*\gl{l}(\alpha)(z_1,\ldots,z_l)$ where $p_X:X\row\op{Spec}(k)$
is the structural map.

To define $\G(X)$ on $H({\frc})$, we first define maps $G_l:(c_{0,0})^{n-l}\row B^*(X)[[z_1,\ldots,z_l]]_{(m)}$ and then check that
these maps factor through $H({\frc})^{n-l}$. Recall that
$$
(c_{0,0})^{n-l}=\bigoplus_{\cv\in Ob(\rc(X))}\op{Im}(\rho^!:A_{\ddim(X)-n+l}(Z)\row A_{\ddim(X)-n+l}(V)=A^{n-l-1}(V)).
$$
Fix $\cv=(Z\row X\stackrel{\rho}{\low}\wt{X})$ in $\rc(X)$ and $\gamma\in A^{n-l-1}(V)$ which is in the image of $\rho^!$.
(Recall that the image of $\gamma$ in $A^*(X)$ is given by $\frac{\rho_*v_*(\gamma)}{\rho_*(1^A)}$ where $v:V\row\wt{X}$
is the obvious inclusion.) Define
\begin{equation*}
\label{nd}
\tag{$**$}
\fl{l}(\cv,\gamma)(z_1,\ldots,z_l)=
\frac{\rho_*\fl{l}(\gamma|V)(z_1,\ldots,z_l)}{\rho_*(1^B)}.
\end{equation*}
Since $\ddim(X)\leq d$, the power series $F(\gamma|V)$ makes sense (see ($*$)).

\begin{lemma}
\label{X-nmo}
If $\ddim(X)\leq d-1$, then $\fl{l}(\cv,\gamma)=\fl{l}(\alpha)$, where $\alpha\in A^{m-l}(X)$ is the image of $\gamma$, i.e.,
is given by $\frac{\rho_*v_*(\gamma)}{\rho_*(1^A)}$.
\end{lemma}

\begin{proof}
Since $\ddim(X)\leq d-1$, from $(b_{ii})$, we have $\fl{l}(\gamma|V)=\fl{l}(v_*(\gamma))$. Suppose that $\gamma=\rho^!(\delta)$
for $\delta\in A_{d-n+l}(Z)$. Then by \cite[Theorem 6.6.6]{LM}, we have $v_*(\gamma)=\rho^*(z_*(\delta))$ where $z:Z\row X$
is the obvious inclusion. It follows that $\fl{l}(\gamma|V)=\fl{l}(\rho^*(z_*(\delta)))=\rho^*\fl{l}(z_*(\delta))$, by $(b_i)$.
This gives
$$
\fl{l}(\cv,\gamma)=\frac{\rho_*\rho^*\fl{l}(z_*(\delta))}{\rho^*(1^B)}=\fl{l}(z_*(\delta)).
$$
The result now follows since $\alpha=z_*(\delta)$.
\Qed
\end{proof}

Before checking that the maps $\fl{l}$ of ($**$) factor through $H({\frc})$, we note the following fact.

\begin{proposition}
\label{flpistar}
In the above situation, $\fl{l}(\gamma|V)\in \op{Image}(\rho^*)$.
\end{proposition}

\begin{proof}
We first assume that $\rho$ is a sequence of blowups in smooth centers $R_j$. Let $E_j\subset\wt{X}$ be the exceptional divisor
over $R_j$ (i.e., the strict transform of the exceptional divisor of the blowup of $R_j$). Then $E_j$ is an irreducible component
of $V$ (but, in general, $V$ is larger than $\cup_jE_j$). Denote by $e_j:E_j\row\wt{X}$ and $\eps_j:E_j\row R_j$ the obvious maps.

Then, by Proposition \ref{vvter}, to prove that $\fl{l}(\gamma|V)\in \op{Image}(\rho^*)$ we need to
show that $e_j^*(\fl{l}(\gamma|V))\in \op{Image}(\eps_j^*)$, for each $j$.
Since $V$ is a divisor with strict normal crossings on $\wt{X}$, and $E_j$ is a component of it,
for any other component $V_i$ of $V$, the left cartesian diagram below is transversal:
$$
\xymatrix @-0.2pc{
H_{i,j} \ar @{->}[r]^{u_{i,j}} \ar @{->}[d]_(0.5){h_{i,j}} &E_j \ar @{->}[d]^(0.5){e_j}\\
V_i \ar @{->}[r]_{v_i} & \wt{X}.
}
\hspace{1cm}
\xymatrix @-0.2pc{
V \ar @{->}[r]^{v} \ar @{->}[d]_(0.5){\rho_V} & \wt{X} \ar @{->}[d]_(0.5){\rho}  &
E_j \ar @{->}[l]_(0.5){e_j}  \ar @{->}[d]^(0.5){\eps_j}   \\
Z \ar @{->}[r]_{z} & X    &  R_j \ar @{->}[l]^(0.5){r_j}
}
$$
Now, fix a presentation $\gamma=\sum_i(\hat{v}_i)_*(\gamma_i)$ where $\hat{v}_i:V_i\row V$ are the obvious inclusions.
Applying Proposition \ref{flmp} (in the trivial case where the divisors are smooth) we get
$e_j^*\fl{l}(\gamma_i|V_i)=\fl{l}(h^*_{i,j}(\gamma_i)|H_{i,j})$.
Since $\ddim(E_j)\leq d-1$, we have:
$$
\fl{l}(h^*_{i,j}(\gamma_i)|H_{i,j})=\fl{l}((u_{i,j})_*h^*_{i,j}(\gamma_i))=
\fl{l}(e_j^*(v_i)_*(\gamma_i)).
$$
And the
same is true for the $E_j$-component:
$e_j^*\fl{l}(\gamma_j|E_j)=\fl{l}(e_j^*(e_j)_*(\gamma_j))$.
Here we use the fact that the map $e_j^*(e_j)_*$ is given by the multiplication by the first Chern class of $O_{\wt{X}}(E_j)$
and Lemma \ref{specializ}.

Hence, using \cite[Theorem 6.6.6 (2)(a)]{LM},
\begin{equation*}
\begin{split}
&e_j^*\fl{l}(\gamma|V)=\sum_ie_j^*\fl{l}(\gamma_i|V_i)=\sum_i\fl{l}(e_j^*(v_i)_*(\gamma_i))=
\fl{l}(e_j^*v_*(\gamma))=\\
&\fl{l}(e_j^*v_*\rho^{!}(\beta))=\fl{l}(e_j^*\rho^*z_*(\beta))=
\fl{l}(\eps_j^*r_j^*z_*(\beta))=\eps_j^*\fl{l}(r_j^*z_*(\beta)).
\end{split}
\end{equation*}
So, $\fl{l}(\gamma|V)\in \op{Image}(\rho^*)$.

For general $\rho$, we may find $\cv'=(Z\row X\stackrel{\rho'}{\low}\wt{X}')$ and a morphism $(id,\pi):\cv'\row\cv$ such that
$\rho'$ is a composition of blowups in smooth centers. By Proposition \ref{flmp},
$\pi^*\fl{l}(\gamma|V)=\fl{l}(\pi_V^{\star}(\gamma)|V')$ which belongs to $\op{Image}(\pi^*\rho^*)$
by the previous discussion (recall that $\pi_V^{\star}(\gamma)=\pi^!(\gamma)$). Since the map $\pi^*$ is injective,
we deduce that $\fl{l}(\gamma|V)\in \op{Image}(\rho^*)$.
\Qed
\end{proof}

\begin{lemma}
\label{d-one-zero}
The maps $\fl{l}: (c_{0,0})^{n-l}\row B^*(X)[[z_1,\ldots,z_l]]_{(m-l)}$, defined above, are zero on $\op{Image}(d^{{\frc}}_{1,0})$.
\end{lemma}

\begin{proof}
The vanishing of $\fl{l}$ on the image of $d^{II}_{1,0}$ is a direct consequence of the definition. To show that $\fl{l}$ vanishes
on the image of $d^I_{1,0}$, let $(id,\pi):\cv_2\row\cv_1$ be in $Mor_I$ and
$\gamma\in\op{Image}(\rho^!_1:A_{\ddim(X)-n+l}(Z)\row A^{n-l-1}(V))$. We need to show that
$\fl{l}(\cv_1,\gamma)=\fl{l}(\cv_2,\pi_V^{\star}(\gamma))$. By Propositions \ref{flmp}, we have
$$
\frac{(\rho_2)_*\fl{l}(\pi_V^{\star}\gamma|V_2)}{(\rho_2)_*(1^B)}=
\frac{(\rho_1)_*\pi_*\pi^*\fl{l}(\gamma|V_1)}{(\rho_1)_*\pi_*(1^B)}.
$$
By Proposition \ref{flpistar}, we know that $\fl{l}(\gamma|V_1)\in\op{Image}(\rho_1^*)$. By the projection formula, the previous fraction
is equal to
$$
\frac{(\rho_1)_*\fl{l}(\gamma|V_1)}{(\rho_1)_*(1^B)}
$$
which is what we want to prove.
\Qed
\end{proof}

We now proceed to check that the maps
$\fl{l}: (c_{0,0})^{n-l}\row B^*(X)[[z_1,\ldots,z_l]]_{(m-l)}$ are zero on the image of $d^{{\frc}}_{0,1}$.
We fix an object $\cw=(Z\row X\times\pp^1\stackrel{\rho}{\low}\wt{X\times\pp^1})$ of $\rc^1(X)$ and an element
$\delta\in\op{Image}(\rho^!:A_{\ddim(X)-n+l+1}(Z)\row A^{n-l-1}(W))$.
We denote by $W_s$ the irreducible components of $W$ and we choose a presentation $\delta=\sum_s(\hat{w}_s)_*(\delta_s)$.
(As usual, $\hat{w}_s:W_s\row W$ are the obvious inclusions.) We need to show that $\fl{l}$ takes the same value on the pairs
$$
(\partial_0 W,\sum_s i^{\star}_{0,s}(\delta_s))\hspace{1cm}\text{and}
\hspace{1cm}(\partial_{\infty} W,\sum_s i^{\star}_{\infty,s}(\delta_s)),
$$
where $i_{0,s}:W_{0,s}=\wt{X}_0\cap W_s\row W_s$ and $i_{\infty,s}:W_{\infty,s}=\wt{X}_{\infty}\cap W_s\row W_s$
are the obvious inclusions.

To prove this, we need some preparation. Let $S$ we an irreducible component of $W$ (i.e., $S$ is one of $W_s$'s), and denote
by $i_0:S_0\row S$ and $i_{\infty}:S_{\infty}\row S$ the inclusions of the fibers over $0,\infty\in\pp^1$.
Assume we are given some object $\ch=(T\row S\stackrel{p}{\low}\wt{S})$ in $\rc(S)$ such that $T$ does not contain any component
of $S_0$ and $S_{\infty}$. As usual, we write $H=p^{-1}(T)$ and we denote by $h:H\row \wt{S}$ the inclusion.
Finally, assume we are given $\gamma=p^!(u)\in\op{Image}(p^!:A_{\ddim(X)-n+l+1}(T)\row A^{n-l-2}(H))$ and let
$\beta\in A^{n-l-1}(S)$ be the push-forward of $u$. (Thus, $p^*(\beta)=h_*(\gamma)$.) We set
$$
\wt{\fl{l}}(\beta|S)(\ov{z})=s_*(\fl{l+1}(\ch,\gamma)(p^*(\lambda^B),\ov{z}))=
s_*\biggl(\frac{p_*\fl{l+1}(\gamma|H)(p^*(\lambda^B),\ov{z})}{p_*(1)}\biggl)
$$
where $s:S\row\wt{X\times\pp^1}$ is the obvious inclusion and $\lambda^B=c_1^B(O_{\wt{X\times\pp^1}}(S))$.

\begin{lemma}
\label{lemd01}
Denote by $\wt{i}_0:\wt{X}_0\row\wt{X\times\pp^1}$ and $\wt{i}_{\infty}:\wt{X}_{\infty}\row\wt{X\times\pp^1}$
the obvious inclusions.
In the above situation,
$$
\wt{i}^*_0(\wt{\fl{l}}(\beta|S))=\fl{l}(i_0^{\star}(\beta)|S_0)
\hspace{5mm}\text{and}\hspace{5mm}
\wt{i}^*_{\infty}(\wt{\fl{l}}(\beta|S))=\fl{l}(i_{\infty}^{\star}(\beta)|S_{\infty}).
$$
\end{lemma}

\begin{proof}
It is enough to prove the first equality. Denote by $S_{0,k}$ the irreducible components of $S_0$.
Since $S_{0,k}$ is not contained in $T$, we may consider its strict transform in $\wt{S}$. Resolving the singularities of the latter,
we get a blowup $p_{0,k}:\wt{S}_{0,k}\row S_{0,k}$ fitting in a commutative diagram
$$
\xymatrix @-0.2pc{
H_{0,k}\ar @{->}[r]^{h_{0,k}} \ar @{->}[d]^(0.5){i_{0,H}}&
\wt{S}_{0,k} \ar @{->}[r]^{p_{0,k}} \ar @{->}[d]^(0.5){\wt{i}_{0,k}}&
S_{0,k} \ar @{->}[r]^{s_{0,k}} \ar @{->}[d]^(0.5){i_{0,k}} &  \wt{X}_0 \ar @{->}[d]^(0.4){\wt{i}_0} \\
H  \ar @{->}[r]^{h} & \wt{S} \ar @{->}[r]^{p}  &  S \ar @{->}[r]^(0.4){s}  &  \wt{X\times\pp^1}
}
$$
where the left square is cartesian. Refining the resolution $p_{0,k}$, we may assume that $H_{0,k}$ is a
divisor with strict normal crossings in $\wt{S}_{0,k}$.

\begin{lemma}
\label{lem2d01}
Consider a commutative square of smooth varieties
$$
\xymatrix @-0.2pc{
E \ar @{->}[r]^{b} \ar @{->}[d]_(0.5){g} &Q \ar @{->}[d]^(0.5){f}\\
D \ar @{->}[r]_{a} & P
}
$$
where $f$ and $g$ are projective and birational. Let $x\in \op{Image}(f^*)$.
Then:
$$
\frac{g_*(b^*(x))}{g_*(1)}=a^*\biggl(\frac{f_*(x)}{f_*(1)}\biggr).
$$
\end{lemma}

\begin{proof} Let $x=f^*(y)$. Then
$g_*(b^*f^*(y))\cdot a^*f_*(1)=g_*g^*a^*(y)\cdot a^*f_*(1)=g_*(1)\cdot a^*(y)\cdot a^*f_*(1)=
g_*(1)\cdot a^*(f_*f^*(y))$,
which implies what we need.
\Qed
\end{proof}
By Proposition \ref{flpistar}, $F_{l+1}(\gamma|H)(p^*(\lambda^B),\ov{z})\in \op{Image}(p^*)$.
Then, by Lemma \ref{lem2d01}, Proposition \ref{flmp} and $(b_i)$, we have:
\begin{equation*}
\begin{split}
&i_{0,k}^*\biggl(\frac{p_*\fl{l+1}(\gamma|H)(p^*(\lambda^B),\ov{z})}{p_*(1)}\biggr)=
\frac{(p_{0,k})_*(\wt{i}_{0,k}^*\fl{l+1}(\gamma|H)(p^*(\lambda^B),\ov{z}))}{(p_{0,k})_*(1)}=
\frac{(p_{0,k})_*(\fl{l+1}(\wt{i}_{0,k}^*h_*(\gamma))(\wt{i}_{0,k}^*p^*(\lambda^B),\ov{z}))}
{(p_{0,k})_*(1)}=\\
&\frac{(p_{0,k})_*(\fl{l+1}(p_{0,k}^*i_{0,k}^*(\beta))(p_{0,k}^*i_{0,k}^*(\lambda^B),\ov{z}))}
{(p_{0,k})_*(1)}=
\fl{l+1}(i_{0,k}^*(\beta))(i_{0,k}^*(\lambda^B),\ov{z}).
\end{split}
\end{equation*}
As above, denote by $C^B_J$ the coefficient $(F^{m_1,\ldots,m_r}_J)^B$ of the Definition \ref{divclass}, where $m_i$ is the multiplicity
of the component $S_{0,i}$ in $S_0$.
We can assume that coefficients $C_J^{B}$ are chosen to be zero, for $|J|>1$ - see the discussion after the Definition \ref{divclass}.
Then, by Proposition \ref{MPEIF}, and Lemma \ref{lemflmp}
\begin{equation*}
\begin{split}
\wt{i}^*_0(\wt{\fl{l}}(\beta|S)(\ov{z}))=&\wt{i}^*_0s_*
\biggl(\frac{p_*\fl{l+1}(\gamma|H)(p^*(\lambda^B),\ov{z})}{p_*(1)}\biggr)=
\sum_k(s_{0,k})_*\biggl(C^B_k\cdot i_{0,k}^*\biggl(\frac{p_*\fl{l+1}(\gamma|H)(p^*(\lambda^B),\ov{z})}{p_*(1)}\biggr)\biggr)=\\
&\sum_k(s_{0,k})_*(C^B_k\cdot \fl{l+1}(i_{0,k}^*(\beta))(i_{0,k}^*(\lambda^B),\ov{z}))=
\fl{l}(i_0^{\star}(\beta)|S_0)(\ov{z}).
\end{split}
\end{equation*}
\Qed
\end{proof}

Denote by $\pi_S:S\row\spec(k)$ the structural projection. Define:
$$
\wt{\fl{l}}(1|S)(\ov{z})=s_*(\pi_S^*\fl{l+1}(1))(\lambda^B,\ov{z}).
$$
More generally, one can define
$$
\wt{\fl{l}}(\pi_S^*(\alpha)|S)(\ov{z})=s_*(\pi_S^*\fl{l+1}(\alpha))(\lambda^B,\ov{z}).
$$

\begin{lemma}
\label{lem3d01}
In the above situation, we have:
$$
\wt{i}^*_0(\wt{\fl{l}}(\pi_S^*(\alpha)|S))=\fl{l}(i_0^{\star}\pi_S^*(\alpha)|S_0)
\hspace{5mm}\text{and}\hspace{5mm}
\wt{i}^*_{\infty}(\wt{\fl{l}}(\pi_S^*(\alpha)|S))=\fl{l}(i_{\infty}^{\star}\pi_S^*(\alpha)|S_{\infty})
$$
\end{lemma}

\begin{proof}
We treat the first equality only. By Proposition \ref{MPEIF} and Lemma \ref{lemflmp}, we have:
\begin{equation*}
\begin{split}
&\wt{i}_0^*s_*(\pi_S^*\fl{l+1}(\alpha))(\lambda^B,\ov{z})=
\sum_k(s_{0,k})_*(C^B_k\cdot (i_{0,k}^*\pi_S^*\fl{l+1}(\alpha))(i_{0,k}^*(\lambda^B),\ov{z}))=\\
&\sum_k(s_{0,k})_*(C^B_k\cdot \fl{l+1}(i_{0,k}^*\pi_S^*(\alpha))(i_{0,k}^*(\lambda^B),\ov{z}))=
\fl{l}(i_0^{\star}\pi_S^*(\alpha)|S_0)(\ov{z}).
\end{split}
\end{equation*}
\Qed
\end{proof}

\begin{proposition}
\label{d01osn}
In the above situation,
$$
\frac{(\rho_0)_*\fl{l}(\sum\nolimits_si_{0,s}^{\star}(\delta_s)|W_0)(\ov{z})}{(\rho_0)_*(1)}=
\frac{(\rho_{\infty})_*\fl{l}(\sum\nolimits_si_{\infty,s}^{\star}(\delta_s)|W_{\infty})(\ov{z})}{(\rho_{\infty})_*(1)}.
$$
\end{proposition}

\begin{proof}

\begin{lemma}
\label{movingd01}
Let $S$ be a smooth quasi-projective variety, and let $T\subset S$ be a divisor.
Let $\beta\in \ov{A}_*(S)$. Then there exists a closed subvariety $Y\subset S$, containing no components of $T$,
and such that $\beta$ is the push-forward of an element of $A_*(Y)$.
\end{lemma}

\begin{proof}
Since $A^*$ is obtained from the Levine-Morel algebraic cobordism by change of coefficients,
it is sufficient to treat the case of $A^*=\Omega^*$. Modulo classes supported in codimension $\geq 2$, our element $\beta$
may be considered as an element of $\op{Gr}_1\Omega^*(S)$.
By \cite[Corollary 4.5.8]{LM}, we have a surjective map
$CH^1(S)\otimes\laz^*\twoheadrightarrow\op{Gr}_1\Omega^*(S)$. 
Since we can always add subvarieties of codimension $\geq 2$ to $Y$, we may assume that $A^*=\op{CH}^*$.
Since $S$ is quasi-projective, we have $O(T)=L_1\otimes L_2^{-1}$, where $L_1,L_2$ are very ample line bundles.
By the Bertini Theorem, $1_{T}^{CH}$ can be
represented as $\sum_k\pm 1_{R_k}^{CH}$, where $R_k$ are irreducible divisors different
from components of $T$.
\Qed
\end{proof}

Let $S$ be an irreducible component of $W$ (i.e., one of $W_s$'s) and let $\delta_S\in A^{n-l-1}(S)$.
We can write in a unique way $\delta_S=\pi_S^*(\alpha)+\beta$ where $\alpha\in A^{n-l-1}(k)$ and $\beta\in\ov{A}^{n-l-1}(S)$.
Applying Lemma \ref{movingd01} and Corollary \ref{NbFree} to $S$, the element $\beta$, considered as an element of $H({\frb})=\ov{A}^*(S)$,
can be represented by a pair $(\ch,x)$ where $\ch=(Y\row S\stackrel{p}{\low}\wt{S})$ is an object of $\rc(S)$ satisfying the extra
property that $Y$ does not contain any irreducible component of $S_0$ and $S_{\infty}$.
Recall also that $x\in A^{n-l-2}(H)$ where $H=p^{-1}(Y)$ and that $\beta=p_*h_*(x)$.

Applying Theorem \ref{Ac} to $S$, the element $\beta$ can be represented by the element $(\ch,\gamma)$ of $c_{0,0}$, where
$\gamma=p^{!}(p_H)_*(x)$. We set $u=(p_H)_*(x)$ so that $\beta$ is the push-forward of $u$. In this way, we are in the situation
discussed previously. In particular, we have power series $\wt{\fl{l}}(\beta|S)\in B^*(\wt{X\times\pp^1})[[z_1,\ldots,z_l]]_{(m-l)}$
such that, by Lemma \ref{lemd01},
$$
\wt{i}^*_0(\wt{\fl{l}}(\beta|S))=\fl{l}(i_0^{\star}(\beta)|S_0)
\hspace{5mm}\text{and}\hspace{5mm}
\wt{i}^*_{\infty}(\wt{\fl{l}}(\beta|S))=\fl{l}(i_{\infty}^{\star}(\beta)|S_{\infty}).
$$
We also have power series $\wt{\fl{l}}(\pi_S^*(\alpha)|S)$ verifying analogous formula by Lemma \ref{lem3d01}.
Thus, setting $\wt{\fl{l}}(\delta_S|S)=\wt{\fl{l}}(\beta|S)+\wt{\fl{l}}(\pi_S^*(\alpha)|S)$, we obtain:
$$
\wt{i}^*_0(\wt{\fl{l}}(\delta_S|S))=\fl{l}(i_0^{\star}(\delta_S)|S_0)
\hspace{5mm}\text{and}\hspace{5mm}
\wt{i}^*_{\infty}(\wt{\fl{l}}(\delta_S|S))=\fl{l}(i_{\infty}^{\star}(\delta_S)|S_{\infty}).
$$

Now, we go back to the irreducible components $W_s$ and the classes $\delta_s$ such that $\delta=\sum_s(\hat{w}_s)_*(\delta_s)$.
We set $\wt{\fl{l}}(\delta|W)=\sum_s\wt{\fl{l}}(\delta_s|W_s)$. From what we just said, we again have
$$
\wt{i}^*_0\wt{\fl{l}}(\delta|W)=\fl{l}(i_0^{\star}(\delta)|W_0)
\hspace{5mm}\text{and}\hspace{5mm}
\wt{i}^*_{\infty}(\wt{\fl{l}}(\delta|W))=\fl{l}(i_{\infty}^{\star}(\delta)|W_{\infty}).
$$
With this in hands it is now easy to finish the proof of Proposition \ref{d01osn}.
Indeed, consider the commutative diagram with cartesian and transversal squares
$$
\xymatrix @-0.2pc{
\wt{X}_0 \ar @{->}[r]^(0.4){\wt{i}_0} \ar @{->}[d]_(0.5){\rho_0} & \wt{X\times\pp^1} \ar @{->}[d]^(0.5){\rho}  &
\wt{X}_{\infty} \ar @{->}[l]_(0.35){\wt{i}_{\infty}}  \ar @{->}[d]^(0.5){\rho_{\infty}}   \\
X \ar @{->}[r]^(0.4){i_0} & X\times\pp^1    &  X \ar @{->}[l]_(0.35){i_{\infty}}.
}
$$
We have
$$
i_0^*\left(\frac{\rho_*\wt{\fl{l}}(\delta|W)}{\rho_*(1)}\right)=
\frac{(\rho_0)_*\wt{i}_0^*\wt{\fl{l}}(\delta|W)}{(\rho_0)_*(1)}=
\frac{(\rho_0)_*\fl{l}(i_0^{\star}\delta|W_0)}{(\rho_0)_*(1)},
$$
and similarly for $\infty$. Since the maps $i_0^*,i_{\infty}^*:B^*(X\times\pp^1)\row B^*(X)$ are equal,
the result follows.
\Qed
\end{proof}

It follows from Proposition \ref{d01osn} that $\fl{l}$ is trivial on the image of $d^{{\frc}}_{0,1}$,
and so it is well-defined on $H({\frc})$, and hence on $A^*(X)$.
Thus, we obtain:

\begin{proposition}
\label{flX}
Suppose that a compatible family is secured in dimension $\leq d-1$. Let $X$ be a smooth quasi-projective variety of dimension $d$.
Then the maps $\gl{l}=\left(\prod_{i=1}^lz_i\right)\cdot\fl{l}:A^{n-l}(X)\row B^*(X)[[z_1,\ldots,z_l]]_{(m)}$
define a compatible family $\G(X)$ for $X$.
\end{proposition}

\begin{proof}
Properties $(a_i)$ and $(a_{ii})$ are clear. For $(a_{iii})$, it is enough to treat the case
$\alpha=v_*(\gamma)$ where $v:V\row X$ is the inclusion of a smooth divisor.
Set $\lambda^A=c_1^A(O_X(V))$ and similarly for $B$. Then:
\begin{equation*}
\begin{split}
&\gl{l}(v_*(\gamma))(x+_By,z_2,\ldots,z_l)=v_*\operatornamewithlimits{Res}_{t=0}
\frac{\gl{l+1}(\gamma)(t+_B\lambda^B,x+_By,z_2,\ldots,z_l)\omega^B_t}{(t+_B\lambda^B)\cdot t}=\\
&v_*\operatornamewithlimits{Res}_{t=0}\sum_{i,j}\frac{\gl{l+i+j}(\gamma\cdot a^A_{i,j})
(t+_B\lambda^B,x^{\times i},y^{\times j},z_2,\ldots,z_l)\omega^B_t}{(t+_B\lambda^B)\cdot t}=
\sum_{i,j}\gl{l+i+j-1}(v_*(\gamma)\cdot a^A_{i,j})
(x^{\times i},y^{\times j},z_2,\ldots,z_l).
\end{split}
\end{equation*}
\Qed
\end{proof}

\begin{proposition}
\label{dd-1}
Suppose that a compatible family is secured in dimension $\leq d-1$. Then, extending this family using the $\G(X)$
constructed above for $X$ of dimension $d$, one gets a compatible family in dimension $\leq d$.
\end{proposition}

\begin{proof}
Above, we have defined compatible families $\G(X)$ for $X$ of dimension $\leq d$, which, in case $X$ has dimension $\leq d-1$,
extend the given compatible families in dimension $\leq d-1$. It remains to check conditions $(b_i)$ and $(b_{ii})$.

We first check $(b_i)$. Let $f:Y\row X$ be a morphism between smooth quasi-projective varieties of dimension $\leq d$
and let $\alpha\in A^{n-l}(X)$. If $\alpha$ is in the image of $\pi_X^*:A^*(k)\row A^*(X)$, the property $(b_i)$ is obvious.
Thus, we may assume that $\alpha\in\ov{A}^{n-l}(X)$. We may find $\cv_X=(Z\row X\stackrel{\rho_X}{\low}\wt{X})$
in $\rc(X)$ and $\gamma\in\op{Image}(\rho_X^!:A_{\ddim(X)-n+l}(Z)\row A^{n-l-1}(V_X))$ such that
$\alpha=\frac{(\rho_X)_*(v_X)_*(\gamma)}{(\rho_X)_*(1)}$.
By Hironaka's resolution of singularities (Theorems \ref{Hif} and \ref{Hip}) we may find
$\cv_Y=(f^{-1}(Z)\row Y\stackrel{\rho_Y}{\low}\wt{Y})$ in $\rc(Y)$ fitting in a commutative diagram where the left
square is cartesian:
$$
\xymatrix @-0.2pc{
V_Y\ar @{->}[r]^{v_Y} \ar @{->}[d]_(0.5){f_V}&
\wt{Y} \ar @{->}[r]^{\rho_Y} \ar @{->}[d]_(0.5){\wt{f}}&
Y  \ar @{->}[d]^(0.5){f}  \\
V_X  \ar @{->}[r]_{v_X} & \wt{X} \ar @{->}[r]_{\rho_X}  &  X.
}
$$
By Proposition \ref{flpistar}, Lemma \ref{lem2d01}, Proposition \ref{flmp}, and Proposition \ref{MPEIF},
\begin{equation*}
\begin{split}
&f^*\fl{l}\biggl(\frac{(\rho_X)_*(v_X)_*(\gamma)}{(\rho_X)_*(1)}\biggr)=
f^*\biggl(\frac{(\rho_X)_*\fl{l}(\gamma|V_X)}{(\rho_X)_*(1)}\biggr)=
\frac{(\rho_Y)_*\wt{f}^*\fl{l}(\gamma|V_X)}{(\rho_Y)_*(1)}=
\frac{(\rho_Y)_*\fl{l}(f_V^{\star}(\gamma)|V_Y)}{(\rho_Y)_*(1)}=\\
&\fl{l}\biggl(\frac{(\rho_Y)_*(v_Y)_*f_V^{\star}(\gamma)}{(\rho_Y)_*(1)}\biggr)=
\fl{l}\biggl(\frac{(\rho_Y)_*\wt{f}^*(v_X)_*(\gamma)}{(\rho_Y)_*(1)}\biggr)=
\fl{l}\biggl(f^*\bigg(\frac{(\rho_X)_*(v_X)_*(\gamma)}{(\rho_X)_*(1)}\bigg)\biggr).
\hspace{1cm}\text{This proves}\,\,(b_{i}).
\end{split}
\end{equation*}
Let now $X\stackrel{j}{\row}Y$ be a regular embedding of codimension $r$ with normal bundle $N_j$,
with $\ddim(Y)\leq d$. Consider the blow-up diagram:
$$
\xymatrix @-0.2pc{
E\ar @{->}[r]^{\wt{j}} \ar @{->}[d]_(0.5){\eps} &\wt{Y} \ar @{->}[d]^(0.5){\pi}\\
X  \ar @{->}[r]_{j} & Y,
}
$$
where $E=\pp_X(N_j)$, and $N_{\wt{j}}=O(-1)$.
Let $M=\eps^*N_j/O(-1)$, $\nu^{A,B}_1,\ldots,\nu^{A,B}_{r-1}$ - be roots of $M$,
$\zeta^{A,B}$ - root of $O(-1)$, and $\alpha\in A^{n-l-r}(X)$. Then, by the already proven  $(b_i)$,
the Excess Intersection Formula (Proposition \ref{excess}), the definition of $\G(\wt{Y})$,
Lemma \ref{specializ}, again $(b_i)$, and Proposition \ref{excess} again, we get:
\begin{equation*}
\begin{split}
&\pi^*\fl{l}(j_*(\alpha))(\ov{z})=\fl{l}(\pi^*j_*(\alpha))(\ov{z})=
\fl{l}(\wt{j}_*(c^A_{r-1}(M)\cdot\eps^*(\alpha)))(\ov{z})=\\
&\wt{j}_*\fl{l+1}(c^A_{r-1}(M)\cdot\eps^*(\alpha))(\zeta^B,\ov{z})=
\wt{j}_*\Big(\prod_{i=1}^{r-1}\nu^B_i\cdot\fl{l+r}(\eps^*(\alpha))
(\zeta^B,\nu^B_1,\ldots,\nu^B_{r-1},\ov{z})\Big)=\\
&\wt{j}_*\Big(c^B_{r-1}(M)\cdot\eps^*(\fl{l+r}(\alpha)(\mu^B_1,\ldots,\mu^B_r,\ov{z}))\Big)=
\pi^*j_*\fl{l+r}(\alpha)(\mu^B_1,\ldots,\mu^B_r,\ov{z}).
\end{split}
\end{equation*}
And since $\pi^*$ is injective, we obtain $(b_{ii})$.
\Qed
\end{proof}

Now we can finish the proof of Theorem $\ref{MAIN}$.

From the hypothesis of Theorem \ref{MAIN}, we have a compatible family in dimension $\leq 0$.
By Proposition \ref{dd-1} and induction, we obtain a compatible family in all dimensions, i.e., for every $X\in\smk$,
we have a compatible family $\G(X)$ and the conditions $(b_i)$ and $(b_{ii})$ are satisfied. In particular,
the map $G_0: A^n(X)\row B^m(X)$, for $X\in\smk$, form an additive operation. It remains to see that these maps
coincide with the original ones for $(\pp^{\infty})^{\times l}$.
From commutativity with the pull-backs for partial diagonals and partial projections, it is sufficient to
compare the values on $\alpha\cdot\prod_{i=1}^lz^A_i\in A^n((\pp^{\infty})^{\times l})$, where
$\alpha\in A^{n-l}(k)$, and $z^A_i=c^A_1(O(1)_i)$. Let
$j: (\pp^{\infty})^{\times l}\row (\pp^{\infty})^{\times l}$ be the product of hyperplane section
embeddings. Then
$\gl{0}(j_*(\alpha))=\gl{l}(\alpha)(z^B_1,\ldots,z^B_l)=G(\alpha\cdot\prod_{i=1}^lz^A_i)$, by $(b_{ii})$
and the definition of $\G(\spec(k))$.
Thus, $\gl{0}$ extends the original homomorphisms on products of projective spaces.
The uniqueness follows from Proposition \ref{vazhnoe}.
\Qed
\end{proof}

As a byproduct of our construction, we have proved that unstable additive operations (from a free theory) satisfy $(b_{ii})$.
This can be considered as a Riemann-Roch type result for unstable additive operations.
A similar result in the case of multiplicative operations was obtained previously in \cite{P}.

\begin{theorem}
\label{adRR}
Let $G:A^n\row B^m$ be an additive operation, where $A^*$ is free. As usual, denote by
$z^A_i\in A^1(X\times (\pp^{\infty})^{\times l})$ the first Chern class of $O(1)_i$ in the sense of the theory $A^*$
(and similarly for $B^*$). For $\alpha\in A^{n-l}(X)$,
denote $G_l(\alpha)(z_1^B,\ldots,z_l^B)=G(\alpha\cdot\prod_{i=1}^lz^A_i)\in
B^m(X\times(\pp^{\infty})^{\times l})=B^*(X)[[z^B_1,\ldots,z^B_l]]_{(m)}$. Let $j:X\row Y$ be a regular embedding of codimension $r$ with
normal bundle $N_j$ with $B$-roots $\mu^B_1,\ldots,\mu^B_r$. Then
$$
\gl{l}(j_*(\alpha))(z^B_1,\ldots,z^B_l)=j_*\operatornamewithlimits{Res}_{t=0}
\frac{\gl{l+r}(\alpha)(t+_B\mu_1^B,\ldots,t+_B\mu_r^B,z^B_1,\ldots,z^B_l)\cdot\omega_t^B}
{(t+_B\mu_1^B)\cdot\ldots\cdot(t+_B\mu_r^B)\cdot t}
$$
where $\omega_t^B$ is the canonical invariant $1$-form - see Subsection \ref{pb-bu}.
\end{theorem}

We will also need the following multiplicative version of Theorem \ref{MAIN}:

\begin{proposition}
\label{mainMULT}
Let $G:A^*((\pp^{\infty})^{\times l})\row B^*((\pp^{\infty})^{\times l})$, for $l\in\zz_{\geq 0}$, be a family of
homomorphisms satisfying the conditions $(i)$-$(v)$ of Theorem $\ref{MAIN}$. Assume also that this family is compatible with
external product. Then, the resulting additive operation $G:A^*\row B^*$ is multiplicative.
\end{proposition}

\begin{proof}
Let $X$ and $Y$ be smooth and quasi-projective varieties. We will show that
$$
\gl{l+m}(\alpha\times \beta)(x_1,\ldots,x_l,y_1,\ldots,y_m)=
\gl{l}(\alpha)(x_1,\ldots,x_l)\times\gl{m}(\beta)(y_1,\ldots,y_m)
$$
for $\alpha\in A^*(X)$ and $\beta\in B^*(Y)$.
We first prove this when $Y=\spec(k)$ by induction on the dimension of $X$.
The base and the case where $\alpha$ is constant follow from our condition.
In the case $\alpha\in\ov{A}^*(X)$, we can find a projective bi-rational morphism
$\wt{X}\stackrel{\rho}{\row}X$ such that $\rho^*(\alpha)$ is supported on some divisor
with strict normal crossings. Since $\rho^*$ is injective,
without loss of generality, we can assume that $\alpha=v_*(\alpha')$, where
$V\stackrel{v}{\row}X$ is a smooth divisor. Let $\lambda^{B}=c^{B}_1(O_X(V))$.
Then
\begin{equation*}
\begin{split}
&\gl{l+m}(\alpha\times\beta)(\ov{x},\ov{y})=
(v\times id)_*\gl{l+m+1}(\alpha'\times\beta)(\lambda^B,\ov{x},\ov{y})=\\
&(v\times id)_*(\gl{l+1}(\alpha')(\lambda^B,\ov{x})\times\gl{m}(\beta)(\ov{y}))=
\gl{l}(\alpha)(\ov{x})\times\gl{m}(\beta)(\ov{y}),
\end{split}
\end{equation*}
which proves the induction step. Now, by the induction on the $\ddim(Y)$, using similar arguments,
we prove the general case.
\Qed
\end{proof}

\section{Applications}
\label{applications}

\subsection{Unstable operations in Algebraic Cobordism}
\label{uoac}

As a first application of our main result (Theorem \ref{MAIN}), let us finish the
description of unstable operations in Algebraic Cobordism:

\begin{theorem}
\label{UOACpoln}
Let $\psi\in\op{Hom}_{\laz}(\laz[\barbi{b}],\laz\otimes_{\zz}\qq)_{(m-n)}$ be a homomorphism of $\laz$-modules.
Denote by $S_{\psi}:\Omega^n\row\Omega^m\otimes_{\zz}\qq$ the respective $\laz\otimes_{\zz}\qq$-linear combination
of the Landweber-Novikov operations, i.e., the composition of
$$
\Omega^*\stackrel{\slnT}{\lrow}\Omega^*[\barbi{b}]\cong\Omega^*\otimes_{\laz}\laz[\barbi{b}]\stackrel{\otimes\psi}{\lrow}
\Omega^{*-n+m}\otimes_{\zz}\qq
$$
in degree $n$. Assume that $S_{\psi}$ satisfies the following integrality condition:
$S_{\psi}(\Omega^n((\pp^{\infty})^{\times r}))\subset \Omega^m((\pp^{\infty})^{\times r})$,
for all $r\geq 0$. Then there exists a unique additive operation $G_{\psi}:\Omega^n\row\Omega^m$ such that
$S_{\psi}=G_{\psi}\otimes\qq$. Moreover, every additive operation arises in this way, for a unique $\psi$.
Thus, $\psi\leftrightarrow G_{\psi}$ is
a 1-to-1 correspondence between linear combinations of Landweber-Novikov operations satisfying integrality conditions
and integral additive operations.
\end{theorem}

\begin{proof}
It follows immediately from Theorems \ref{unstUNIQ}  and \ref{MAIN}.
\Qed
\end{proof}

If $A^*$ is a theory of rational type, and $B^*$ is any theory in the sense of Definition \ref{goct},
then (unstable) additive operations $A^n\row B^m$ can be described as follows.

\begin{theorem}
\label{unstAB}
Let $A^*$ be a theory of rational type, and $B^*$ be any theory. Then
there is $1$-to-$1$ correspondence between the set of (unstable) additive operations $A^n\stackrel{G}{\row} B^m$ and the
set consisting of the following data $\{\gl{l},\,l\in\zz_{\geq 0}\}$:
$$
\gl{l}\in \op{Hom}_{\zz-lin}(A^{n-l},B[[z_1,\ldots,z_l]]_{(m)})\hspace{3mm}\text{satisfying:}
$$
\begin{itemize}
\item[$(a_{i})$ ] $\gl{l}$ is symmetric with respect to ${\frak{S}}_l$;
\item[$(a_{ii})$ ] $\gl{l}(\alpha)=\prod_{i=1}^lz_i\cdot\fl{l}(\alpha)$, for some $\fl{l}(\alpha)\in B[[z_1,\ldots,z_l]]_{(m-l)}$.
\item[$(a_{iii})$ ]
$
\gl{l}(\alpha)(x+_By,z_2,\ldots,z_l)=
\sum_{i,j}\gl{i+j+l-1}(\alpha\cdot a_{i,j}^A)(x^{\times i},y^{\times j},z_2,\ldots,z_l),
$
\\
where $a_{i,j}^A$ are the coefficients of the
formal group law of $A^*$.
\end{itemize}
\end{theorem}

\begin{proof}
It follows immediately from Proposition \ref{vazhnoe}, Theorem \ref{MAIN}, and the discussion
right after it.
\Qed
\end{proof}

With the data $\{\gl{l},\,l\in\zz_{\geq 0}\}$ one can associate the data
$\{\wgl{l},\,l\in\zz_{\geq 0}\}$,
where $\wgl{l}:A^{n-l}(k)\row B^{m-l}(k)$ is the constant term of the
$\fl{l}:A^{n-l}(k)\row B^*(k)[[z_1,\ldots,z_l]]_{(m-l)}$. We have:

\begin{proposition}
\label{wglGL}
If $B$ has no torsion, then $\{\wgl{l},\,l\in\zz_{\geq 0}\}$ carries the same information
as $\{\gl{l},\,l\in\zz_{\geq 0}\}$.
\end{proposition}

\begin{proof}
We can write $\fl{l}(\alpha)=\sum_{\ov{i}}h_{l,\ov{i}}(\alpha)\ov{z}^{\ov{i}}$, where $h_{l,\ov{i}}(\alpha)\in B^*(k)$.
Let us prove by induction on the degree of $\ov{i}$ (simultaneously for all $l$) that $h_{l,\ov{i}}$ is determined by
$h_{r,\ov{0}}$, for all $r$. The base is evident.
Let $\ov{z}^{\ov{i}}=z_1^{i_1}\cdot\ldots\cdot z_l^{i_l}$. Consider the equation $(a_{iii})$.
Compare the coefficients at $x^{i_1}yz_2^{i_2+1}\cdot\ldots\cdot z_l^{i_l+1}$. We may ignore the terms involving $h_{s,\ov{j}}$
corresponding to monomials of smaller degree since these are determined by $h_{r,\ov{0}}$ (for all $r$) by the inductive assumption.
Then, using the fact that $x+_By=x+y+\,\text{higher terms}$, we see that the only contributing term on the left is
$h_{l,\ov{i}}(\alpha)(x+y)^{i_1+1}z_2^{i_2+1}\cdot\ldots\cdot z_l^{i_l+1}$, while on the right no terms contribute.
(Note, that the terms of
$G_{l+i+j-1}(\alpha\cdot a^A_{i,j})$ for  $i,j\geq 1$ don't contribute, because the degree of each term of $G_r$ is greater than that
of the respective term of $F_r$ by $r$.) Thus, $(i_1+1)\cdot h_{l,\ov{i}}(\alpha)$ is
expressible in terms of $h_{r,\ov{j}}(\,something\,)$, for $|\ov{j}|<|\ov{i}|$, and all $r$.
Since $B$ has no torsion, and $i_1+1\neq 0$, $h_{l,\ov{i}}(\alpha)$ is determined by these smaller terms.
\Qed
\end{proof}

\begin{corollary}
\label{corBnontor}
Let $A^*$ be a theory satisfying $(CONST)$, and $B^*$ be any theory in the sense of Definition \ref{goct}
with $B$ torsion-free. Then an additive (unstable) operation $A^n\row B^m$ is determined by it's action
on the image of $(j_l)_*$, for all $l$, where $j_l:\spec(k)\row(\pp^1)^{\times l}$ is an embedding of a rational point.
\end{corollary}

\begin{proof}
This follows from Propositions \ref{vazhnoe} and \ref{wglGL}.
\Qed
\end{proof}

But if $B$ has torsion, then $\{\wgl{l},\,l\in\zz_{\geq 0}\}$ does not determine
$\{\gl{l},\,l\in\zz_{\geq 0}\}$.

\begin{example}
\label{wglCH}
Consider $A^*=B^*=\op{CH}^*/p$, $p$-prime. Then $F_B(x,y)=x+y$ is additive, and $A^{n-l}(k)=0$, for $l\neq n$.
Thus, $\gl{l}=0$, for $l\neq n$, and the only conditions on
$\gl{n}$ are: symmetry and additivity. Thus, $\gl{n}(z_1,\ldots,z_n)$ is an arbitrary symmetric
polynomial with $\zz/p$-coefficients of degree $m$ containing monomials where each $z_i$ enters in degree $p^{r_i}$, where $r_i\geq 0$. And $\wgl{n}$ is the coefficient at $z_1\cdot\ldots\cdot z_n$
(so, it is zero if $n\neq m$, and an element of $\zz/p$, if $n=m$). Of course, it does not determine
$\gl{n}$.
\end{example}

In the case of Chow groups modulo $p$ we can describe all the operations explicitly.
These appear to be essentially stable, and so expressible in term of Steenrod operations
(defined by V.Voevodsky \cite{VoOP} and P.Brosnan \cite{Br}).

\begin{theorem}
\label{unstCH}
Any additive operation $\op{CH}^n/p\row\op{CH}^m/p$ extends to a stable operation.
The $\ff_p$-vector space of such operations has a basis consisting of Steenrod operations
$S^{\ov{k}}$, where $\ov{k}=(k_1,\ldots,k_s)$ is a partition with $k_i=p^{r_i}-1$, $r_i\geq 0$,
$|\ov{k}|=(m-n)$, $s\leq n$.
\end{theorem}

\begin{proof}
In the example \ref{wglCH} we saw that any additive operation $\op{CH}^n/p\stackrel{G}{\row}\op{CH}^m/p$
is determined by some symmetric polynomial $\fl{n}(z_1,\ldots,z_n)$ of degree $(m-n)$, where
each variable $z_i$ enters in degree $p^{r_i}-1$, for some $r_i\geq 0$.
The value of $G$ on the class $x_n=\prod_{i=1}^nh_i\in\op{CH}^n((\pp^{\infty})^{\times n})$ is
equal to $x_n\cdot\fl{n}(h_1,\ldots,h_n)$, which coincides with the value of the (stable!) Steenrod operation
$S^{\fl{n}}$.
Since $\op{CH}^{n-l}(\spec(k))/p=0$, for $l\neq n$, these two operations $\op{CH}^n/p\row\op{CH}^m/p$ coincide on $\smk$, by Proposition \ref{vazhnoe}.
Clearly, the $\ff_p$-vector space of mentioned polynomials $\fl{n}$ has a basis consisting  of the symmetrizations of monomials corresponding to partitions as above.
\Qed
\end{proof}

\begin{remark}
In particular, Theorem \ref{unstCH} provides another construction of Steenrod operations
in Chow groups.
\end{remark}

Consider now $A^*=K_0$ which is a free theory with the formal group law $(\zz,x+y-xy)$. Then an additive operation
$G:K_0\row K_0$ is given by the collection of $\{\gl{l},\,l\in\zz_{\geq 0}\}$
satisfying $(a_i)-(a_{iii})$. In our case, the condition $(a_{iii})$ is:
$$
\gl{l}(u)(x+y-xy,\ov{z})=\gl{l}(u)(x,\ov{z})+\gl{l}(u)(y,\ov{z})-
\gl{l+1}(u)(x,y,\ov{z}).
$$
Hence, $\gl{l+1}$ can be expressed in terms of $\gl{l}$.
Thus, all $\gl{l},\,l\geq 1$ can be expressed in terms of $\gl{1}$.
Since $\zz$ is additively generated by $1$, everything
is determined by $g(z)=\gl{1}(1)(z)$. Moreover, if $g(z)\in\zz[[z]]\cdot z$ is any power series, then
we can define
$$
\gl{l}(1)(z_1,\ldots,z_l)=\sum_{I\subset\{1,\ldots,l\}}(-1)^{|I|-1}g({\sum_{i\in I}}^{K_0}z_i).
$$
These power series are clearly symmetric and satisfy $(a_{iii})$, and it is not difficult to see
that they satisfy $(a_{ii})$ (cf. the proof of Lemma \ref{lemflmp}).
How to describe the operation corresponding to $g(z)$?
Recall from Theorem \ref{Adams} that Adams operations $\Psi_k:K_0\row K_0$ are multiplicative (and so, additive) operations
with $\gamma_{\Psi_k}(z)=1-(1-z)^k$. And $\gamma_{\Psi_k}(z)$ is exactly the respective $\gl{1}(1)(z)$.
Consider additive operations $\Upsilon_k=\sum_{i=0}^k (-1)^{i-1}\binom{k}{i}\Psi_k$. The respective
$\gl{1}(1)(z)$ is just $z^k$. Thus, we obtain:

\begin{theorem}
\label{unstK}
Additive operations in $K_0$ are exactly all possible linear combinations
$\sum_{k\geq 0}\lambda_k\cdot\Upsilon_k$, where $\lambda_k\in\zz$.
\end{theorem}

\subsection{Multiplicative operations between theories of rational type}

The following result reduces the study of multiplicative operations on
theories of rational type to the study of morphisms of formal group laws (recall, that
such theories are in $1$-to-$1$ correspondence with the formal group laws).

\begin{theorem}
\label{multFGL}
Let $A^*$ be a free theory, and $B^*$ be any oriented cohomology theory.
The map sending the multiplicative operation $A^*\row B^*$ to the induced homomorphism of
formal group laws $(A^*(k),F_A)\row (B^*(k),F_B)$ is a bijection.
\end{theorem}

\begin{proof}
Any multiplicative operation $G$ defines the homomorphism $(\ffi_G,\gamma_G):(A^*(k),F_A)\row(B^*(k),F_B)$
of formal group laws. On the other hand, any homomorphism
$(\ffi,\gamma)$ of formal group laws defines the homomorphisms
$H:A^*((\pp^{\infty})^{\times r})\row B^*((\pp^{\infty})^{\times r})$ by the rule:
$$
H(f(z^A_1,\ldots,z^A_r)):=\ffi(f)(\gamma(z^B_1),\ldots,\gamma(z^B_r)),
$$
where
$f\in A[[z^A_1,\ldots,z^A_r]]=A^*((\pp^{\infty})^{\times r})$. Clearly, this homomorphisms
commute with the pull-backs for the action of ${\frak{S}}_r$, and for partial diagonals.
As for partial Segre embeddings, let $Seg=(Segre\times id^{\times (r-1)})$. Then we have:
\begin{equation*}
\begin{split}
&Seg^*f(z^A_1,\ldots,z^A_r)=f(F_A(x^A,y^A),z^A_2,\ldots,z^A_r),\hspace{2mm}\text{while}\\
&Seg^*\ffi(f)(\gamma(z^B_1),\ldots,\gamma(z^B_r))=
\ffi(f)(\gamma(F_B(x^B,y^B)),\gamma(z^B_2),\ldots,\gamma(z^B_r)).
\end{split}
\end{equation*}
Since $\ffi(F_A)(\gamma(x^B),\gamma(y^B))=\gamma(F_B(x^B,y^B))$, we get that our
homomorphisms commute with the pull-backs for Segre embeddings as well.
Commutativity with the morphisms
$(\op{Spec{k}}\hookrightarrow\pp^{\infty})\times(\pp^{\infty})^{\times r}$
follows from the fact that $\gamma$ has no constant term.
Thus, it extends to a unique operation $H:A^*\row B^*$.
Since our homomorphisms on $(\pp^{\infty})^{\times r}$ commute with the external
products of projective spaces, it follows from Proposition \ref{mainMULT} that the resulting operation
will be multiplicative.
Finally, Proposition \ref{vazhnoe} implies that the above two assignments are inverse to each other.
\Qed
\end{proof}

The situation here is simpler than in Topology. It is the reflection of the same phenomenon as the
fact that our {\it theories of rational type}
are in $1$-to-$1$ correspondence with the {\it formal group laws}.
In Topology, some special cases of the above result are known - see, for example,
\cite[Theorem 3.7]{BT}.

Consider now the case where $A^*=\Omega^*$.
We can extend the Theorem \ref{PSLM}.

\begin{theorem}
\label{neobrB0}
Let $B^*$ be an oriented cohomology theory. Let $\gamma=b_0x+b_1x^2+b_2x^3+\ldots\in B^*(k)[[x]]$ be a power series such that
$b_0\in B^*(k)$ is a non zero-divisor. Then there exists a multiplicative operation $G:\Omega^*\row B^*$
with $\gamma_G=\gamma$ if and only if the twisted formal group law $F_B^{\gamma}\in B^*(k)[b_0^{-1}][[x,y]]$
has coefficients in $B^*(k)$.
In this case, such an operation is unique.
\end{theorem}

\begin{proof}
Since $\ffi_G(F_{\Omega})(\gamma(x^B),\gamma(y^B))=\gamma(F_B(x^B,y^B))$, and
$\ffi_G(F_{\Omega})$ has coefficients in $B^*(k)$, the above condition is necessary.
On the other hand, if $F_B^{\gamma}$ has coefficients in $B^*(k)$,
by universality of the formal group law $(\laz,F_{\Omega})$, we get a ring homomorphism
$\ffi:\laz\row B^*(k)$ such that $\ffi(F_{\Omega})=F_B^{\gamma}$, and hence,
a morphism of formal group laws
$(\ffi,\gamma):(\laz,F_{\Omega})\row (B^*(k),F_B)$
which provides the needed operation by Theorem \ref{multFGL}.
\phantom{a}\hspace{5mm}
\Qed
\end{proof}

The above two results provide an effective tool in constructing multiplicative operations.
We will use them below to construct {\it Integral Adams Operations} and T.tom Dieck - style
Steenrod operations in Algebraic Cobordism.

Let us describe the morphisms of formal group laws (and so, the multiplicative operations between the
respective theories) in some situations.

For $r>1$, denote: $d(r):=G.C.D.(\binom{r}{i},0<i<r)$. Then
\begin{equation*}
d(r)=
\begin{cases}
p,\hspace{2mm}\text{if}\hspace{1mm}r=p^k,\hspace{1mm}\text{for some}\hspace{1mm}k;\\
1,\hspace{2mm}\text{otherwise}.
\end{cases}
\end{equation*}

\begin{lemma}
\label{d(r)}
Let $(\ffi,\gamma):(A,F_A)\row (B,F_B)$ be a morphism of formal group laws, where $\gamma=b_0x+b_1x^2+\ldots$.
Then either $b_0\neq 0$, or the first non-zero coefficient $b_{r-1}$ of $\gamma$ satisfies: $d(r)\cdot b_{r-1}=0$.
\end{lemma}

\begin{proof}
Suppose, $b_0=0$, and $b_{r-1}$ is the first non-zero coefficient of $\gamma$.
From the equality:
$$
\ffi(F_A)(\gamma(x),\gamma(y))=\gamma(F_B(x,y)),
$$
we get:
$b_{r-1}x^r+b_{r-1}y^r+\,\text{higher terms}\,=b_{r-1}(x+y)^r+\,\text{higher terms}$, which implies that $d(r)\cdot b_{r-1}=0$.\\
\phantom{a}
\Qed
\end{proof}

Suppose now that $B$ is an integral domain.
Then the characteristic $char(B)$ is either a prime $p$, or $0$. Consider these two cases separately.

1) $char(B)=0$:
\begin{corollary}
\label{b0zero}
Let $A^*$ and $B^*$ be any theories in the sense of Definition \ref{goct} with torsion-free $B^*(k)$, and $G:A^*\row B^*$ be a multiplicative operation. Then either
$\gamma_G=0$, or $b_0\neq 0$.
\Qed
\end{corollary}

We say that an operation is {\it of the main type}, if $b_0\neq 0$.
The respective power series $\gamma_G$ will be also called {\it of the main type}.

2) $char(B)=p$:

Let $B^*$ be a theory, where $B^*(k)$ is a ring of characteristic $p$. We can obtain a new theory
$\op{Fr}(B)^*$ from $B^*$ by the change of coefficients: $B^*(k)\stackrel{Fr}{\row}B^*(k)$, where $Fr$ is the Frobenius homomorphism.
In particular, $F_{\op{Fr}(B)}=Fr(F_B)$. We have natural multiplicative operation:
$\wt{\op{Fr}}:\op{Fr}(B)^*\row B^*$ defined by: $\wt{\op{Fr}}(u\otimes b)=u^p\cdot b$.
The respective morphism of formal group laws will be: $(id,x^p)$.

Let $G:A^*\row B^*$ be a multiplicative operation, and
$(\ffi_G,\gamma_G):(A^*(k),F_A)\row (B^*(k),F_B)$ be the respective morphism of formal group laws, such that
$\gamma_G(x)=b_{r-1}x^r+\ldots$, and $b_{r-1}\neq 0$. Then it follows from
Lemma \ref{d(r)}, that $r=p^k$, for some $k\geq 0$.

\begin{lemma}
\label{gammaFr}
In the above situation, $\gamma_G(x)=\delta(x^{p^k})$, for some $\delta\in B^*(k)[[y]]$
with $\delta_0\neq 0$.
\end{lemma}

\begin{proof}
We need to show that the degrees of all non-zero terms of $\gamma_G$ are divisible by
$p^k$. From the contrary, let $b_{s-1}$ be the smallest non-zero coefficient with $p^k\nmid s$.
Then, looking at the degree $s$ component of the equality:
$\ffi_G(F_A)(\gamma_G(x),\gamma_G(y))=\gamma_G(F_B(x,y))$,
we get: $b_{s-1}x^s+b_{s-1}y^s=b_{s-1}(x+y)^s$, which implies that $d(s)\cdot b_{s-1}=0$.
Here either $d(s)=1$, or $s$ is a prime power. Since $p^k\nmid s>p^k$, this must be a power
of some other prime.
But since $B^*(k)$ has characteristic $p$, this implies that $b_{s-1}=0$, in any case.
\Qed
\end{proof}

Thus, any such morphism $(\ffi_G,\gamma_G)$ of formal group laws can be presented as the composition
$$
(A^*(k),F_A)\xrightarrow{(\ffi_G,\delta)}(B^*(k),F_{\op{Fr}^k(B)})\xrightarrow{(id,x^{p^k})}(B^*(k),F_B).
$$
Return now to the situation where $A^*$ is a theory of rational type. Then the morphism
$(\ffi_G,\delta)$ of formal group laws defines a multiplicative operation
$H:A^*\row\op{Fr}^k(B)^*$, and we get that $G=\wt{\op{Fr}}^k\circ H$,
where $H$ is an {\it operation of the main type}.

Combining Theorem \ref{neobrB0} with the above considerations, we get:

\begin{theorem}
\label{OBmult}
Let $B^*$ be any theory in the sense of Definition \ref{goct} with
$B^*(k)$ - an integral domain. Then:
\begin{itemize}
\item[$1)$ ] If $char(B^*(k))=0$, then the assignment $G\mapsto\gamma_G$ provides a
$1$-to-$1$ correspondence between multiplicative operations $G:\Omega^*\row B^*$ and
such $\gamma=b_0x+\ldots\in B^*(k)[[x]]$ that either $\gamma=0$, or $b_0\neq 0$ and
$F_B^{\gamma}\in B^*(k)[b_0^{-1}][[x,y]]$ has coefficients in $B^*(k)$.
\item[$2)$ ] If $char(B^*(k))=p$, then the assignment $G\mapsto (k,\gamma_H)$, where
$G=\wt{\op{Fr}}^k\circ H$, with $H$ of the main type, provides a $1$-to-$1$ correspondence between multiplicative operations $G:\Omega^*\row B^*$ and pairs $(k,\gamma)$, where
either $(k,\gamma)=(\infty,0)$, or $k\in\zz_{\geq 0}$, and $\gamma=b_0x+\ldots\in B^*(k)[[x]]$
has $b_0\neq 0$, and $(\op{Fr}^k(F_B))^{\gamma}\in B^*(k)[b_0^{-1}][[x,y]]$ has coefficients
in $B^*(k)$.
\end{itemize}
\end{theorem}

One can compose the morphisms of formal group laws. Moreover, if $(\ffi,\gamma)$ and $(\ffi,\beta)$ have common
homomorphism of coefficient rings, we can also "add" such morphisms (just as one can add
morphisms into an abelian group). Namely, we can set:
$(\ffi,\beta)+(\ffi,\gamma)=(\ffi,\delta)$, where
$\delta(x)=\ffi(F_A)(\beta(x),\gamma(x))$.

In particular, if $A^*$ is a theory of rational type, and there exists only one ring endomorphism
$\ffi:A^*(k)\row A^*(k)$, then the set of multiplicative operations $G:A^*\row A^*$
has a natural ring structure with multiplication = the composition, and addition as above.
This happens for Chow groups, and for $K_0$.
In the case of $\op{CH}^*/p$, we get:

\begin{theorem}
\label{CHmultRING}
The ring of multiplicative operations $\op{CH}^*/p\row\op{CH}^*/p$ is
$\zz/p[[\wt{\op{Fr}}]]$. In particular, the composition is commutative.
\end{theorem}

\begin{proof}
Since there is only one ring homomorphism $\zz/p\row\zz/p$,
the multiplicative operations $\op{CH}^*/p\row\op{CH}^*/p$ are in $1$-to-$1$
correspondence with the additive power series $\gamma(x)=\sum_{r}b_{p^r-1}x^{p^r}$.
Moreover, $\op{Fr}(\op{CH}/p)^*=\op{CH}^*/p$, and $\wt{\op{Fr}}:\op{CH}^*/p\row\op{CH}^*/p$
is given by the power series $x^p$. The composition of operations corresponds to the composition
of $\gamma$'s, and addition is the usual addition of $\gamma$'s. Thus, our ring can be
naturally identified with $\zz/p[[\wt{\op{Fr}}]]$.
\Qed
\end{proof}

Under the identification above, the total Steenrod operation $\stT=id+S^1+S^2+\ldots$
corresponds to $1+\wt{\op{Fr}}$, and the Integral Adams Operation $\Psi_k$ (see below)
corresponds to $k$. In particular, $\Psi_0$ which is the identity on $\op{CH}^0$ and zero
on $\op{CH}^i,\,i>0$ corresponds to $0$.

\subsection{Integral Adams Operations}

Adams operations $\Psi_k$ provide an important tool in studying $K$-groups.
In topology, analogous operations were constructed by S.P.Novikov for complex-oriented
cobordisms $MU$ in \cite{N}. This construction required inverting $k$, since $\Psi_k$
were basically expressed in terms of Landweber-Novikov operations, and the respective
formulas do have $k$-denominators. Only much later it was shown by W.S.Wilson that
these operations can be defined integrally and are naturally multiplicative unstable
operations - see \cite[Theorem 11.53]{Wi}.
Using our main results we can construct similar operations in Algebraic Cobordism
and all other theories of rational type
(it is worth noting, that although we produce a similar object, our methods are completely different as we are working with the theories themselves, not with spectra).

\begin{theorem}
\label{Adams}
For any free theory $A^*$, there are multiplicative $A^*(k)$-linear operations
$\Psi_k:A^*\row A^*$, $k\in\zz$, such that $\gamma_{{\Psi_k}}=[k]\cdot_{A}x$.
These operations do not depend on the choice of orientation of $A^*$.
In the case of $K_0$ these are the usual Adams operations.
\end{theorem}

\begin{proof}
Consider $\gamma_k=[k]\cdot_{A}x$. Since $(id,\gamma_k)$,
is an endomorphism of the formal group law $(A^*(k),F_A)$, by Theorem \ref{multFGL},
we get a unique multiplicative operation $\Psi_k:A^*\row A^*$ with such $\gamma$.
Since $\ffi_{\Psi_k}=id$, this operation is $A^*(k)$-linear.

Finally, by \cite{PS}, the reorientation (change of push-forward structure) of the theory $A^*$ corresponds to
the choice of a generator $\beta(x)$ of the power series ring $A^*[[x]]$, so that for the new twisted theory $\wt{A}^*$
one has $c^{\wt{A}}_1(M)=\beta(c^A_1(M))$, for any line bundle $M$. And our operation $\Psi_k$ is characterized by
the property that $\Psi_k(c^A_1(L))=c^A_1(L^{\otimes k})$ which is obviously stable under reorientation, since $\Psi_k(\beta)=\beta$.

The fact that for $K_0$ these are the classical Adams operations follows from the definition of the latter.
\phantom{a}\hspace{5mm}
\Qed
\end{proof}

As the above operations are $A^*(k)$-linear they can be obtained from the ones in
Algebraic Cobordism by change of coefficients.

All Adams operations define the same endomorphism of the coefficient ring equal to the
identity, and so form a ring $R_{\Psi,A}$. Clearly, $\Psi_k$ is just the image of $k$
under the canonical surjective ring homomorphism $\zz\twoheadrightarrow R_{\Psi,A}$.
The operation $\Psi_0$ can be described as follows: it acts as $id$ on constant elements,
and as zero on $\ov{A}^*$. Thus, it is responsible for the decomposition which we used throughout
the paper.

Adams operations can be used in the study of the graded Algebraic Cobordism
(see \cite[Subsection 4.5.2]{LM} for the definition of the latter).
Being operations, they respect the codimension of support of an element:
$\Psi_k(F^{(n)}\Omega^*(X))\subset F^{(n)}\Omega^*(X)$, and so act on the
graded ring $Gr^*\Omega^*(X)$.
We have the natural surjection:
$$
\op{CH}^*\otimes_{\zz}\laz^*\twoheadrightarrow Gr^*\Omega^*,
$$
which commutes with the action of $\Psi_k$ (recall, that these operations
are $\laz$-linear). Thus, $\Psi_k|_{Gr^n\Omega^*}$ is the multiplication by $k^n$.
Suppose now, $X\stackrel{f}{\row}Y$ is a morphism of smooth varieties.
Then we get the morphism of the respective filtrations:
$f^*:F^{(n)}\Omega^*(Y)\row F^{(n)}\Omega^*(X)$.
This provides the spectral sequence computing $\kker$ and $\coker$ of $f^*$:
$$
E^{p,q,n}_r\Rightarrow \op{H}^p(f^*:\Omega^{q}(Y)\row\Omega^{q}(X)),
$$
where $E^{p,q,n}_2=\op{H}^p(Gr(f)^*:Gr^n\Omega^q(Y)\row Gr^n\Omega^q(X))$,
$p=0,1$, and $d_r:E^{0,q,n}_r\row E^{1,q,n+r-1}_r$.
Adams operations permit to estimate the exponent of $d_r$.
Denote: $e(n,r)=\op{G.C.D.}(k^n(k^{r-1}-1),\,k\in\zz)$.

\begin{proposition}
\label{AdamsDr}
$e(n,r)\cdot d_r|_{E^{0,q,n}_r}=0$.
\end{proposition}

\begin{proof}
Since Adams operations respect the filtration, they act on the spectral sequence.
Then $\Psi_k$ must act as multiplication by $k^n$ on $E^{p,q,n}_r$. Since
$d_r: E^{0,q,n}\row E^{1,q,n+r-1}$, we get that, for any $k$,
$k^n(k^{r-1}-1)$ multiplied by such $d_r$ is zero.
\Qed
\end{proof}

It is easy to see that $e(0,r)=1$, and $e(n,2s)=2$, for all $n,s\geq 1$.
And prime factors of $e(n,r)$ are exactly those $p$ for which $(p-1)|(r-1)$.
In particular, these do not depend on $n$. But the powers of these primes do.
Thus, the "unstable information" is concentrated in these powers.

In particular, the above considerations apply to the extension of fields morphism.

\subsection{Symmetric Operations for all primes, and T.tom Dieck - style Steenrod operations}
\label{SymSt}
These topics represent the main content of the paper \cite{SPso}.
Here we just present briefly the main results and ideas.
The construction of Symmetric Operations for all primes was the main motivation behind
the current paper. For about 5 years the author tried to construct them, until he realized
that it is about as simple as constructing all unstable operations in
Algebraic Cobordism.
But let me start with the Steenrod operations.

Steenrod operations provide an important structure on $\op{CH}^*/p$ which permits to
do more elaborate tricks with algebraic cycles than the usual addition and multiplication.
Individual Steenrod operations can be organized into "larger" multiplicative operations.
One of the possible approaches is to consider the multiplicative operation:
$St:\op{CH}^*/p\row\op{CH}^*/p[[t]]$ given by the morphism of the respective formal group laws
(see Theorem \ref{multFGL}):
$(\ffi,\gamma)$, where $\ffi:\zz/p\row\zz/p[[t]]$ is the natural embedding
(the unique morphism of rings), and $\gamma=-t^{p-1}x+x^p$ (notice, that our $\gamma$
is additive in $x$).
Then the individual Steenrod operation $S^r|_{\op{CH}^m/p}$
will be the coefficient of $S|_{\op{CH}^m/p}$ at $t^{(m-r)(p-1)}$.
At the first glance it looks like we complicate things by making our operation
unstable (the coefficient at $x$ is not $1$), but it appears to be convenient
in various respects.

The original approach to Steenrod operations in Chow groups due to P.Brosnan
(see \cite{Br})
is through $\zz/p$-equivariant Chow groups. In this construction, one
produces the multiplicative operation  $Sq:\op{CH}^*(X)/p\row\op{CH}^*(X)/p\otimes_{\zz/p}\op{CH}^*(\op{B}\zz/p)/p$.
We have $\op{CH}^*(\op{B}\zz/p)/p=\zz/p[[t]]$, and one can show (see \cite{Br})
that the only non-trivial coefficients of $Sq$ will be at $t^{r(p-1)}, \, r\geq 0$.
The fact that the two constructions agree follows from
Theorem \ref{multFGL} (the morphism of formal group laws for $Sq$ is easy to compute).

All of the above was known in topology for quite a while. And both mentioned constructions
were extended to complex-oriented cobordism $MU$.
The equivariant version is due to T. tom Dieck (\cite{tD}), and it goes completely parallel
to the $H^*/p$ (and $\op{CH}^*/p$) case. Here
$MU^*(X\times\op{B}\zz/p)=MU^*(X)[[t]]/([p]\cdot_{MU}(t))$,
and one gets a multiplicative operation
$$
Sq:MU^*(X)\row MU^*(X\times\op{B}\zz/p)\row
MU^*(X)[[t]]/(\textstyle\frac{[p]\cdot_{MU}(t)}{t}).
$$

The other version is due to D.Quillen
(\cite{Qu71}). One observes that \\
$-t^{p-1}x+x^p\equiv\prod_{i=0}^{p-1}(x+it)\hspace{3mm}(\,mod\,p).$
Now we can produce an $MU$-analogue of this power series:
$\gamma=\prod_{i=0}^{p-1}(x+_{MU}[i]\cdot_{MU}t)\in\laz[[t]][[x]]$,
which by universality of $MU^*$ defines the multiplicative operation:
$$
St: MU^*\row MU^*[(p-1)!^{-1}][[t]][t^{-1}].
$$
Notice, that this time, we have to invert $t$
and $(p-1)!$,
since the shifted formal group law  $F_{MU[[t]]}^{\gamma}$ has denominators.
Also, $St$ has non-trivial coefficients at $t^j$, for $j$ not divisible
by $(p-1)$. It was shown by D.Quillen that his approach agrees with the one of T. Tom Dieck.
More precisely, one has the following commutative diagram:
$$
\xymatrix @-0.2pc{
MU^* \ar @{->}[r]^(0.3){St} \ar @{->}[d]_(0.5){Sq}&
MU^*[(p-1)!^{-1}][[t]][t^{-1}] \ar @{->}[d]^(0.5){}\\
MU^*[[t]]/(\frac{[p]\cdot_{MU}(t)}{t}) \ar @{->}[r]_(0.5){} &
MU^*[[t]][t^{-1}]/([p]\cdot_{MU}(t)).
}
$$

Let us try to extend these constructions to the case of Algebraic Cobordism $\Omega^*$.
The Quillen's version is completely straightforward. Here one needs only the universality
of $\Omega^*$ supplied by M.Levine-F.Morel (\cite[Theorem 1.2.6]{LM})
and the change of orientation of I.Panin-A.Smirnov (\cite{PS}).
Let us do a more general case
(suggested by D.Quillen). Namely, chose representatives $\{i_j,\,0<j<p\}$ of
all non-zero cosets modulo $p$, and denote $\iis:=\prod_{j=1}^{p-1}i_j$.
Then we can consider the power series
$\gamma=\prod_{j=0}^{p-1}(x+_{\Omega}[i_j]\cdot_{\Omega}t)\in\laz[[t]][[x]]$,
which, by Theorem \ref{PSLM}, defines the multiplicative operation
$$
St(\ov{i}):\Omega^*\row\Omega^*[\iis^{-1}][[t]][t^{-1}].
$$

The situation with the version of T.tom Dieck is rather different. Although one can easily
define the $\zz/p$-equivariant Algebraic Cobordism $\Omega_{\zz/p}^*(X)$, one encounters
problems trying to prove that the natural map $\Omega^n(X)\row\Omega^{np}_{\zz/p}(X^{\times p})$
is well-defined. It is easy to show that the standard cobordism relations are respected,
but the author was unable to handle the double-points relations.
The only case where the author succeeded was $p=2$, where he had to employ the Symmetric
Operations (modulo $2$) constructed in \cite{so1}, \cite{so2}. These operations, which are more
subtle than the Steenrod ones, until now were unavailable for $p>2$.

Fortunately, our Theorem \ref{multFGL} permits to construct what we need.

\begin{theorem} {\rm (\cite[Theorem 6.4]{SPso})}
\label{TTD}
There is the multiplicative operation $Sq$ which fits into the commutative diagram:
$$
\xymatrix @-0.2pc{
\Omega^* \ar @{->}[r]^(0.4){St(\ov{i})} \ar @{->}[d]_(0.5){Sq}&
\Omega^*[\iis^{-1}][[t]][t^{-1}] \ar @{->}[d]^(0.5){}\\
\Omega^*[[t]]/(\frac{[p]\cdot_{\Omega}(t)}{t}) \ar @{->}[r]_(0.5){} & \Omega^*[[t]][t^{-1}]/([p]\cdot_{\Omega}(t)).
}
$$
\end{theorem}

Notice, that $Sq$ is a bit more "canonical" than $St$ - it does not depend on $\ov{i}$.

Now, since the target of $Sq$ has no negative powers of $t$,
the commutativity of the above diagram shows that the {\it negative part}
of $St(\ov{i})$ is divisible by $\frac{[p]\cdot_{\Omega}t}{t}$.
I should point out that this fact itself does not require the above Theorem,
or the methods of the current paper. But what is much deeper,
it appears that one can divide "canonically", and the quotient is what we call
{\it Symmetric operation}.

\begin{theorem} {\rm (\cite[Theorem 7.1]{SPso})}
\label{SOp}
There is a unique additive operation $\Phi(\ov{i}):\Omega^*\row\Omega^*[\iis^{-1}][t^{-1}]t^{-1}$ such that
$$
(St(\ov{i})+\textstyle\frac{[p]\cdot_{\Omega}t}{t}\cdot\Phi(\ov{i})):
\Omega^*\row\Omega^*[\iis^{-1}][[t]].
$$
\end{theorem}

Some traces of the $MU$-analogue of this operation were used by D.Quillen in \cite{Qu71},
and they provide the main tool of the mentioned article.

In Algebraic Cobordism the described operation appeared originally in the works \cite{so1}
and \cite{so2} of the author in the case $p=2$ in a different form. Namely, in the form of "slices",
which were constructed geometrically.
Only substantially later the author had realized that these slices can be combined
into the "formal half" of the "negative part" of some multiplicative operation, which
had a power series $\gamma=x\cdot(x-_{\Omega}t)$ reminiscent of a Steenrod operation
in Chow groups mod $2$.
How to view the operation $\Phi(\ov{i})$? The natural approach would be to consider the
coefficients of it at particular monomials $t^{-n}$, or, equivalently,
$\operatornamewithlimits{Res}_{t=0}\frac{t^n\cdot\Phi(\ov{i})\omega_t}{t}$, for all $n$.
And, if one thinks about it, there is no point restricting oneself to monomials, so
one can consider
$$
\Phi(\ov{i})^{q(t)}:=
\operatornamewithlimits{Res}_{t=0}\frac{q(t)\cdot\Phi(\ov{i})\omega_t}{t},
$$
where $q(t)=q_1t+q_2t^2+\ldots\in\laz[[t]]$ is any power series without the constant term.
Of course, there are various relations among these slices which bind them together
into something "larger" - the operation $\Phi(\ov{i})$.
For $p=2$, these are exactly the Symmetric operations $\Phi^{q(t)}$ of \cite{so2}:

\begin{proposition} {\rm (\cite[Proposition 7.2]{SPso})}
\label{SOoldnew}
In the case $p=2$, with $\ov{i}=\{-1\}$,
for any power series as above, we have:
$$
\Phi(\ov{i})^{q(t)}=\Phi^{q(t)}.
$$
\end{proposition}

Notice, that for $p=2$, there is, in addition, a non-additive operation $\Phi^1$
(see \cite{so2}).
The methods of the given article don't permit to produce it's analogues for $p>2$ as here we are restricted to additive operations only.\footnote{The non-additive case was done in the next paper of the author - see \cite{PO}, and the operation was constructed.}
Fortunately, additive Symmetric Operations
are sufficient for most applications.

The cases $p=2$ and $3$ are special, since we can choose our representatives $\ov{i}$ to
be invertible in $\zz$. For $p=2$, we have two such choices: $\{1\}$, or $\{-1\}$ (in \cite{so2},
$\{-1\}$ was "chosen"). For $p=3$, the choice is canonical: $\{1,-1\}$.
Thus, we get integral operations $\Phi(\ov{i}):\Omega^*\row\Omega^*[t^{-1}]$.
And, for arbitrary $p$, we can choose our remainders to be the powers of some fixed prime $l$
(generating $(\zz/p)^*$), so that only one prime would be inverted. Moreover, this prime
can be chosen in infinitely many ways, so, in a sense, the picture is as good as integral.

For $p=2$ the Symmetric operations were applied to the study of $2$-torsion effects
in Chow groups - they provide the only known method to get "clean results" on rationality -
see \cite{GPQCG} and \cite{RIC}. And similar applications are expected for other primes.
Other applications involve the study of the structure of the $\laz$-module $Gr\Omega^*(X)$.
Here the construction of Symmetric Operations for all primes changes the statements $\otimes\zz_{(2)}$
into integral ones.

\section{Basic tools}
\label{basictools}

Here we present various results which permit to
work effectively with cohomology theories.

\subsection{Projective bundle and blow-up results}
\label{pb-bu}

We start with the {\it excess intersection formula} - see \cite[Theorem 5.19]{so2} and
\cite[Theorem 6.6.9]{LM}.
Consider cartesian square
$$
\begin{CD}
W&@>{f'}>>&Z\\
@V{g'}VV&&@VV{g}V\\
Y&@>>{f}>&X
\end{CD}
$$
with $f,f'$ - regular embeddings, and
$(g')^*(N_{Y\subset X})/N_{W\subset Z}=M$ the vector bundle
of dimension $d$.

\begin{proposition}
\label{excess}
Let $A^*$ be a theory in the sense of Definition \ref{goct}.
In the above situation,
$$
g^*f_*(v)=f'_*(c^A_d(M)\cdot (g')^*(v));
$$
If $g$ is projective, then also:
$$
f^*g_*(u)=g'_*(c^A_d(M)\cdot (f')^*(u)).
$$
\end{proposition}

\begin{proof}
Both of the above references are dealing with the $\Omega^*$-case.
Although, the statement of \cite[Theorem 6.6.9]{LM} is more general, it
requires the development of the whole theory of refined pull-backs. For Algebraic
Cobordism such a theory is constructed in \cite{LM}, but it requires some
work to extend it to a more general context. In contrast, the proof of \cite[Theorem 5.19]{so2}
does not use any specifics of $\Omega^*$ and works in general.
\Qed
\end{proof}

Another important tool is the formula of Quillen - see \cite[Theorem 1]{Qu69}, \cite[Formula (24)]{PS}, and
\cite[Theorem 5.35]{so2}. It describes push-forwards for projective bundles.

Recall that, for an $n$-dimensional vector bundle $W$, the {\it roots} are elements $\lambda_i\in A^1(Flag_X(W)), i=1,\ldots n$ such that $\prod_{i=1}^n(t+\lambda_i)=\sum_{i=0}^nc^A_i(W)t^{n-i}$,
where $Flag_X(W)$ is a variety of complete flags of $W$, and $c^A_i(W)$ are Chern classes
in the theory $A^*$. The important point here is that the pull-back map $A^*(X)\row A^*(Flag_X(W))$
is split injective.

Recall also that $\omega_A\in A^*(k)[[x]]dx$ is the canonical invariant 1-form satisfying: $w_A(0)=dx$.
Such a form can be obtained from the formal group law $F_A(x,y)$ of $A^*$ by the formula:
$\omega_A=\left(\frac{\partial F_A}{\partial y}|_{y=0}\right)^{-1}dx$.
By the formula of Mistchenko it can be expressed as:
$$
\left([\pp^0]_A+[\pp^1]_A\cdot x+[\pp^2]_A\cdot x^2+\ldots \right)dx,
$$
where $[\pp^r]_A$ is the class of $\pp^r$ in $A^*(k)$.

\begin{proposition}
\label{Quillen}
Let $A^*$ be a theory in the sense of Definition \ref{goct}. Let $X$ be some smooth quasi-projective
variety, $W$ be some $n$-dimensional vector bundle on it, and $\pi:\pp_X(W)\row X$ be the corresponding
projective bundle. Let $f(t)\in A^*(X)[[t]]$, and $\xi=c_1^A(O(1))$. Then
$$
\pi_*(f(\xi))=\operatornamewithlimits{Res}_{t=0}
\frac{f(t)\cdot\omega_A}{\prod_i(t+_A\lambda_i)},
$$
where $\lambda_i$ are roots of $W$, and $+_A$ is the formal addition in the sense of $F_A$.
\end{proposition}

\begin{proof}
Clearly, both parts of the formula are $A^*(X)$-linear, so it is sufficient to
prove the result in the case: $f(t)=t^r$ - a monomial. Then it formally
follows from the $\Omega^*$-case proven in \cite[Theorem 5.35]{so2}
(using the universality of $\Omega^*$ - \cite[Theorem 1.2.6]{LM}).
\Qed
\end{proof}

We will need various results concerning the blow up morphism.

Let $X$ be a smooth variety, $R$ be a smooth closed subvariety, $\wt{X}=Bl_{R}X$ - the blow up of
$X$ at $R$, and $E$ - the exceptional divisor on $\wt{X}$. These fit into the blow-up diagram:
$$
\xymatrix @-0.7pc{
E \ar @{->}[r]^(0.5){j} \ar @{->}[d]_(0.5){\eps}&
\wt{X} \ar @{->}[d]^(0.5){\pi}\\
R \ar @{->}[r]_(0.5){i} & X.
}
$$
Let $N$ be the normal bundle of $R$ in $X$, then $E\cong\pp_R(N)$. Let $d=\ddim(N)=\op{codim}(R\subset X)$.
Denote $\wt{N}=N\oplus O(1)$, and
$\wt{E}=\pp_R(\wt{N})\stackrel{\wt{\eps}}{\row}R$.

For projective bi-rational morphisms we have:

\begin{proposition}
\label{pi1invert}
Let $A^*$ be a theory in the sense of Definition \ref{goct}, and $\pi:\wt{X}\row X$
be projective bi-rational morphism of smooth varieties. Then
\begin{itemize}
\item[$(1)$ ] $\pi_*(1)$ is invertible in $A^*(X)$.
\item[$(2)$ ] $\pi_*:A_*(\wt{X})\row A_*(X)$ is surjective.
\end{itemize}
\end{proposition}

\begin{proof}
By universality of $\Omega^*$ (\cite[Theorem 1.2.6]{LM}), we have the canonical map
of theories $\Omega^*\row A^*$, and $\pi_*(1)$ is in the image of this map. So, it is
sufficient to treat the case $A^*=\Omega^*$. Since $\pi$ is bi-rational, we have a
closed subscheme $Z\subset X$ of positive codimension, such that $\pi$ is an
isomorphism outside $Z$. Then $\pi_*(1)=1+u$, where $u$ is supported on $Z$.
That means that $u$ has positive codimension of support, and so is nilpotent by
\cite[Statement 5.2]{so2}. Hence, $\pi_*(1)$ is invertible. It remains to apply the
projection formula.
\Qed
\end{proof}

In a more specific situation, the following result of M.Levine and F.Morel describes the class of the blow up in the 
$A^*$ of the base explicitly.

\begin{proposition} {\rm (\cite[Proposition 2.5.2]{LM})}
\label{blowupformula}
Let $A^*$ be a theory in the sense of Definition \ref{goct}. Then
$$
\pi_*(1)=1+i_*\wt{\eps}_*\left(\frac{c_1^A(O(1))}{c_1^A(O(-1))}\right).
$$
\end{proposition}

Here under $\frac{c_1^A(O(1))}{c_1^A(O(-1))}$ we mean $g(c_1^A(O(1)))$, where $g(t)=\frac{t}{[-1]\cdot_At}\in A^*(k)(t)$.

The following result describes what happens to the whole $A^*$ when you blow up some smooth variety at a smooth center.

\begin{proposition} {\rm (cf.\cite[Proposition 5.24]{so2})}
\label{razd}
Let $A^*$ be a theory in the sense of Definition \ref{goct}. Then we have split exact sequences:
\begin{equation*}
(1)\hspace{1.5cm}
0\llow A_*(X)\stackrel{\pi_*,-i_*}{\llow}A_*(\wt{X})\oplus A_*(R)
\stackrel{j_*,\eps_*}{\llow}A_*(E)\llow 0;\hspace{3cm}\phantom{a}
\end{equation*}
\begin{equation*}
(2)\hspace{1.5cm}
0\lrow A^*(X)\stackrel{\pi^*,-i^*}{\lrow}A^*(\wt{X})\oplus A^*(R)
\stackrel{j^*,\eps^*}{\lrow}A^*(E)\lrow 0.\hspace{3cm}\phantom{a}
\end{equation*}
\end{proposition}

\begin{proof}
In the case $A^*=\Omega^*$, (1) was proven in \cite[Proposition 5.24]{so2}, and the same proof
works for arbitrary $A^*$. Let us recall some details.
Let $K=\eps^*N/O(-1)$ be the excess bundle on $E$.
It is easy to see
(see \cite[Proposition 5.22]{so2}) that the class of the diagonal on $E\times_R E$ is given by
$c^A_{d-1}(K_1\otimes O(1)_2)$, where $V_l$ denotes the bundle $V$ lifted from the $l$-th component.
This class can be written as $c^A_{d-1}(K)\times 1 +\sum_{i\geq 1}\gamma_{d-1-i}\times\zeta^i$,
where $\gamma_j\in A^j(E)$ are some elements, and $\zeta=c^A_1(O(-1))$.
Let us introduce the elements $\alpha:=c^A_{d-1}(K)\in A^*(E)$, and
$\beta:=\wt{\eps}_*\left(\frac{c^A_1(O(1))}{c^A_1(O(-1))}\right)\in A^*(R)$.
Then for any $u\in A^*(E)$, we have:
\begin{equation*}
\begin{split}
u=\alpha\cdot\eps^*\eps_*(u)+\sum_{j\geq 1}
\gamma_{d-1-j}\cdot\eps^*\eps_*(u\cdot\zeta^j);\\
u=\eps^*\eps_*(u\cdot\alpha)+\sum_{j\geq 1}
\zeta^j\cdot\eps^*\eps_*(u\cdot\gamma_{d-1-j}),
\end{split}
\end{equation*}
Consider the maps $F:A_*(E)\row A_*(E)$ and $G:A^*(E)\row A^*(E)$ given by:
$$
F(u)=\sum_{j\geq 0}\gamma_{d-2-j}\cdot\eps^*\eps_*(u\cdot\zeta^j);\hspace{1cm}
G(u)=\sum_{j\geq 0}\zeta^j\cdot\eps^*\eps_*(u\cdot\gamma_{d-2-j}).
$$
Consider the diagram:
$$
\xymatrix @-0.7pc{
E\times_R E \ar @{->}[r]^(0.5){id\times e} \ar @{->}[dr]_(0.5){p_1}& E\times_R\wt{E} \ar @{->}[r]^(0.7){\rho} \ar @{->}[d]_(0.5){\wt{\rho}}&
\wt{E} \ar @{->}[d]^(0.5){\wt{\eps}}\\
& E \ar @{->}[r]_(0.5){\eps} & R.
}
$$
Let $E\stackrel{e}{\row}\wt{E}$ be the natural embedding. Then $e_*(1)=c^A_1(O(1))$.
We get:
\begin{equation*}
\begin{split}
F(1)&=(p_1)_*\left(\frac{c^A_{d-1}(K_1\otimes O(1)_2)-c^A_{d-1}(K_1)}{c^A_1(O(-1)_2)}\right)\\
&=\wt{\rho}_*\left((c^A_{d-1}(K_1\otimes O(1)_2)-c^A_{d-1}(K_1))\cdot
\frac{c_1^A(O(1)_2)}{c_1^A(O(-1)_2)}\right)\\
&=\wt{\rho}_*\left(c^A_{d-1}(K_1\otimes O(1)_2)\frac{c_1^A( O(1)_2)}{c_1^A( O(-1)_2)}\right)-
c^A_{d-1}(K)\cdot\wt{\rho}_*\rho^*\left(\frac{c_1^A(O(1))}{c_1^A(O(-1))}\right)\\
&=\operatornamewithlimits{Res}_{t=0}
\frac{c^A_{\bullet}(K)(t)\cdot t\cdot\omega_A}{c^A_{\bullet}(\wt{N})(t)\cdot(-_At)}-
c^A_{d-1}(K)\cdot\eps^*\wt{\eps}_*\left(\frac{c_1^A(O(1))}{c_1^A(O(-1))}\right)\\
&=\operatornamewithlimits{Res}_{t=0}
\frac{\omega_A}{(t+_A\xi)(-_At)}-\alpha\cdot\eps^*(\beta)=-\alpha\cdot\eps^*(\beta).
\end{split}
\end{equation*}

Now we can construct contracting homotopies $\lambda$ and $\mu$ for $(1)$ and $(2)$:
$$
\xymatrix @-0.3pc{
A_*(E) \ar @/^0.5pc/ @{->}[r]^(0.5){d_2} \ar @/^0.5pc/ @{->}[d]^(0.5){d_1}&
A_*(\wt{X}) \ar @/^0.5pc/ @{->}[d]^(0.5){d_4} \ar @/^0.5pc/ @{->}[l]^(0.5){\lambda_2} \\
A_*(R) \ar @/^0.5pc/ @{->}[r]^(0.5){d_3} \ar @/^0.5pc/ @{->}[u]^(0.5){\lambda_1}&
A_*(X) \ar @/^0.5pc/ @{->}[l]^(0.5){\lambda_3}
\ar @/^0.5pc/ @{->}[u]^(0.5){\lambda_4},
}
\hspace{15mm}
\xymatrix @-0.3pc{
A^*(E) \ar @/^0.5pc/ @{->}[r]^(0.5){\mu_2} \ar @/^0.5pc/ @{->}[d]^(0.5){\mu_1}&
A^*(\wt{X}) \ar @/^0.5pc/ @{->}[d]^(0.5){\mu_4} \ar @/^0.5pc/ @{->}[l]^(0.5){d_2} \\
A^*(R) \ar @/^0.5pc/ @{->}[r]^(0.5){\mu_3} \ar @/^0.5pc/ @{->}[u]^(0.5){d_1}&
A^*(X) \ar @/^0.5pc/ @{->}[l]^(0.5){d_3}
\ar @/^0.5pc/ @{->}[u]^(0.5){d_4},
}
$$
in the following way: $\lambda_4=\pi^*$; $\lambda_3=\beta\cdot i^*$,
$\lambda_1=\alpha\cdot\eps^*$; and $\lambda_2=F\circ j^*$, while
$\mu_4=\pi_*$; $\mu_3=i_*(\beta\cdot\phantom{u})$,
$\mu_1=\eps_*(\alpha\cdot\phantom{u})$; and $\mu_2=j_*\circ G$.

From the equality $F(1)=-\alpha\cdot\eps^*(\beta)$ (using several times the projection formula) one easily obtains the left ones of the following identities:
\begin{equation}
\label{A}
\lambda_2\circ\lambda_4=-\lambda_1\circ\lambda_3;\hspace{5mm}
d_2\circ\lambda_1=-\lambda_4\circ d_3;
\end{equation}
\begin{equation}
\label{B}
\mu_4\circ\mu_2=-\mu_3\circ\mu_1;\hspace{5mm}
\mu_1\circ d_2=-d_3\circ \mu_4,
\end{equation}
while the right ones are the Excess Intersection Formula (Proposition \ref{excess}).
The identity: $d_3\circ\lambda_3+d_4\circ\lambda_4=id_{A_*(X)}$ is just the Proposition \ref{blowupformula}
(plus the projection formula).
The identity: $d_1\circ\lambda_1+\lambda_3\circ d_3=id_{A_*(R)}$ follows from the Excess Intersection Formula
and Proposition \ref{blowupformula}.
The identity: $\lambda_1\circ d_1+\lambda_2\circ d_2=id_{A_*(E)}$ follows from the definition of $F$.
Finally, the identity: $d_2\circ\lambda_2+\lambda_4\circ d_4=id_{A_*(\wt{X})}$ follows from the ones already
proven, plus (\ref{A}), plus the fact that the map
$(j_*,\pi^*):A_*(E)\oplus A_*(X)\row A_*(\wt{X})$ is surjective, which follows from the $(EXCI)$ axiom
(see the proof of \cite[Proposition 5.24]{so2}).

The identity: $\mu_3\circ d_3+\mu_4\circ d_4=id_{A^*(X)}$ follows from Proposition \ref{blowupformula},
and the projection formula.
The identity: $\mu_1\circ d_1+d_2\circ\mu_2=id_{A^*(R)}$ follows from
the Excess Intersection Formula and Proposition \ref{blowupformula}.
The identity: $d_1\circ\mu_1+ d_2\circ\mu_2=id_{A^*(E)}$ follows from the definition of $G$.
Finally, the identity: $\mu_2\circ d_2+d_4\circ\mu_4=id_{A^*(\wt{X})}$ follows the ones already
proven, plus (\ref{B}), plus the fact that the map
$(j^*,\pi_*):A^*(\wt{X})\row A^*(E)\oplus A^*(X)$ is injective, which follows from the fact
that $\lambda$ is a contracting homotopy for the complex $(1)$.
\Qed
\end{proof}

In the case of multiple blow-ups we get:

\begin{proposition}
\label{vvter}
Let $A^*$ be a theory in the sense of Definition \ref{goct},
and $\pi:\widetilde{V}\row V$ be a sequence of blowups of a smooth variety in smooth centers
$R_i$.  Let $\eps_i:E_i\row R_i$ be the respective components of the exceptional divisor
(that is, $E_i$ is the strict transform of the exceptional divisor of the blowup of $R_i$).
Then one has exact sequences:
$$
(1)\hspace{1.5cm}
0\low A_*(V)\stackrel{\pi_*}{\llow}A_*(\widetilde{V})\llow
\oplus_i\kker(A_*(E_i)\stackrel{(\eps_i)_*}{\row}A_*(R_i)).\hspace{3cm}\phantom{a}
$$
$$
(2)\hspace{1.5cm}
0\row A^*(V)\stackrel{\pi^*}{\lrow}A^*(\widetilde{V})\lrow
\oplus_i\coker(A^*(R_i)\stackrel{(\eps_i)^*}{\row}A^*(E_i))\hspace{3cm}\phantom{a}
$$
\end{proposition}

\begin{proof}
The Proposition \ref{razd} settles the case where $\pi$ is a single blow up.
Let us use induction on the number of blowings.
Suppose, $\wt{V}$ is the result of $n$ blowings, and $Y$ is the result of $(n-1)$ (first)
of them. Then $\rho:\wt{V}\row Y$ is a single blow up with the center $R$.
Let $F_i,\,i=1,\ldots,n-1$ be the components of the exceptional divisor of $Y$,
and $E_i,\,i=1,\ldots,n-1$ be their strict transforms under $\rho$, and $E$ be the
exceptional divisor of $\rho$. By inductive assumption and Proposition \ref{razd},
we have exact sequences:
\begin{equation*}
\begin{split}
&0\low A_*(V)\stackrel{\pi_*}{\llow}A_*(Y)\llow
\oplus_{i=1}^{n-1}\kker(A_*(F_i)\stackrel{(\eps_i)_*}{\row}A_*(R_i));\\
&0\low A_*(Y)\stackrel{\rho_*}{\llow} A_*(\wt{V})\llow\kker(A_*(E)\stackrel{\eps_*}{\row}A_*(R)).
\end{split}
\end{equation*}
Taking into account that the map:
\begin{equation*}
\begin{split}
&\kker(A_*(F_i)\row A_*(R_i))\twoheadleftarrow\kker(A_*(E_i)\row A_*(R_i))
\end{split}
\end{equation*}
is surjective, we get the first exact sequence. The second one can be proven in a similar fashion.
\Qed
\end{proof}

The following "singular" variant of the above result is an important tool in our calculations, and it permits to present $A_*(Z)$ in terms of $A_*$ of finitely many smooth varieties.

\begin{proposition}
\label{AbmZ}
Let $Z$ be a variety, and $\wt{Z}\stackrel{\pi}{\row} Z$ be the sequence of blowups with
smooth centers $R_i$ and the respective components $E_i$ of the exceptional divisor. Then
we have an exact sequence:
$$
0\low A_*(Z)\llow\left(A_*(\widetilde{Z})\oplus\left(\oplus_iA_*(R_i)\right)\right)\llow
\oplus_iA_*(E_i).
$$
\end{proposition}

\begin{proof}

\begin{lemma}
\label{lVT}
Let $\pi:\wt{V}\row V$ be a projective birational map of smooth varieties, which
is an isomorphism outside the closed subvariety $T\row V$, and such that $W=\pi^{-1}(T)$
is a divisor with strict normal crossings with components $E_i$. Then we have an exact sequence:
$$
0\low A_*(V)\stackrel{\pi_*}{\llow}A_*(\wt{V})\llow
\oplus_i\kker(A_*(E_i)\row A_*(T)).
$$
\end{lemma}

\begin{proof}
Let $\pi':\wt{V}'\row V$ be the permitted blow up with centers over $T$ resolving $T$
to a divisor $W'$ with strict normal crossings (Theorem \ref{Hip}). Let $E'_j$ be the components
of $W'$, and $R'_j$ be the respective smooth centers. Then, by Proposition \ref{vvter},
we have an exact sequence:
$$
0\low A_*(V)\stackrel{\pi'_*}{\llow}A_*(\wt{V}')\llow
\oplus_i\kker(A_*(E'_i)\row A_*(R'_i)).
$$
Since the map $A_*(\wt{V}')/(\oplus_i\kker(A_*(E'_i)\row A_*(R'_i)))\row A_*(V)$
factors through $A_*(\wt{V}')/(\oplus_i\kker(A_*(E'_i)\row A_*(T)))$,
we have the statement for $\wt{V}'$.
Let us denote $B(\wt{V}):=\coker\left(\oplus_i\kker(A_*(E_i)\row A_*(T))\row A_*(\wt{V})\right)$.
We have a natural surjective map $B(\wt{V})\twoheadrightarrow A_*(V)$.
Since $\wt{V}$ and $\wt{V}'$ are isomorphic outside $W$ and $W'$, by the Weak Factorization Theorem
(Theorem \ref{WF}(6)), we have a diagram:
$$
\xymatrix @=3mm{
& Y_1 \ar @{->}[ld] \ar @{->}[rd] && Y_3 \ar @{->}[ld] \ar @{->}[rd] &&
 && Y_{n-2} \ar @{->}[ld] \ar @{->}[rd] && Y_n \ar @{->}[ld] \ar @{->}[rd]  &\\
\wt{V}' && Y_2 && Y_4 &\ldots  & Y_{n-3}  && Y_{n-1}&& \wt{V},
}
$$
where all $Y_i$'s are projective either over $\wt{V}'$, or $\wt{V}$, and all the maps are blowings
up/down w.r.to smooth centers which belong to exceptional divisor, and meet all of it's
components properly.
In particular, each $Y_i$ has a natural map to $V$, which is an isomorphism outside $T$, and the
preimage of $T$ is the exceptional divisor (with strict normal crossings) on $Y_i$. Since the maps
$Y_{2n-1}\lrow Y_{2n}\llow Y_{2n+1}$ are blowings up/down with centers belonging to an exceptional
divisor, we see (using Proposition \ref{vvter}) that the maps $B(Y_{2n-1})\row B(Y_{2n})\low B(Y_{2n+1})$
are isomorphisms. Clearly, these identifications are compatible with the maps $B(Y_i)\row A_*(V)$.
Since the map $B(\wt{V}')\row A_*(V)$ is an isomorphism, so is the map $B(\wt{V})\row A_*(V)$.
\phantom{a}\hspace{5mm}
\Qed
\end{proof}

\begin{lemma}
\label{Abm1}
Let $\pi:\wt{Z}\row Z$ be a projective map of varieties, which is an
isomorphism outside the closed subvariety $R\row Z$ with the preimage
$E=\pi^{-1}(R)$. Then one has an exact sequence:
$$
0\low A_*(Z)\llow\left(A_*(\widetilde{Z})\oplus A_*(R)\right)\llow
A_*(E).
$$
\end{lemma}

\begin{proof}
The fact that it is a complex is evident.
Let us construct the map
$$
\ffi: A_*(Z)\lrow\coker\left(A_*(E)\row A_*(\wt{Z})\oplus A_*(R)\right)
$$
inverse to our projection.
Let $v:V\row Z$ be some projective map with $V$ smooth irreducible.
If the image of $v$ is contained in $R$, then we get a natural map
$A_*(V)\row A_*(R)\row\coker(A_*(E)\row A_*(\wt{Z})\oplus A_*(R))$.
Otherwise, we have a birational map $V\dashrightarrow\wt{Z}$, which
can be resolved by blowing up smooth centers over $v^{-1}(R)$.
Then the exceptional set $W$ of this blowup $\rho:\wt{V}\row V$ is a divisor with strict normal crossings on $\wt{V}$, and
the natural map $\wt{v}:\wt{V}\row\wt{Z}$ maps $W$ to $E$. Moreover, we can assume that $W=(v\circ\rho)^{-1}(R)$.
If $F_j$ are components of $W$, and $S_j$ are the respective
smooth centers, then by Lemma \ref{lVT},
$$
0\low A_*(V)\stackrel{\rho_*}{\llow}A_*(\widetilde{V})\llow
\oplus_j\kker(A_*(F_j)\row A_*(v^{-1}(R))).
$$
Since $\wt{v}$ maps $F_j$ to $E$, and $v$ maps $v^{-1}(R)$ to $R$, the map $(\wt{v})_*:A_*(\wt{V})\row A_*(\wt{Z})$
provides a well-defined map
$\ffi_v:A_*(V)\lrow \coker\left(A_*(E)\row A_*(\wt{Z})\oplus A_*(R)\right)$.

Let $\wt{V}_1,\wt{V}_2$ be two resolutions as above, with the exceptional divisors $W_1$ and $W_2$.
Then $\wt{V}_1\backslash W_1\cong V\backslash v^{-1}(R)\cong\wt{V}_2\backslash W_2$. Hence, by the
Weak Factorization Theorem (Theorem \ref{WF}(6)), there exists a diagram:
$$
\xymatrix @=3mm{
& Y_1 \ar @{->}[ld] \ar @{->}[rd] && Y_3 \ar @{->}[ld] \ar @{->}[rd] &&
 && Y_{n-2} \ar @{->}[ld] \ar @{->}[rd] && Y_n \ar @{->}[ld] \ar @{->}[rd]  &\\
\wt{V}_1 && Y_2 && Y_4 &\ldots  & Y_{n-3}  && Y_{n-1}&& \wt{V}_2,
}
$$
where all $Y_i$'s are projective either over $\wt{V}_1$, or $\wt{V}_2$, and all the maps are blowings
up/down w.r.to smooth centers which belong to exceptional divisor, and meet all of it's
components properly.
In particular, each $Y_i$ has a natural map to $\wt{Z}$, so that the preimage of $E$ is the exceptional
divisor.

Using notations from the proof of Lemma \ref{lVT}, let us
define
$$
B(Y_i):=\coker(\left(\oplus_j\kker(A_*(G_j)\row A_*(v^{-1}(R)))\right)\row A_*(Y_i)),
$$ where
$G_j$ are components of the exceptional divisor of $Y_i$.
Then we have a natural map $B(Y_i)\row \coker(A_*(E)\row A_*(\wt{Z})\oplus A_*(R))$, which is
compatible with the identifications: $B(Y_{2n-1}) = B(Y_{2n}) = B(Y_{2n+1})$ (as in the proof of
Lemma \ref{lVT}).
This shows that the map
$$
\ffi_v:A_*(V)\row\coker\left(A_*(E)\row A_*(\wt{Z})\oplus A_*(R)\right)
$$
does not depend on the choice of the resolution $\wt{V}\row V$.

Let $V_1\stackrel{f}{\lrow}V_2\stackrel{v_2}{\row} Z$ be some projective maps with $V_1$ and $V_2$ smooth,
and $v_1=v_2\circ f$. We can assume $V_1$ and $V_2$ irreducible.
If $image(v_2)\subset R$, then both maps $\ffi_{v_1}$ and $\ffi_{v_2}$
are passing through $A_*(R)$ and are clearly
compatible with $f_*$. So, we can assume that $image(v_2)\not\subset R$.
Let $\wt{V}_2\row V_2$ be the permitted blow up resolving indeterminacy of $\pi^{-1}\circ v_2$, and resolving $v_2^{-1}(R)$ to a divisor
$W_2$ with strict normal crossings.

If $image(v_1)\subset R$, then since the fibers of the projection $\wt{V}_2\row V_2$ are unions of rational varieties, we get
a rational map $V_1\dashrightarrow W_2$. Resolve the indeterminacies of this map:
$V_1\stackrel{\rho}{\llow}\wt{V}_1\stackrel{f'}{\lrow} W_2$, which gives $\wt{f}:\wt{V}_1\row\wt{V}_2$.
Since the map $\rho_*:A_*(\wt{V}_1)\row
A_*(V_1)$ is surjective, and the compatibility of this map with $\ffi_{\wt{v}_1}, \ffi_{v_1}$
is already known (the image is in $R$),
we can substitute $V_1$ by $\wt{V}_1$. Since $W_2$ is mapped to $E$, we get that
$\ffi_{v_2}\circ\wt{f}_*=\ffi_{\wt{v}_1}:A_*(\wt{V}_1)\row
\coker\left(A_*(E)\row A_*(\wt{Z})\oplus A_*(R)\right)$.

Finally, if $image(v_1)\not\subset R$, then we get a rational map $V_1\dashrightarrow\wt{V}_2$ with
indeterminacies only over $v_1^{-1}(R)$ which can be resolved by $\wt{V}_1\row V_1$ making the premiage
$W_1$ of $R$ a divisor with strict normal crossings. We get a map $\wt{f}:\wt{V}_1\row\wt{V_2}$. Then
we can take $\wt{v}_1=\wt{v}_2\circ\wt{f}$, and so
$\ffi_{v_1}=\ffi_{v_2}\circ f_*$.

Since $A_*(Z)=\op{colim}_{v:V\row Z}A_*(V)$, where $v$ runs over all projective maps from smooth quasi-projective varieties,
we obtain a well-defined map:
$$
A_*(Z)\stackrel{\ffi}{\lrow}\coker\left(A_*(E)\row A_*(\wt{Z})\oplus A_*(R)\right)
$$
It is easy to see that it is inverse to the natural projection:
$$
A_*(Z)\stackrel{\psi}{\llow}\coker\left(A_*(E)\row A_*(\wt{Z})\oplus A_*(R)\right)
$$
On the left:
let $v:V\row Z$ be some projective map with $V$ smooth irreducible.
There are two cases: 1) $image(v)\subset R$; 2) $image(v)\not\subset R$.
In both cases, the fact that $\psi\circ\ffi$ is the identity on the image of $v_*:A_*(V)\row A_*(Z)$
is evident from the very definition. \\
On the right: the fact that $\ffi\circ\psi$ is the identity on the
$A_*(R)$-component is evident. As for the $A_*(\wt{Z})$-component, if we have some
projective map $v: V\row\wt{Z}$ then in the definition of $\ffi_{\pi\circ v}$ we can choose $\wt{v}=v\circ\rho$,
where $\rho:\wt{V}\row V$ is the smooth blowup such that $(v\circ\rho)^{-1}(E)$ is a divisor with strict normal crossings.
Then the respective map $\wt{v}:\wt{V}\row\wt{Z}$ factors through $v$, and hence, the composition $\ffi\circ\psi$ is the identity
on the image of $A_*(V)\stackrel{v_*}{\lrow}A_*(\wt{Z})\row\coker\left(A_*(E)\row A_*(\wt{Z})\oplus A_*(R)\right)$.
Thus, we get the identity map on the $A_*(\wt{Z})$-component as well.
Hence, our complex is exact.
\Qed
\end{proof}

\begin{remark}
Of course, if $A^*$ can be extended to a "large" theory, the above Lemma follows automatically
from the respective localization (excision) axiom. The point is that it is true for any theory
in the sense of Definition \ref{goct}.
\end{remark}

Lemma \ref{Abm1} settles the case where $\pi$ is a single blow up.
The rest is done by the induction on the number of blowings in the same way as the proof of
Proposition \ref{vvter}.
\Qed
\end{proof}

\begin{remark}
1) In particular, this applies when $\wt{Z}\row Z$ is the resolution of $Z$ as in Theorems \ref{Hi},\ref{Hif},
that is the permitted blow up with smooth centers which meet the components of the exceptional divisor
properly, and resolves the singularities of $Z$, and then makes the special divisor the one with the strict normal crossings. In this case, all the varieties aside from $Z$ participating in the formula are smooth, and
we get the "finite" presentation of $A_*(Z)$ in terms of smooth varieties.

2) The map $A_*(\wt{Z})\row A_*(Z)$ is not surjective, in general, if $Z$ is not smooth.
Take, for example, $Z$ - the cone over an anisotropic conic, and $R$ - it's vertex.
Then $\wt{Z}$ has no zero cycles of odd degree, while $Z$ has a rational point.
\end{remark}

We will also need the following Bertini-type result.

\begin{proposition}
\label{Bd}
Let $X$ be smooth quasi-projective variety, and $Z\subset X$ be a proper closed subvariety of it.
Then there exists a divisor $Y$ of $X$ which contains $Z$, and is smooth outside $Z$,
as well as in the generic points of the components of $Z$.
\end{proposition}

\subsection{Multiple points excess intersection formula}
\label{subsMPEIF}

In this subsection, $A^*$ is any theory in the sense of Definition \ref{goct}.
Our main aim here is Proposition \ref{MPEIF}.
This analogue of the usual Excess Intersection Formula, where regular
embeddings (of smooth varieties) are substituted by strict normal crossings divisors,
is a very useful computational tool. To state it, one needs to define
the pull-back maps for such divisors.
In the case of Algebraic Cobordism, or any theory obtained from it by
change of coefficients, this is just a (small) piece of the theory of refined
pull-backs developed by M.Levine-F.Morel (following W.Fulton \cite{Fu}).
But this piece is much more explicit than the general one and
is sufficient for almost all applications we need.
The exception is Subsection \ref{c}, where the refined pull-backs of more general
type appear, and where we have to restrict to {\it theories of rational type} (= {\it free theories}
of M.Levine-F.Morel) as a result.
The formula is valid for arbitrary theory in the sense of Definition \ref{goct}, but since
our main statements are valid only for
{\it theories of rational type}, we formulate it only for constant theories
and refer to the case of Algebraic Cobordism done by M.Levine and F.Morel
(see \cite[Theorem 6.6.6(2)(a)]{LM}).

We recall:

\begin{definition} {\rm (\cite[Definition 3.1.4]{LM})}
\label{ncd}
Let $X$ be a smooth variety, and $D=\sum_{i=1}^r l_iD_i$ be an effective Weil divisor on $X$. We call $D$ a divisor with
strict normal crossing, if for any $J\subset\{1,\ldots,r\}$, the intersection scheme $\cap_{i\in J}D_i$ is
a smooth subvariety of $X$ of codimension $=|J|$.
\end{definition}

Denote as $|D|\stackrel{d}{\row}X$ the support $(\cup_{i=1}^r D_i)_{red}$. By $A_*(D)$ we will always mean $A_*(|D|)$. In particular, it
does not depend on the multiplicity of the components as long as one is positive.
Recall, that we have an exact sequence:
$$
0\low A_*(D)\low\oplus_i A_*(D_i)\low\oplus_{i\neq j}A_*(D_i\cap D_j).
$$
Thus, an element of $A_*(D)$ can be thought of as a collection of elements of $A_*(D_i)$
modulo some identifications.

The strict normal crossings divisor has a {\it divisor class} $[D]\in A^0(D)$ such that
$d_*([D])=c^A_1(O(D))\in A^1(X)$.
Having $\lambda_i=c^A_1(O(D_i))$, the idea is to write
\begin{equation}
\label{div-coeff}
[l_1]\cdot_{F_A}\lambda_1+_{F_A}[l_2]\cdot_{F_A}\lambda_2+_{F_A}+\ldots+_{F_A}[l_r]\cdot_{F_A}\lambda_r=
\sum_{\emptyset\neq I\subset\{1,\ldots,r\}}(\prod_{i\in I}\lambda_i)\cdot F^{l_1,\ldots,l_r}_I(\lambda_1,\ldots,\lambda_r),
\end{equation}
where $F^{l_1,\ldots,l_r}_I$ are some power series in $r$ variables with $A$-coefficients, and then
define:
\begin{definition} {\rm (\cite[Definition 3.1.5]{LM})}
\label{divclass}
$$
[D]:=\sum_{\emptyset\neq I\subset\{1,\ldots,r\}}(\hat{d}_I)_*(1)\cdot F^{l_1,\ldots,l_r}_I(\lambda_1,\ldots,\lambda_r),
$$
where $\hat{d}_I:D_I=\cap_{i\in I}D_i\row |D|$ is the closed embedding.
\end{definition}
The result does not depend on how you subdivide
the above formal sum into pieces, but there is some standard way. The convention is
(see \cite[Subsection 3.1]{LM}) to define $F^{l_1,\ldots,l_r}_I$ as the sum of those monomials which are
made exactly of $\lambda_i,\,i\in I$ divided by the $(\prod_{i\in I}\lambda_i)$.
Denoting $\lambda_I=\sum_{i\in I}^{F_A}[l_i]\cdot_{F_A}\lambda_i$ and noticing that $\lambda_I$ is exactly the sum of those monomials
in $\lambda_{\{1,\ldots,r\}}$ which are made of some subset of $\lambda_i,\,i\in I$, we obtain that this standard choice of coefficients
is given by the formula:
\begin{equation}
\label{F-coeff}
F^{l_1,\ldots,l_r}_J=\frac{\sum_{I\subset J}(-1)^{|J|-|I|}\lambda_I}{\prod_{j\in J}\lambda_j},
\end{equation}
where we treat $\lambda_i$'s as formal variables.
 For our purposes, though, it will be convenient to be flexible in choosing $F^{l_1,\ldots,l_r}_I$, so below it will be any collection of
power series satisfying the above equation. In applications we will be often using the choice where $F^{l_1,\ldots,l_r}_I=0$ for $|I|>1$.
Such a choice is, clearly, possible, since every term of the expression $\lambda_{\{1,\ldots,r\}}$ is divisible by some $\lambda_i$.

For a divisor $D$ with strict normal crossings as above we denote as $\komp{D}$ the disjoint union
$\coprod_{\emptyset\neq I\subset \{1,\ldots,r\}}D_I$ of all it's faces, with the natural map $\hat{d}:\komp{D}\row D$.

\begin{definition}
\label{starpullback}
Having a divisor $D=\sum_{i=1}^r l_iD_i$ with strict normal crossings on $X$, we can define
the pull-back:
$$
d^{\star}: A_*(X)\row A_{*-1}(D)
$$
as $\hat{d}_* d^{\sstar}(x)$ where
$$
d^{\sstar}(x)=\sum_{\emptyset\neq I\subset\{1,\ldots,r\}}
d_I^*(x)\cdot F^{l_1,\ldots,l_r}_I(\lambda_1,\ldots,\lambda_r),
$$
and $d_I:D_I\row X$ is the regular embedding of the $I$-th face of $D$.
\end{definition}
Notice, that such a pull-back clearly depends on the multiplicity of the components
(in our notations it is manifested only by the target).
Also, since for $I\subset J$, for $d_{J/I}:D_J\row D_I$, we have:
$(d_{J/I})_*(1)=\prod_{i\in J\backslash I}\lambda_i$, the projection formula shows that it does not matter, how
one chooses the $F^{l_1,\ldots,l_r}_I$ (in particular, one can choose these to be zero for $|I|>1$).

Immediately from the definition, we obtain:

\begin{lemma}
\label{ddstar}
The composition $d_*\circ d^{\,\star}:A^*(X)\row A^{*+1}(X)$ is the multiplication by $c^A_1(O(D))$.
\end{lemma}
\Qed

Let $w:W\row X\times\pp^1$ be a projective map, with $W$ smooth, such that
$W_0=w^{-1}(X\times 0)\stackrel{i_0}{\hookrightarrow}W$ and
$W_1=w^{-1}(X\times 1)\stackrel{i_0}{\hookrightarrow}W$ are divisors with strict normal crossings.
Let $W_0\stackrel{w_0}{\row}X$, $W_1\stackrel{w_1}{\row}X$ be natural maps.
As a corollary of Lemma \ref{ddstar} we get:

\begin{proposition}
\label{MPcor}
In the above situation, $(i_0)_*\circ i_0^{\star}=(i_1)_*\circ i_1^{\star}$ in $A_*(W)$.
In particular, $(w_0)_*\circ i_0^{\star}=(w_1)_*\circ i_1^{\star}$
\end{proposition}

\begin{proof}
Observe that $O_W(W_0)\cong\pi^*(O_{\pp^1}(1))\cong O_W(W_1)$.
\Qed
\end{proof}

Let
\begin{equation}
\label{divsquare}
\xymatrix @-0.7pc{
E \ar @{->}[r]^(0.5){e} \ar @{->}[d]_(0.5){\ov{f}}&
Y \ar @{->}[d]^(0.5){f}\\
D \ar @{->}[r]_(0.5){d} & X.
}
\end{equation}
be a Cartesian square, where $X$ and $Y$ are smooth and $D\stackrel{d}{\lrow}X$ and
$E\stackrel{e}{\lrow}Y$ are divisors with strict normal crossings (closed codimension 1 subschemes given by principal ideals whose $div$ is a strict normal crossings divisor).
Then we can define the {\it combinatorial pull-back}:
$$
\ov{f}^{\,\star}:A^*(D)\row A^*(E)
$$
as follows.
Suppose, $D=\sum_{i=1}^rl_iD_i$, $E=\sum_{j=1}^sm_jE_j$, where $D_i$ and $E_j$ are irreducible components;
$\lambda_i=c^A_1(O(D_i))$, $\mu_j=c^A_1(O(E_j))$, and $f^*(D_i)=\sum_jp_{i,j}E_j$. If $P,L,M$ are
matrices $(p_{i,j})$, $(l_i)$, $(m_j)$, then we have: $L\cdot P=M$.
Notice, that if $p_{i,j}\neq 0$, for some $i$ and $j$, then we have the natural map $f_{j,i}:E_j\row D_i$,
and so the map $f_{J,i}:E_J\row D_i$, for any $J\ni j$.
Assume that $F^{p_{i,1},\ldots,p_{i,s}}_J=0$, if $p_{i,j}=0$, for some $j\in J$
(notice, that there are no monomials divisible by $\prod_{j\in J}\mu_j$ in the $\sum_j^{F_A}[p_{i,j}]\cdot_{F_A}\mu_j$,
so any "reasonable" choice will do).

\begin{definition}
\label{fstar}
Let $x=\sum_i(\hat{d}_i)_*(x_i)$, for some $x_i\in A^*(D_i)$. Define:
$$
\ov{f}^{\,\star}(x):=\sum_{i=1}^r\sum_{\emptyset\neq J\subset\{1,\ldots,s\}}
(\hat{e}_J)_*f_{J,i}^*(x_i)\cdot F^{p_{i,1},\ldots,p_{i,s}}_J(\mu_1,\ldots,\mu_s)\in A^*(E),
$$
where we ignore the terms with the zero $F^{p_{i,1},\ldots,p_{i,s}}_J$.
\end{definition}
Again, , since for $I\subset J$, for $e_{J/I}:E_J\row E_I$, we have:
$(e_{J/I})_*(1)=\prod_{j\in J\backslash I}\mu_j$, the projection formula shows that it does not matter, how
we choose the $F^{p_{i,1},\ldots,p_{i,s}}_J$. Also, it is clear that we get a well-defined map in the case
$D$ - smooth irreducible (of multiplicity 1).
One can show that this map is well-defined in general, but we will
spare the reader from that.

Our {\it combinatorial pull-backs} are functorial.
Suppose,
$$
\xymatrix @-0.7pc{
D  \ar @{->}[d]_(0.5){d}&
E \ar @{->}[l]_(0.5){\ov{u}} \ar @{->}[d]^(0.5){e} &  F \ar @{->}[l]_(0.5){\ov{v}} \ar @{->}[d]^(0.5){f} \\
X  & Y \ar @{->}[l]^(0.5){u} & Z \ar @{->}[l]^(0.5){v}.
}
$$
be the cartesian diagram, where $X,Y,Z$ are smooth, and $D$, $E$ and $F$ are divisors
with strict normal crossing.

\begin{proposition}
\label{starfunctoriality}
Let $A^*$ be a theory in the sense of Definition \ref{goct}. Then, in the above situation,
$$
(\ov{u}\circ \ov{v})^{\,\star}=\ov{v}^{\,\star}\circ\ov{u}^{\,\star}.
$$
\end{proposition}

\begin{proof}
Let $D=\sum_{i=1}^rl_i\cdot D_i$, $E=\sum_{j=1}^sm_j\cdot E_j$, and $F=\sum_{k=1}^tn_k\cdot F_k$.
Let $u^*(D_i)=\sum_{j=1}^sp_{i,j}\cdot E_j$, and $v^*(E_j)=\sum_{k=1}^t q_{j,k}\cdot F_k$.
If $L,M,N,P,Q$ are the respective matrices, then $L\cdot P=M$ and $M\cdot Q=N$. The matrix of
$(u\circ v)^*$ is then given by $R=P\cdot Q$.
Let $\lambda_i=c^A_1(O_X(D_i))$, $\mu_j=c^A_1(O_Y(E_j))$, $\nu_k=c^A_1(O_Z(F_k))$.
Now, we can assume that $F^{p_{i,1},\ldots,p_{i,s}}_J=0$,
for $|J|>1$. Then, for $x=\sum_i(\hat{d}_i)_*(x_i)$, we have:
\begin{equation*}
\begin{split}
&\ov{v}^{\,\star}\ov{u}^{\,\star}(x)=\ov{v}^{\,\star}\sum_{i=1}^r\sum_{j=1}^s(\hat{e}_j)_*u_{j,i}^*(x_i)\cdot F^{p_{i,1},\ldots,p_{i,s}}_j(\mu_1,\ldots,\mu_s)=\\
&\sum_{i=1}^r\sum_{j=1}^s\sum_{K\subset\{1,\ldots,t\}}v_{K,j}^*u_{j,i}^*(x_i)\cdot
v_{K,j}^*(F^{p_{i,1},\ldots,p_{i,s}}_j(\mu_1,\ldots,\mu_s))\cdot F^{q_{j,1},\ldots,q_{j,t}}_K
(\nu_1,\ldots,\nu_t)=\\
&\sum_{i=1}^r\sum_{K\subset\{1,\ldots,t\}}(v\circ u)_{K,i}^*(x_i)\cdot
F^{r_{i,1},\ldots,r_{i,t}}_K(\nu_1,\ldots,\nu_t)=(\ov{u}\circ \ov{v})^{\,\star}(x).
\end{split}
\end{equation*}
\Qed
\end{proof}

Finally, in the case of a {\it free theory} in the sense of Levine-Morel - see \cite[Remark 2.4.14]{LM}
(by Proposition \ref{KlassRT}, these theories are exactly our {\it theories of rational type}), $e^{\star}$ appears to be the same as a {\it refined pull-back} morphism.

\begin{lemma}
\label{star!-e-only}
Let $A^*=\Omega^*\otimes_{\laz}A$ be a theory obtained from Algebraic Cobordism by change of
coefficients. Then for any square {\rm (\ref{divsquare})}, we have:
$$
e^{\star}=d^{\,!}.
$$
\end{lemma}

\begin{proof}
This identity follows from Lemma 6.6.2, Lemma 6.5.6, Definition 6.5.1, and definitions of Subsection 6.2.1 of \cite{LM}.
\Qed
\end{proof}

\begin{proposition} {\rm (Multiple points excess intersection formula)}\\
\label{MPEIF}
Let $A^*$ be a theory satisfying $(CONST)$. Then,
in the above situation, we have:
\begin{itemize}
\item[$(1)$ ]
$$
e_*\circ \ov{f}^{\,\star} = f^*\circ d_*.
$$
\item[$(2)$ ] Suppose, $f$ is projective. Then
$$
\ov{f}_*\circ e^{\star}= d^{\,\star}\circ f_*.
$$
\end{itemize}
\end{proposition}

\begin{proof}
Part (2): If $A^*=\Omega^*\otimes_{\laz}A$ is obtained from Algebraic Cobordism by change
of coefficients, by Lemma \ref{star!-e-only}, this is a particular case of \cite[Theorem 6.6.6(2)(a)]{LM}.
The general case follows from Proposition \ref{A0}.

Part (1):

\begin{lemma}
\label{star!}
Let $A^*$ be a free theory in the sense of Levine-Morel. Then for any square {\rm (\ref{divsquare})}, we have:
$$
\ov{f}^{\,\star}=f^{\,!}.
$$
\end{lemma}

\begin{proof}
This needs to be checked only for the case where $D$ is a smooth divisor and $f$
is a regular embedding, where, in the case of codimension $1$, it follows from Lemma \ref{star!-e-only}.
By blowing up $Y$ inside $X$ and resolving the preimage of $D$, and using functoriality of $\ov{f}^{\,\star}$ and $f^{\,!}$
(Proposition \ref{starfunctoriality} and \cite[Theorem 6.6.6(3)]{LM}), we may reduce the general case to the case of a regular embedding
of codimension $1$ and to the case of a projective bi-rational $f$. Any such projective bi-rational map is dominated by a sequence
of smooth blowups permitted with respect to the preimage of $D$. Now it follows from the (already proven) part (2) of the Proposition
that, for any $f$, the composition $\ov{f}_*\circ\ov{f}^{\,\star}$ coincides with the multiplication by $f_*(1)$.
By the item 2 of the Subsection 6.6.7 of \cite{LM}, the same is true about the composition $\ov{f}_*\circ f^{\,!}$.
Since, for a projective birational map $g$, the element
$g_*(1)$ is invertible (Proposition \ref{pi1invert}), in view of the above domination,
we may reduce our problem to the case of a single blowup $f:\wt{X}\row X$ at a smooth center
$R$ which is in good position with respect to a smooth divisor $D$. If $R$ is not contained in $D$, then $E=\wt{D}$ is smooth and
both maps $\ov{f}^{\,\star}$ and $f^{\,!}$ are clearly equal to $\ov{f}^*$.
If $R\subset D$, then $E=\wt{D}\cup F$, where $\wt{D}$ is the strict
transform of $D$ and $F$ is the exceptional divisor of the blowup. We use the fact that the map
$(\ov{f}_*,j^{\,!}):A_*(\wt{D}\cup F)\row A_*(D)\oplus A_{*-1}(F)$ is injective, where $j:F\row\wt{X}$ is the natural embedding.
This, in turn, follows from Proposition \ref{razd}(1)
and the injectivity of $(\eps_*,\cdot c^A_1(O(-1))):A_*(P)\row A_*(R)\oplus A_{*-1}(P)$,
for any projective bundle $P=\pp_R(V)\stackrel{\eps}{\row} R$. The latter one may be easily seen, for example, from Quillen's formula.
It remains to check that $\ov{f}_*\circ\ov{f}^{\,\star}=\ov{f}_*\circ f^{\,!}$ and $j^{\,!}\circ\ov{f}^{\,\star}=j^{\,!}\circ f^{\,!}$.
The first equality was already established for arbitrary $f$, while the second one may be checked by a direct calculation
(using functoriality
of the refined pull-backs and the Excess Intersection Formula). This shows that $\ov{f}^{\,\star}=f^{\,!}$.
\Qed
\end{proof}

Now, we can prove (1).
Again, if $A^*$ is a free theory in the sense of Levine-Morel, then, by Lemma \ref{star!},
this is a particular case of \cite[Theorem 6.6.6(2)(a)]{LM}.
The general case follows from Proposition \ref{A0}.
\phantom{a}\hspace{5mm}
\Qed
\end{proof}

\begin{remark}
\label{nonCONST}
Proposition \ref{MPEIF} is valid for any theory in the sense of Definition \ref{goct}.
\end{remark}

Finally, we will need the following fact about divisors with strict normal crossings.

\begin{proposition}
\label{DEFi}
Let $D,F_i,i=1,\ldots,k$ be smooth irreducible divisors on some smooth variety $A$.
Suppose that:
\begin{itemize}
\item[$(1)$] $\cup_{i=1}^kF_i$ is a d.w.s.n.c. on $A$;
\item[$(2)$] $D\cap(\cup_{i=1}^kF_i)$ is a d.w.s.n.c. on $D$.
\end{itemize}
Then $(3)$ $D\cap(\cap_{i=1}^kF_i)$ is a d.w.s.n.c. on $\cap_{i=1}^kF_i$.
\end{proposition}

\begin{proof}
The condition (3) can be checked locally. Let $e\in D\cap(\cap_{i=1}^kF_i)$ be some point.
Consider the following cases:

(i) $T_e(D)\neq T_e(F_i)$, for all $i$. Then (2) $\Leftrightarrow$
$T_e(F_i),i=1,\ldots,k$ are linearly independent modulo $T_e(D)$ $\Leftrightarrow$
$T_e(D),T_e(F_i),i=1,\ldots,k$ are linearly independent $\Rightarrow$ $D\cap(\cap_{i=1}^kF_i)$ is smooth divisor on
$\cap_{i=1}^kF_i$ at $e$.

(ii) $T_e(D)=T_e(F_i)$ for some (unique) $i$.
Then $T_e(\cap_{j=1}^kF_j)$ may be identified with $T_e(D)\cap(\cap_{j\neq i}T_e(F_j))$.
But we know that the faces of $D\cap F_i$ are transversal to $D\cap(\cap_{j\neq i}F_j)$ on $D$.
This implies that $D\cap(\cap_{j=1}^kF_j)$ is a d.w.s.n.c. on $D\cap(\cap_{j\neq i}F_j)$ at $e$.
Hence, $D\cap(\cap_{j=1}^kF_j)$ is a d.w.s.n.c. on $\cap_{j=1}^kF_j$ at $e$.
\Qed
\end{proof}

\section{Resolution of singularities}
\label{RoS}

In this section we list the results related to Resolution of Singularities
and the Weak Factorization Theorem which are widely used throughout the text.

\begin{definition}
\label{perm-bu}
Let $X$ be a smooth variety and $D$ - a divisor with strict normal crossings on it.
By a {\it permitted blow-up w.r.to} $D$ we will understand such a sequence of
blow-ups with smooth centers $R_i\subset X_i$:
$$
\wt{X}=X_n\stackrel{\pi_n}{\row}X_{n-1}\stackrel{\pi_{n-1}}{\row}\ldots\stackrel{\pi_2}{\row}
X_1\stackrel{\pi_1}{\row}X
$$
such that, for the exceptional divisor $E_i$ of $\pi^i=\pi_1\circ\ldots\circ\pi_i: X_i\row X$,
and the total transform $(\pi^i)^*(D)$, the divisor $E_i+(\pi^i)^*(D)$ has
strict normal crossings, and $R_i$ has normal crossings with it.

If $D$ is empty, we will call it just a {\it permitted blow-up}.
\end{definition}

\begin{theorem} {\rm (Hironaka, \cite{Hi})}
\label{Hi}
Let $Z$ be a subvariety of a smooth variety $X$. Then there exists a permitted
blow-up $\pi:\wt{X}\row X$ such that:
\begin{itemize}
\item[$(1)$ ] All the centers $R_i$ are lying over the singular locus of $Z$.
\item[$(2)$ ] The strict transform $\wt{Z}\subset\wt{X}$ of $Z$ is smooth and has normal
crossings with $E_n$.
\end{itemize}
\end{theorem}

\begin{theorem} {\rm (Hironaka, \cite{Hi})}
\label{Hif}
Let $f:X\dashrightarrow Y$ be a rational map of reduced varieties.
Then there is a permitted blow-up $\pi:\wt{X}\row X$ such that:
\begin{itemize}
\item[$(1)$ ] All the centers $R_i$ are lying over the locus of $X$ where it is not smooth, or
$f$ is not a morphism.
\item[$(2)$ ] The rational map $f\circ\pi:\wt{X}\row Y$ is a morphism.
\end{itemize}
\end{theorem}

\begin{theorem} {\rm (Hironaka, \cite{Hi}, see also \cite[1.2.3]{AKMW} and \cite{BM})}
\label{Hip}
Let $\ci$ be a sheaf of ideals on a smooth variety $X$, and $U\subset X$ be an
open subvariety such that $\ci|_U$ is an ideal sheaf of a divisor with strict normal
crossings. Then
there is a permitted blow-up $\pi:\wt{X}\row X$ with centers outside $U$ such that the total
transform $\pi^*(\ci)$ is an ideal of a strict normal crossing
divisor $\wt{E}$.
\end{theorem}

There is also a relative to divisor $D$ version (see \cite[1.2.2]{AKMW} and \cite{BM}).

\begin{proposition}
\label{HipD}
Let $X$ be smooth quasi-projective variety, $Z\subset X$ - a closed
subvariety, and $D$ - a divisor with strict normal crossings on $X$. Then
there exists a permitted w.r.to $D$ blow up $\wt{X}\stackrel{\pi}{\lrow}X$ with centers over $Z$
such that $\pi^{-1}(Z)\cup\pi^{-1}(D)$ is a divisor with strict normal crossings.
\end{proposition}

The following result is the Weak Factorization Theorem - \cite[Theorem 0.3.1]{AKMW}, see also \cite{Wwf}.

\begin{theorem} {\rm (Abramovich-Karu-Matsuki-Wlodarczyk)}
\label{WF}
Let $\theta:X_1\dashrightarrow X_2$ be birational map of smooth proper varieties over $k$,
which is an isomorphism on the open set $U\subset X_1$. Then $\theta$ can be factored into
a sequence of blowings up and blowings down with nonsingular centers disjoint from $U$.
Namely, to any such $\theta$ we can associate a diagram:
$$
\xymatrix @-0.5pc{
X_1=Y_0 \ar @{-->}[r]^(0.7){\ffi_1}&
Y_1 \ar @{-->}[r]^(0.5){\ffi_2}&
\ldots \ar @{-->}[r]^(0.5){\ffi_{i-1}}&
Y_{i-1} \ar @{-->}[r]^(0.5){\ffi_i}&
Y_i \ar @{-->}[r]^(0.5){\ffi_{i+1}}&
\ldots \ar @{-->}[r]^(0.5){\ffi_{l-1}}&
Y_{l-1} \ar @{-->}[r]^(0.3){\ffi_l}&
Y_l=X_2
}
$$
where
\begin{itemize}
\item[$(1)$ ] $\theta=\ffi_l\circ\ffi_{l-1}\circ\ldots\circ\ffi_2\circ\ffi_1$,
\item[$(2)$ ] $\ffi_i$ are isomorphisms on $U$, and
\item[$(3)$ ] either $\ffi_i$, or $\ffi_i^{-1}$ is a blow up morphism with smooth
center disjoint from $U$.
\item[$(4)$ ] Functoriality: if $g:\theta\row\theta'$ is an absolute isomorphism carrying
$U$ to $U'$, and $\ffi_i':Y_{i-1}'\dashrightarrow Y_i$ is the factorization of $\ffi'$, then
the resulting rational maps $g_i:Y_i\dashrightarrow Y_i'$ is an absolute isomorphism.
\item[$(5)$ ] There is an index $i_0$ such that, for $i\leq i_0$, the map $Y_i\dashrightarrow X_1$
is projective map, while for $i\geq i_0$, $Y_i\dashrightarrow X_2$ is projective map.
\item[$(6)$ ] Let $E_i\subset Y_i$ be the exceptional divisor of $Y_i\row X_1$ (respectively, of
$Y_i\row X_2$) in case $i\leq i_0$ (respectively, $i\geq i_0$). Then the above centers of blow up have normal crossing with $E_i$. If, moreover, $X_1\backslash U$ (respectively, $X_2\backslash U$)
is a normal crossing divisor, then the centers of blow up have normal crossing with the inverse
images of this divisor.
\end{itemize}

\end{theorem}

\end{document}